\documentclass[10pt]{amsart}
\usepackage{amsmath,amssymb,latexsym,cancel,rotating}
\usepackage{graphicx,amssymb,mathrsfs,amsmath,color,fancyhdr,amsthm}
\usepackage[all]{xy}
\usepackage{setspace}

\newtheorem{theorem}{Theorem}[section]
\newtheorem{lemma}[theorem]{Lemma}
\newtheorem{corollary}[theorem]{Corollary}
\newtheorem{definition}[theorem]{Definition}
\newtheorem{proposition}[theorem]{Proposition}

\theoremstyle{definition}
\newtheorem{remark}[theorem]{Remark}

\newtheorem{example}[theorem]{Example}

\begin{document}
\begin{spacing}{1.2}
\title[Lusztig sheaves and tensor products]
{Lusztig sheaves and tensor products of integrable highest weight modules}
\author[Fang,Lan]{Jiepeng Fang,Yixin Lan}

\address{Department of Mathematics and New Cornerstone Science Laboratory,The University of Hong Kong, Pokfulam, Hong Kong, Hong Kong SAR, P. R. China}
\email{fangjp@hku.hk (J.Fang)}

\address{Academy of Mathematics and Systems Science, Chinese Academy of Sciences, Beijing 100190, P.R.China}
\email{lanyixin@amss.ac.cn (Y.Lan)}


\date{\today}

\bibliographystyle{abbrv}

\maketitle
\begin{abstract}
	By introducing  $N$-framed quivers, we  define the  localization of Lusztig's sheaves for $N$-framed quivers and functors $E^{(n)}_{i},F^{(n)}_{i},K^{\pm}_{i}$ for localizations. This gives a categorical realization of tensor products of integrable highest weight modules of the quantized enveloping algebra. The simple perverse sheaves in the localization provide the canonical basis of  tensor products. Moreover, we give a categorical interpretation of the $\mathcal{R}$-matrix and Yang-Baxter equations.
\end{abstract}

\setcounter{tocdepth}{1}\tableofcontents

\section{Introduction}
\subsection{}  A symmetric Cartan datum $(I,(-,-))$ is given by a finite set $I$ and a symmetric bilinear form $(-,-):\mathbb{Z}[I]\times \mathbb{Z}[I] \rightarrow \mathbb{Z}$. The associated generalized Cartan matrix  $(a_{i,j})_{i,j \in I}$ satisfies $a_{i,i}=2$ and $a_{i,j}=(i,j)$ for $i \neq j$. Let  $\mathfrak{g}$ be the Kac-Moody Lie algebra associated to the generalized Cartan matrix $(a_{i,j})_{i,j \in I}$, let  $\mathbf{U}=\mathbf{U}_{v}(\mathfrak{g})$ be the quantized enveloping algebra. The quantized enveloping algebra admits triangular decomposition.

Given a symmetric Cartan datum, we can associate a finite graph $\mathbf{\Gamma}$ without loops such that the set of vertices is $I$ and there are exactly $|a_{i,j}|$ edges connecting  $i$ and $j$ for $i \neq j$. Choose an orientation $\Omega$, we get a quiver $Q=(I,H,\Omega)$. Lusztig considered the moduli space $\mathbf{E}_{\mathbf{V},\Omega}$ of representations of the quiver $Q$ and categorified the integral form $_{\mathcal{A}}\mathbf{U}^{+}={_{\mathcal{A}}\mathbf{U}}^{+}_{v}(\mathfrak{g})$ of positive (or negative) part  of the quantized enveloping algebra via perverse sheaves on $\mathbf{E}_{\mathbf{V},\Omega}$ in \cite{MR1088333}, where $\mathcal{A}=\mathbb{Z}[v,v^{-1}]$.

It is interesting to give a geometric realization of representations of $\mathbf{U}$. Basing on Lusztig's theory and Zheng's work \cite{zheng2007geometric}, for a given  dominating weight $\Lambda$, one can define a certain Lusztig's semisimple category $\mathcal{Q}_{\mathbf{V},\mathbf{W}}$ on framed quivers and defined its localization $\mathcal{Q}_{\mathbf{V},\mathbf{W}}/\mathcal{N}_{\mathbf{V}}$. After we define functors $E^{(n)}_{i},F^{(n)}_{i}$ and $K_{i}^{\pm}$ for $i \in I,n \in \mathbb{N}$,  the Grothendieck group $\mathcal{K}_{0}(\Lambda)$ of $\coprod\limits_{\mathbf{V} }\mathcal{Q}_{\mathbf{V},\mathbf{W}}/\mathcal{N}_{\mathbf{V}}$ becomes a $_{\mathcal{A}}\mathbf{U}$-module under the action of $E^{(n)}_{i},F^{(n)}_{i}$ and $K_{i}^{\pm}$, which is canonically isomorphic to the irreducible integrable highest weight module $_{\mathcal{A}}L(\Lambda)$.  Moreover, simple objects in localization $\mathcal{Q}_{\mathbf{V},\mathbf{W}}/\mathcal{N}_{\mathbf{V}}$ form the canonical basis of $_{\mathcal{A}}L(\Lambda)_{|\mathbf{V}|}$. See \cite{fang2023lusztig}. 

A natural problem is to generalize this categorical construction to tensor products. In the framework of Nakajima's quiver varieties, \cite{maulik2019quantum}, \cite{MR1865400} and \cite{MR3077693} provide a geometric (but not categorical) construction via equivariant $K$-theory.  A. Sartori and C. Stroppel in \cite{MR3314294}  give a categorical construction of tensor products of finite dimensional representations of $\mathfrak{sl}_{n}$ via BGG category $\mathcal{O}$. H.Zheng categorically realizes tensor products of integrable modules of $\mathbf{U}_{v}(\mathfrak{sl}_{2})$ in \cite{zheng2007geometric} and generalizes it to the other cases in \cite{MR3200442}. B. Webster also uses  diagrammatic approaches to categorify tensor products in \cite{webster2015canonical}.

In \cite{lusztig1992canonical}, Lusztig uses the quasi $\mathcal{R}$-matrix to define a linear map $\Psi$ of certain tensor products and provides a construction of canonical bases of tensor products by algebraic methods. Later, H.Bao and W.Wang generalize Lusztig's construction to general cases in \cite{bao2016canonical}.  Another natural problem is to provide a geometric construction of canonical bases by Lusztig's perverse sheaf theory. Indeed,  by studying the singular supports of Lusztig's sheaves, Y. Li proved in \cite{MR3177922} that Zheng's constructon provides a realization of the (signed) canonical basis of the tensor product.

\subsection{} The purpose of this paper is to give a geometric realization of tensor products of the integrable highest weight modules and their canonical bases via Lusztig's semisimple perverse sheaves on quivers' moduli spaces. We introduce the $N$-framed quiver $Q^{(N)}$ of a given quiver $Q$ and consider the Lusztig's theory of perverse sheaves for $Q^{(N)}$. For given a sequence of dominant  weights $\Lambda^{\bullet}=(\Lambda_{1},\Lambda_{2},\cdots \Lambda_{N})$, we choose a collection of graded spaces $\mathbf{W}^{\bullet}=(\mathbf{W}^{1},\cdots, \mathbf{W}^{N}) $. We consider a certain full subcategory $\mathcal{Q}_{\mathbf{V},\mathbf{W}^{\bullet}}$ of semisimple perverse sheaves on $Q^{(N)}$. We also consider  certain localizations $\mathcal{Q}_{\mathbf{V},\mathbf{W}^{\bullet}}/ \mathcal{N}_{\mathbf{V},i}$ and $\mathcal{Q}_{\mathbf{V},\mathbf{W}^{\bullet}}/ \mathcal{N}_{\mathbf{V}}$ of $\mathcal{Q}_{\mathbf{V},\mathbf{W}^{\bullet}}$ and define functors $E^{(n)}_{i}$ and $F^{(n)}_{i}$ of the localizations, which satisfy the following relations
\begin{equation*}
	E_{i}^{(n)}F_{j}\cong F_{j}E_{i}^{(n)}, i\neq j,
\end{equation*}
\begin{equation*}
	E^{(n)}_{i}F_{i} \oplus \bigoplus\limits_{0\leqslant m \leqslant M-1} E^{(n-1)}_{i} [M-1-2m] \cong F_{i}E^{(n)}_{i} \oplus \bigoplus\limits_{0\leqslant m \leqslant -M-1} E^{(n-1)}_{i} [-2m-M-1].
\end{equation*}
Hence with the functors $E_{i}$ and $F_{i}$, the Grothendieck group $\mathcal{K}_{0}(\Lambda^{\bullet})$ of $\coprod \limits_{\mathbf{V}}\mathcal{Q}_{\mathbf{V},\mathbf{W}^{\bullet}}/ \mathcal{N}_{\mathbf{V}}$ becomes an integrable $_{\mathcal{A}}\mathbf{U}$-module. 

The main reason for multi-framing is that one can  use Lusztig's restriction functors to construct the following isomorphisms of $_{\mathcal{A}}\mathbf{U}$-modules:
\begin{equation*}
	\Delta_{N}:\mathcal{K}_{0}(\Lambda^{\bullet}) \rightarrow \mathcal{K}_{0}(\Lambda^{\bullet-1},0) \otimes \mathcal{K}_{0}(0,\Lambda_{N}),
\end{equation*}
\begin{equation*}
	\Delta'_{N}:\mathcal{K}_{0}(\Lambda^{\bullet}) \rightarrow  \mathcal{K}_{0}(0,\Lambda_{N})   \otimes \mathcal{K}_{0}(\Lambda^{\bullet-1},0) .
\end{equation*}
then by induction on $N$, we can prove that $\mathcal{K}_{0}(\Lambda^{\bullet})$ is isomorphic to the tensor products of integrable highest weight modules $_{\mathcal{A}}L(\Lambda_{k}), 1\leqslant k \leqslant N$. The images of simple peverse sheaves in $\mathcal{K}_{0}(\Lambda^{\bullet})$ form an $\mathcal{A}$-basis of the tensor product, which is invariant under Verdier duality. Moreover, we confirm that the linear map induced by Verdier duality is exactly equal to the linear map $\Psi$ considered in \cite{lusztig1992canonical} and \cite{bao2016canonical}. By the positivity of restriction functor, we can determine the sign appearing in  \cite[Theorem 7.16]{MR3177922} and prove that the basis consisting of simple peverse sheaves in $\mathcal{K}_{0}(\Lambda^{\bullet})$ coincides with the canonical basis of tensor products of integrable highest weight modules in the sense of \cite{bao2016canonical} without  sign. (See Proposition \ref{CBT}.) As a corollary, we obtain the positivity of the canonical basis of tensor products. Furthermore, we also deduce from the positivity of restriction functor that the transition matrix between the canonical basis of tensor product and the tensor product of canonical bases is positive, which is not included in \cite{zheng2007geometric} and  \cite{MR3177922}. (See Corollary \ref{tmp}.)

As a corollary, if we compose $\Delta'_{2}=\Delta'$ with the inverse of $\Delta_{2}=\Delta$ , we obtain an isomorphism of $_{\mathcal{A}}\mathbf{U}$-modules
\begin{equation*}
	R_{12}= \Delta' \circ (\Delta )^{-1}:L(\Lambda_{1}) \otimes L(\Lambda_{2}) \rightarrow L(\Lambda_{2}) \otimes L(\Lambda_{1}).
\end{equation*}
We prove that this isomorphism can be obtained from the quasi $\mathcal{R}$-matrix $\Theta$. Moreover, by the coassociativity of Lusztig's restriction functor, we discover  the following Yang-Baxter equation of isomorphisms between $_{\mathcal{A}}\mathbf{U}$-modules in the framework of Lusztig's semisimple perverse sheaves 
\begin{equation*}
	R_{23}R_{13}R_{12}=R_{12}R_{13}R_{23}: L(\Lambda_{1}) \otimes L(\Lambda_{2}) \otimes L(\Lambda_{3}) \xrightarrow{\cong} L(\Lambda_{3}) \otimes L(\Lambda_{2}) \otimes L(\Lambda_{1}).
\end{equation*}

 We remark that the construction above can be generalized to symmetrizable cases by using the periodic functor of quivers with automorphisms, which has been done in \cite{Lan2024LusztigSA}.  

\subsection*{Acknowledgements.}
J. Fang is partially supported by the New Cornerstone Science Foundation through the New Cornerstone Investigator Program awarded to Professor Xuhua He, and Y. Lan is supported by the National Natural Science Foundation of China [Grant No. 1288201]. This paper is a continuation of a collaborative work \cite{fang2023lusztig} with Prof. Jie Xiao, we are very grateful to his detailed discussion and constant encouragement. We are also  grateful to G. Lusztig and W. Wang. After the first version of this article appeared on arxiv, they asked us if the bases provided by simple perverse sheaves coincide with the canonical bases of tensor products defined in \cite{lusztig1992canonical} and \cite{bao2016canonical} for some certain cases, then we confirm it in the present paper.

 \section{Lusztig sheaves for $N$-framed quivers}
 In this section, we recall and generalize Lusztig's theory of semisimple perverse sheaves to $N$-framed quivers and refer \cite{MR1227098} for details.

 Given a symmetric Cartan datum $(I,(-,-))$, let $\mathbf{\Gamma}$ be the finite graph without loops associated to $(I,(-,-))$, where $I$ is the set of vertices and $H$ is the set of pairs consisting of edges with an orientation. More precisely, to give an edge with an orientation is equivalent to give $h',h'' \in I$ and we adapt the notation $h' \xrightarrow{h} h''$. Let $-:h \mapsto \bar{h}$ be the involution of $H$ such that $\bar{h}'=h'',\bar{h}''=h'$ and $\bar{h} \neq h$. An orientation of the graph $\Gamma$ is a subset $\Omega \subseteq H$ such that $\Omega \cap \bar{\Omega} =\emptyset$ and $\Omega \cup \bar{\Omega} = H$.

 \begin{definition}
 	For the quiver $Q=(I,H,\Omega)$ and a given positive integer $N$,  define the $N$-framed quiver of $Q$ to be the quiver $Q^{(N)}=(I^{(N)},H^{(N)},\Omega^{N} )$. The set of vertices is $I^{(N)}= I \cup \bigcup\limits_{1 \leqslant k \leqslant N}I^{k}$, where each $I^{k}= \{ i^{k}| i \in I \}$ is a copy of $I$. The set $H^{(N)}$ of edges is $H \cup \{i \rightarrow i^{k},i^{k} \rightarrow i| i \in I,1\leqslant  k \leqslant N \} $, and the orientation  $\Omega^{N}$ is $\Omega \cup \{i \rightarrow i^{k}| i \in I,1\leqslant  k \leqslant N \} $. 
 \end{definition}
 
 Here we give an example of the $2$-framed quiver for type $A_{2}$:
 
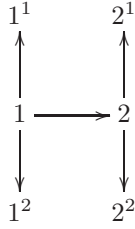
\[
 \xymatrix{
 	1^{1}        &  2^{1}  \\
 	1 \ar[r] \ar[u] \ar[d] & 2  \ar[u] \ar[d] \\
 	1^{2} & 2^{2}
 }
 \]
 
 Given a sequence of dominating weights $\Lambda^{\bullet}=(\Lambda_{1},\Lambda_{2},\cdots \Lambda_{N})$, 
 we fix a sequence of graded spaces $\mathbf{W}^{\bullet}=(\mathbf{W}^{1},\mathbf{W}^{2},\cdots \mathbf{W}^{N})$  such that each $\mathbf{W}^{k}$ is $I^{k}$-graded spaces and $\langle \Lambda_{k}, \alpha_{i}^{\vee} \rangle =\omega_{i^{k}} ={{\rm{dim}}} \mathbf{W}_{i^{k}}$. From now on, we denote the $i^{k}$ component $\mathbf{W}^{k}_{i^{k}}$ by $\mathbf{W}_{i^{k}}$ for simplicity. Let $\mathbf{k}=\overline{\mathbb{F}}_q$ be the algebraic closure of the finite field $\mathbb{F}_q$. Given $\nu \in \mathbb{N}[I]$ and an $I$-graded $\mathbf{k}$-vector space $\mathbf{V}$ of dimension vector $|\mathbf{V}|=\nu$, define the moduli space of $N$-framed quivers for $\mathbf{V}\oplus\mathbf{W}^{\bullet}=\mathbf{V}\oplus \bigoplus\limits_{1\leqslant k \leqslant N} \mathbf{W}^{k}$ as the following:
 \begin{equation*}
 	\begin{split}
 		\mathbf{E}_{\mathbf{V},\mathbf{W}^{\bullet},\Omega^{N}}
 		= \bigoplus\limits_{h \in \Omega} \mathbf{Hom}(\mathbf{V}_{h'},\mathbf{V}_{h''}) \oplus \bigoplus\limits_{i \in I, 1\leqslant k \leqslant N } \mathbf{Hom}(\mathbf{V}_{i},\mathbf{W}_{i^{k}}).
 	\end{split}
 \end{equation*}

 The algebraic group $G_{\mathbf{V}}= \prod \limits_{ i\in I} \mathbf{GL}(\mathbf{V}_{i})$ acts on $\mathbf{E}_{\mathbf{V},\mathbf{W}^{\bullet},\Omega^{N}}$ by
 \begin{equation*}
 	g \cdot ((x_{h})_{h\in \Omega}, (x^{k}_{i})_{i\in I,1\leqslant k \leqslant N}) =((g_{h''}x_{h}g_{h'}^{-1})_{h\in \Omega}, ( x^{k}_{i}g_{i}^{-1})_{i \in I,1\leqslant k \leqslant N}),
 \end{equation*}
 where $x_{h}$ corresponds to the component $\mathbf{Hom}(\mathbf{V}_{h'},\mathbf{V}_{h''})$ and $x_{i}^{k}$ corresponds to the component $\mathbf{Hom}(\mathbf{V}_{i},\mathbf{W}_{i^{k}})$, and the  group $G_{\mathbf{W}^{\bullet}}=\prod \limits_{ i\in I,1 \leqslant k \leqslant N}  \mathbf{GL}(\mathbf{W}_{i^{k}})$ acts on $\mathbf{E}_{\mathbf{V},\mathbf{W}^{\bullet},\Omega^{N}}$ by
 \begin{equation*}
 	g \cdot ((x_{h})_{h\in \Omega}, (x^{k}_{i})_{i\in I,1\leqslant k \leqslant N}) =((x_{h})_{h\in \Omega}, ( g_{i^{k}}x^{k}_{i})_{i \in I,1\leqslant k \leqslant N}).
 \end{equation*}

   Denote the algebraic group $G_{\mathbf{V}}\times G_{\mathbf{W}^{\bullet}}$ by $G_{\mathbf{V}\oplus \mathbf{W}^{\bullet}}$. The direct sum  $\bigoplus\limits_{1 \leqslant k \leqslant N}\mathbf{W}^{k}$ is an $I^{(N)}$-graded space, still denoted by $\mathbf{W}^{\bullet}$. Let $\mathcal{D}^{b}_{G_{\mathbf{V}\oplus \mathbf{W}^{\bullet}}}(\mathbf{E}_{\mathbf{V},\mathbf{W}^{\bullet},\Omega^{N}})$ be the  $G_{\mathbf{V}\oplus \mathbf{W}^{\bullet}}$-equivariant derived category of $\overline{\mathbb{Q}}_{l}$ constructible sheaves on $\mathbf{E}_{\mathbf{V},\mathbf{W}^{\bullet},\Omega^{N}}$. For any $n \in \mathbb{Z}$, we denote by $[n]$ the shift functor. (Following \cite{MR1088333}, the shift functor  $[n]$ always appear together with the Tate twist $(\frac{n}{2})$, and we always omit the Tate twist for simplicity.) We say complexes $A$ and $B$ are isomorphic up to shifts, if  $A$ and $B[n]$ are isomorphic for some $n \in \mathbb{Z}$. 
 
 \subsection{Lusztig's sheaves for $N$-framed quiver}
 For an $I^{(N)}$-graded space $\mathbf{U}=\mathbf{V} \oplus \mathbf{W}^{\bullet}$, we denote by $\mathcal{S}_{|\mathbf{U}|}$ the set of sequences $\boldsymbol{\mu}=(\mu^{1},\mu^{2},\cdots, \mu^{m})$ such that each $\mu^{l}$ is of the form $(i_{l})^{a_{l}}$ for some $i_{l}\in I^{(N)},a_{l}\in \mathbb{N}$ and $\sum_{l=1}^m\mu^l=|\mathbf{U}|$ and call an element $\boldsymbol{\mu} \in \mathcal{S}_{|\mathbf{U}|}$ a flag type of $\mathbf{U}$.  If $\boldsymbol{\mu}'=((i'_{1})^{a'_{1}} , \cdots ,(i'_{m})^{a'_{m}}) \in \mathcal{S}_{|\mathbf{U}'|}$ and $\boldsymbol{\mu}''=((i''_{1})^{a''_{1}},\cdots,(i''_{n})^{a''_{n}} )\in \mathcal{S}_{|\mathbf{U}''|}$, then let the sequence 
 $\boldsymbol{\mu}' \boldsymbol{\mu}''=((i'_{1})^{a'_{1}} , \cdots ,(i'_{m})^{a'_{m}},(i''_{1})^{a''_{1}},\cdots,(i''_{n})^{a''_{n}} ) \in \mathcal{S}_{|\mathbf{U}'\oplus\mathbf{U}''|}$  be the flag type obtained by connecting $\boldsymbol{\mu}'$ and $\boldsymbol{\mu}''$.
 
 For a fixed order $i_{1},i_{2},\cdots,i_{n}$ of $I$, let $\boldsymbol{d}^{k}=((i^{k}_{1} )^{\omega_{i^{k}_{1}}},(i^{k}_{2} )^{\omega_{i^{k}_{2}}},\cdots, (i^{k}_{n} )^{\omega_{i^{k}_{n}}})$ be the associated flag type of $\mathbf{W}^{k}$. Then for any flag types $\boldsymbol{\nu}^{k}$ of $I$-graded spaces $\mathbf{V}^{k}, 1 \leqslant k \leqslant N$, the sequence $\boldsymbol{\nu}^{1}\boldsymbol{d}^{1}\boldsymbol{\nu}^{2}\boldsymbol{d}^{2}\cdots \boldsymbol{\nu}^{N}\boldsymbol{d}^{N}$ is a flag type of $\mathbf{V}\oplus\mathbf{W}^{\bullet}$ for some $\mathbf{V}$ such that $|\mathbf{V}|= \sum\limits_{1\leqslant k \leqslant N} |\mathbf{V}^{k}|$.

 For any $\boldsymbol{\mu}=(\mu^{1},\mu^{2},\cdots, \mu^{m}) \in \mathcal{S}_{|\mathbf{V}\oplus \mathbf{W}^{\bullet}|}$, the flag variety $\mathcal{F}_{\boldsymbol{\mu},\Omega^{N}}$ is defined to be the variety consisting of $(x,f)$, where $x \in \mathbf{E}_{\mathbf{V},\mathbf{W}^{\bullet},\Omega^{N}}$ and $f=(0=\mathbf{U}^{m} \subseteq \mathbf{U}^{m-1} \subseteq \cdots \subseteq \mathbf{U}^{0}=\mathbf{U})$  is a flag of $\mathbf{U}=\mathbf{V}\oplus \mathbf{W}^{\bullet}$ such that  $x(\mathbf{U}^{k}) \subseteq \mathbf{U}^{k}$ and $|\mathbf{U}^{k-1}/\mathbf{U}^{k}|=\mu^{k}$ for every $1\leqslant k \leqslant m$. Following \cite{MR1088333}, the flag variety $\mathcal{F}_{\boldsymbol{\mu},\Omega^{N}}$ is smooth and there is a natural proper map $\pi_{\boldsymbol{\mu}}: \mathcal{F}_{\boldsymbol{\mu},\Omega^{N}} \rightarrow \mathbf{E}_{\mathbf{V},\mathbf{W}^{\bullet},\Omega^{N}}$ sending $(x,f)$ to $x$, then by the decomposition theorem in \cite{MR751966}, the complex $L_{\boldsymbol{\mu}}= (\pi_{\boldsymbol{\mu}})_{!} \bar{\mathbb{Q}}_{l}[{\rm{dim}} \mathcal{F}_{\boldsymbol{\mu},\Omega^{N}}]$ is a $G_{\mathbf{V}\oplus\mathbf{W}^{\bullet}}$-equivariant semisimple complex on $\mathbf{E}_{\mathbf{V},\mathbf{W}^{\bullet},\Omega^{N}}$.

\begin{definition}
(1)	Let $\mathcal{P}_{\mathbf{V}\oplus\mathbf{W}^{\bullet},\Omega^{N}}$ be the set of isomorphism classes of simple perverse sheaves $L$ on $\mathbf{E}_{\mathbf{V},\mathbf{W}^{\bullet},\Omega^{N}}$ such that $L[d]$ appears as a direct summand  of  $L_{\boldsymbol{\mu}}$ for some $d \in \mathbb{Z}$ and $\boldsymbol{\mu} \in \mathcal{S}_{|\mathbf{V}\oplus \mathbf{W}^{\bullet}|}$. Define $\mathcal{Q}_{\mathbf{V}\oplus\mathbf{W}^{\bullet},\Omega^{N}}$ and $\mathcal{Q}'_{\mathbf{V}\oplus\mathbf{W}^{\bullet},\Omega^{N}}$ to be the full subcategory of $\mathcal{D}^{b}_{G_{\mathbf{V}}}(\mathbf{E}_{\mathbf{V},\mathbf{W}^{\bullet},\Omega^{N}})$ and $\mathcal{D}^{b}_{G_{\mathbf{V}\oplus \mathbf{W}^{\bullet}}}(\mathbf{E}_{\mathbf{V},\mathbf{W}^{\bullet},\Omega^{N}})$ consisting of finite direct sums of complexes of the form $L[d']$ for various $L$ in $\mathcal{P}_{\mathbf{V}\oplus\mathbf{W}^{\bullet},\Omega^{N}}$ and various $d' \in \mathbb{Z}$ respectively.
	
(2)	Let $\mathcal{P}_{\mathbf{V},\mathbf{W}^{\bullet},\Omega^{N}}$  be the set of isomorphism classes of simple perverse sheaves $L$ on $\mathbf{E}_{\mathbf{V},\mathbf{W}^{\bullet},\Omega^{N}}$ such that $L[d]$ appears as a direct summand  of   $L_{\boldsymbol{\nu}^{1}\boldsymbol{d}^{1}\boldsymbol{\nu}^{2}\boldsymbol{d}^{2}\cdots \boldsymbol{\nu}^{N}\boldsymbol{d}^{N}}$, where $d$ is some integer and each $\boldsymbol{\nu}^{k}$ is some flag type for some $I$-graded space $\mathbf{V}^{k}$ such that $\sum\limits_{1\leqslant k \leqslant N} |\mathbf{V}^{k}|=|\mathbf{V}|$. Define $\mathcal{Q}_{\mathbf{V},\mathbf{W}^{\bullet},\Omega^{N}}$ and $\mathcal{Q}'_{\mathbf{V},\mathbf{W}^{\bullet},\Omega^{N}}$ to be the full subcategory of $\mathcal{D}^{b}_{G_{\mathbf{V}}}(\mathbf{E}_{\mathbf{V},\mathbf{W}^{\bullet},\Omega^{N}})$ and $\mathcal{D}^{b}_{G_{\mathbf{V}\oplus\mathbf{W}^{\bullet}}}(\mathbf{E}_{\mathbf{V},\mathbf{W}^{\bullet},\Omega^{N}})$ consisting of finite direct sums of complexes of the form $L[d']$ for various $L$ in $\mathcal{P}_{\mathbf{V},\mathbf{W}^{\bullet},\Omega^{N}}$ and various $d' \in \mathbb{Z}$ respectively.
\end{definition}

\subsection{Induction functor}

 Fix two sequences of graded spaces ${^{'}\mathbf{W}}^{\bullet}=({^{'}\mathbf{W}}^{1},\cdots,{^{'}\mathbf{W}}^{N}) $ and ${^{''}\mathbf{W}}^{\bullet}=({^{''}\mathbf{W}}^{1},\cdots,{^{''}\mathbf{W}}^{N})$ such that $|\mathbf{W}^{k}|=|{^{'}\mathbf{W}}^{k}|+|{^{''}\mathbf{W}}^{k}|$ for each $1 \leqslant k \leqslant N$ and fix $I$-graded vector spaces $\mathbf{V},\mathbf{V}',\mathbf{V}''$ such that $|\mathbf{V}|=|\mathbf{V}'|+|\mathbf{V}''|$. Let $\mathbf{E}'_{\Omega}$ be the variety consisting of $(x,\tilde{\mathbf{U}}, \rho_{1}, \rho_{2})$, where $x \in \mathbf{E}_{\mathbf{V},\mathbf{W}^{\bullet},\Omega^{N}}$, $\tilde{\mathbf{U}}$ is an $I^{(N)}$-graded $x$-stable subspace of $\mathbf{V}\oplus \mathbf{W}^{\bullet}$ of dimension vector $|\mathbf{V}''\oplus {^{''}\mathbf{W}}^{\bullet}|$ and $ \rho_{1}: \mathbf{V}\oplus \mathbf{W}^{\bullet}/\tilde{\mathbf{U}} \simeq \mathbf{V}'\oplus {^{'}\mathbf{W}}^{\bullet},\rho_{2}:\tilde{\mathbf{U}} \simeq \mathbf{V}''\oplus {^{''}\mathbf{W}}^{\bullet}$ are linear isomorphisms. Here we say $\tilde{\mathbf{U}}$ is $x$-stable if and only if $x_{h}(\tilde{\mathbf{U}}_{h'}) \subseteq \tilde{\mathbf{U}}_{h''}$ for any $h \in \Omega^{N}$.  Let $\mathbf{E}''_{\Omega}$ be the variety consisting of $(x,\tilde{\mathbf{U}})$ as above. Consider the following diagram
 \begin{center}
 	$\mathbf{E}_{\mathbf{V}',{^{'}\mathbf{W}}^{\bullet},\Omega^{N}} \times \mathbf{E}_{\mathbf{V}'',{^{''}\mathbf{W}}^{\bullet},\Omega^{N}} \xleftarrow{p_{1}} \mathbf{E}'_{\Omega} \xrightarrow{p_{2}} \mathbf{E}''_{\Omega} \xrightarrow{p_{3}} \mathbf{E}_{\mathbf{V},\mathbf{W}^{\bullet},\Omega^{N}}$
 \end{center}
 where $p_{1}(x,\tilde{\mathbf{U}},\rho_{1},\rho_{2})=(\rho_{1,\ast}(\bar{x}|_{\mathbf{V}\oplus \mathbf{W}^{\bullet}/\tilde{\mathbf{U}}}),\rho_{2,\ast}(x|_{\tilde{\mathbf{U}}})  )$, $p_{2}(x,\tilde{\mathbf{U}},\rho_{1},\rho_{2}) =(x, \tilde{\mathbf{U}}) $ and $p_{3}(x,\tilde{\mathbf{U}})=x$,  where $\bar{x}|_{\mathbf{V}\oplus \mathbf{W}^{\bullet}/\tilde{\mathbf{U}}}$ is the natural linear map induced by $x$ on the quotient space $\mathbf{V}\oplus \mathbf{W}^{\bullet}/\tilde{\mathbf{U}}$ and $x|_{\tilde{\mathbf{U}}}$ is the restriction of $x$ on the subspace $\tilde{\mathbf{U}}$, then $ \rho_{1,\ast}(\bar{x}|_{\mathbf{V}/\tilde{\mathbf{U}}})= \rho_{1} (\bar{x}|_{\mathbf{V}/\tilde{\mathbf{U}}}) \rho_{1}^{-1}\in \mathbf{E}_{\mathbf{V}',{^{'}\mathbf{W}}^{\bullet},\Omega^{N}}$ and $\rho_{2,\ast}(x|_{\tilde{\mathbf{U}}})=\rho_{2}(x|_{\tilde{\mathbf{U}}}) \rho_{2}^{-1}\in \mathbf{E}_{\mathbf{V}'',{^{''}\mathbf{W}}^{\bullet},\Omega^{N}}$.   Lusztig's induction functor for $N$-framed quiver is defined by
 \begin{center}
 $\mathbf{Ind}^{\mathbf{V}}_{\mathbf{V}',\mathbf{V}''}:\mathcal{D}^{b}_{G_{\mathbf{V}'\oplus{^{'}\mathbf{W}}^{\bullet}}}(\mathbf{E}_{\mathbf{V}',{^{'}\mathbf{W}}^{\bullet},\Omega^{N}}) \boxtimes \mathcal{D}^{b}_{G_{\mathbf{V}''\oplus{^{''}\mathbf{W}}^{\bullet}}}(\mathbf{E}_{\mathbf{V}'',{^{''}\mathbf{W}}^{\bullet},\Omega^{N}}) \rightarrow \mathcal{D}^{b}_{G_{\mathbf{V}\oplus \mathbf{W}^{\bullet}}}(\mathbf{E}_{\mathbf{V},\mathbf{W}^{\bullet},\Omega^{N}}),$\\
  	$ \mathbf{Ind}^{\mathbf{V}}_{\mathbf{V}',\mathbf{V}''}(A\boxtimes B)=\mathbf{Ind}^{\mathbf{V}\oplus \mathbf{W}^{\bullet}}_{\mathbf{V}'\oplus{^{'}\mathbf{W}}^{\bullet},\mathbf{V}''\oplus{^{''}\mathbf{W}}^{\bullet}} (A\boxtimes B)= (p_{3})_{!}(p_{2})_{\flat}(p_{1})^{\ast}(A\boxtimes B)[d_{1}-d_{2}].$
 \end{center}
 Here $p_{1}$ is smooth with connected fibers, $p_{2}$ is a principle $G_{\mathbf{V}'} \times G_{ {^{'}\mathbf{W}}^{\bullet}} \times G_{\mathbf{V}''}\times G_{ {^{''}\mathbf{W}}^{\bullet}}$-bundle, $p_{3}$ is proper and $d_{1},d_{2}$ are the dimensions of the fibers of $p_{1}$ and $p_{2}$ respectively. For principle $G$-bundle  $f:Y \rightarrow X$, the functor $f_{\flat}$ is defined to be the inverse of $f^{\ast}$, which gives equivalence $\mathcal{D}^{b}(X) \cong \mathcal{D}^{b}_{G}(Y)$. (See details in \cite[Theorem 6.5.9]{MR4337423}.)

 \begin{proposition}\cite[Lemma 3.2]{MR1088333} \label{indformula}
 	For any flag types $\boldsymbol{\mu}'\in \mathcal{S}_{|\mathbf{V}'\oplus{^{'}\mathbf{W}}^{\bullet} |}$ and $\boldsymbol{\mu}''\in \mathcal{S}_{|\mathbf{V}''\oplus{^{''}\mathbf{W}}^{\bullet}|}$,
 	\begin{equation*}
 		\mathbf{Ind}^{\mathbf{V}}_{\mathbf{V}',\mathbf{V}''}(L_{\boldsymbol{\mu}'} \boxtimes L_{\boldsymbol{\mu}''})= L_{\boldsymbol{\mu}' \boldsymbol{\mu}''}.
 	\end{equation*}
 In particular, $\mathbf{Ind}^{\mathbf{V}}_{\mathbf{V}',\mathbf{V}''}$ restricts to a functor $ \mathcal{Q}'_{\mathbf{V}'\oplus{^{'}\mathbf{W}}^{\bullet},\Omega^{N} } \boxtimes \mathcal{Q}'_{\mathbf{V}''\oplus{^{''}\mathbf{W}}^{\bullet},\Omega^{N}} \rightarrow \mathcal{Q}'_{\mathbf{V}\oplus \mathbf{W}^{\bullet},\Omega^{N}}$.
\end{proposition}
 
 \subsection{Restriction functor}
 Fix a decomposition $(\mathbf{T} \oplus {^{1}\mathbf{W}}^{\bullet} )\oplus (\mathbf{W}\oplus {^{2}\mathbf{W}}^{\bullet} )=\mathbf{V} \oplus \mathbf{W}^{\bullet}$ such that $|\mathbf{T}|=\nu'$ and $|\mathbf{W}|=\nu''$, let $F_{\Omega}$ be the closed subvariety of $\mathbf{E}_{\mathbf{V},\mathbf{W}^{\bullet},\Omega^{N}}$ consisting of $x$ such that the subspace $\mathbf{W}\oplus {^{2}\mathbf{W}}^{\bullet} $ is $x$-stable. Consider the following diagram
 \begin{center}
 	$\mathbf{E}_{\mathbf{T},{^{1}\mathbf{W}}^{\bullet},\Omega^{N}} \times \mathbf{E}_{\mathbf{W},{^{2}\mathbf{W}}^{\bullet},\Omega^{N}} \xleftarrow{\kappa_{\Omega} } F_{\Omega} \xrightarrow{\iota_{\Omega}} \mathbf{E}_{\mathbf{V},\mathbf{W}^{\bullet},\Omega^{N}}$
 \end{center} 
 where $\iota_{\Omega}$ is the natural embedding and $\kappa_{\Omega}(x)=(\overline{x}|_{\mathbf{T}\oplus {^{1}\mathbf{W}}^{\bullet}},x|_{\mathbf{W}\oplus {^{2}\mathbf{W}}^{\bullet}}) \in \mathbf{E}_{\mathbf{T},{^{1}\mathbf{W}}^{\bullet},\Omega^{N}} \times \mathbf{E}_{\mathbf{W},{^{2}\mathbf{W}}^{\bullet},\Omega^{N}} $ for any $x \in F_{\Omega}$. Notice that $\kappa_{\Omega}$ is a vector bundle. Lusztig's restriction functor for $N$-framed quiver is defined by
 \begin{center}
 	$\mathbf{Res}^{\mathbf{V}}_{\mathbf{T},\mathbf{W}}: \mathcal{D}^{b}_{G_{\mathbf{V}\oplus\mathbf{W}^{\bullet}}}(\mathbf{E}_{\mathbf{V},\mathbf{W}^{\bullet},\Omega^{N}}) \rightarrow \mathcal{D}^{b}_{G_{\mathbf{T}\oplus{^{1}\mathbf{W}}^{\bullet}} \times G_{\mathbf{W}\oplus{^{2}\mathbf{W}}^{\bullet}}}(\mathbf{E}_{\mathbf{T},{^{1}\mathbf{W}}^{\bullet},\Omega^{N}}\times \mathbf{E}_{\mathbf{W},{^{2}\mathbf{W}}^{\bullet},\Omega^{N}}),$\\
 	$\mathbf{Res}^{\mathbf{V}}_{\mathbf{T},\mathbf{W}}(C)=\mathbf{Res}^{\mathbf{V}\oplus \mathbf{W}^{\bullet}}_{\mathbf{T}\oplus {^{1}\mathbf{W}}^{\bullet},\mathbf{W}\oplus {^{2}\mathbf{W}}^{\bullet}}(C)=(\kappa_{\Omega})_{!} (\iota_{\Omega})^{\ast}(C)[-\langle |\mathbf{T} \oplus {^{1}\mathbf{W}}^{\bullet}| ,|\mathbf{W}\oplus {^{2}\mathbf{W}}^{\bullet} | \rangle],$
 \end{center}
 where $\langle-,-\rangle $is the  Euler form associated to $Q^{(N)}$, see the definition of Euler form in \cite[Equation (1.2)]{MR3202708}.

 \begin{proposition}\cite[Proposition 4.2]{MR1088333} \label{res formula}	
 	For $\boldsymbol{\mu}\in \mathcal{S}_{|\mathbf{V}\oplus \mathbf{W}^{\bullet}|}$,
 	\begin{equation*}
 		\mathbf{Res}^{\mathbf{V}}_{\mathbf{T},\mathbf{W}}( L_{\boldsymbol{\nu}}) =\bigoplus \limits_{\boldsymbol{\tau}+\boldsymbol{\omega}=\boldsymbol{\nu} } L_{\boldsymbol{\tau}} \boxtimes L_{\boldsymbol{\omega}}[M(\boldsymbol{\tau},\boldsymbol{\omega})],
 	\end{equation*}
 	where $\boldsymbol{\tau}$ runs over $\mathcal{S}_{|\mathbf{T}\oplus {^{1}\mathbf{W}}^{\bullet}|}$, $\boldsymbol{\omega}$ runs over $ \mathcal{S}_{|\mathbf{W}\oplus {^{2}\mathbf{W}}^{\bullet}|} $,  and the integer $	M(\boldsymbol{\tau},\boldsymbol{\omega})$ is a certain interger given by the formula in \cite[Section 9.2.11]{MR1227098}. In particular, the restriction functor $\mathbf{Res}^{\mathbf{V}}_{\mathbf{T},\mathbf{W}}$ restricts to a functor $\mathcal{Q}'_{\mathbf{V}\oplus \mathbf{W}^{\bullet},\Omega^{N}}\rightarrow \mathcal{Q}'_{\mathbf{T}\oplus {^{1}\mathbf{W}}^{\bullet},\Omega^{N}}\boxtimes \mathcal{Q}'_{\mathbf{W}\oplus {^{2}\mathbf{W}}^{\bullet},\Omega^{N}}. $
 \end{proposition}

 \begin{remark}
 	Let $\mathcal{D}^{b,ss}_{G_{\mathbf{V}\oplus\mathbf{W}^{\bullet}}}(\mathbf{E}_{\mathbf{V},\mathbf{W}^{\bullet},\Omega^{N}})$ be the full subcatgory of $\mathcal{D}^{b}_{G_{\mathbf{V}\oplus\mathbf{W}^{\bullet}}}(\mathbf{E}_{\mathbf{V},\mathbf{W}^{\bullet},\Omega^{N}})$ consisting of semisimple complexes, then by \cite{MR4524567} the induction functor and restrcition functor also restrict to 
 	\begin{align*}
 		&\mathbf{Ind}^{\mathbf{V}}_{\mathbf{V}',\mathbf{V}''}:\mathcal{D}^{b,ss}_{G_{\mathbf{V}'\oplus{^{'}\mathbf{W}}^{\bullet}}}(\mathbf{E}_{\mathbf{V}',{^{'}\mathbf{W}}^{\bullet},\Omega^{N}})\times \mathcal{D}^{b,ss}_{G_{\mathbf{V}''\oplus{^{''}\mathbf{W}}^{\bullet}}}(\mathbf{E}_{\mathbf{V}'',{^{''}\mathbf{W}}^{\bullet},\Omega^{N}})\rightarrow \mathcal{D}^{b,ss}_{G_{\mathbf{V}\oplus\mathbf{W}^{\bullet}}}(\mathbf{E}_{\mathbf{V},\mathbf{W}^{\bullet},\Omega^{N}}),\\
 		&\mathbf{Res}^{\mathbf{V}}_{\mathbf{T},\mathbf{W}}:\mathcal{D}^{b,ss}_{G_{\mathbf{V}\oplus\mathbf{W}^{\bullet}}}(\mathbf{E}_{\mathbf{V},\mathbf{W}^{\bullet},\Omega^{N}})\rightarrow \mathcal{D}^{b,ss}_{G_{\mathbf{T}\oplus{^{1}\mathbf{W}}^{\bullet}}\times G_{\mathbf{W}\oplus{^{2}\mathbf{W}}^{\bullet}}}(\mathbf{E}_{\mathbf{T},{^{1}\mathbf{W}}^{\bullet},\Omega^{N}} \times \mathbf{E}_{\mathbf{W},{^{2}\mathbf{W}}^{\bullet},\Omega^{N}}).
 	\end{align*}
 \end{remark}
 
 \subsection{Fourier-Deligne transform}

 If $\tilde{\Omega}$ is another orientation of $Q$, define
 \begin{equation*}
 	\begin{split}
 		\mathbf{E}_{\mathbf{V},\mathbf{W}^{\bullet},\Omega^{N}\cup \tilde{\Omega}^{N}}=& \bigoplus\limits_{h \in \Omega^{N}\cup \tilde{\Omega}^{N}} \mathbf{Hom}(\mathbf{V}_{h'},\mathbf{V}_{h''})\\
 		=&   \bigoplus\limits_{h \in \Omega\cup \tilde{\Omega}} \mathbf{Hom}(\mathbf{V}_{h'},\mathbf{V}_{h''}) \oplus \bigoplus\limits_{i \in I, 1\leqslant k \leqslant N } \mathbf{Hom}(\mathbf{V}_{i},\mathbf{W}_{i^{k}}).
 	\end{split}
 \end{equation*}
 and  define $T: \mathbf{E}_{\mathbf{V},\mathbf{W}^{\bullet},\Omega^{N}\cup \tilde{\Omega}^{N}} \longrightarrow \mathbf{k}$ by $$T(x)=\sum \limits_{h \in \Omega \backslash \tilde{\Omega}}tr(x_{h}x_{\bar{h}}).$$

  Fix a nontrivial character $\mathbb{F}_{q} \rightarrow \bar{\mathbb{Q}}_{l}^{\ast}$, then this character defines an Artin-Schreier local system of rank $1$ on $\mathbf{k}$ and the pull back $\mathcal{L}_{T}$ of the Artin-Schreier local system under $T$ is a well-defined $G_{\mathbf{V}}$-equivariant local system of rank $1$ on $\mathbf{E}_{\mathbf{V},\mathbf{W}^{\bullet},\Omega^{N}\cup \tilde{\Omega}^{N}}$. The Fourier-Deligne transform for $N$-framed quivers is defined by
 \begin{align*}
 	&\mathcal{F}_{\Omega,\tilde{\Omega}}:\mathcal{D}^{b}_{G_{\mathbf{V}\oplus\mathbf{W}^{\bullet}}}(\mathbf{E}_{\mathbf{V},\mathbf{W}^{\bullet},\Omega^{N}}) \rightarrow \mathcal{D}^{b}_{G_{\mathbf{V}\oplus\mathbf{W}^{\bullet}}}(\mathbf{E}_{\mathbf{V},\mathbf{W}^{\bullet},
 		\tilde{\Omega}^{N}})\\ &\mathcal{F}_{\Omega,\tilde{\Omega}}(L)=\delta'_{!}(\delta^{\ast}(L)\otimes \mathcal{L}_{T})[\sum\limits_{h \in \Omega \backslash \tilde{\Omega}}{\rm{dim}} \mathbf{V}_{h'}{\rm{dim}}\mathbf{V}_{h''}],
 \end{align*}  
where $\delta,\delta'$ are the natural projections defined by	$$\delta((x_{h})_{h \in \Omega^{N}\cup \tilde{\Omega}^{N}} )= ((x_{h})_{h \in \Omega^{N}}),$$  $$\delta'((x_{h})_{h \in \Omega^{N}\cup \tilde{\Omega}^{N}} )= ((x_{h})_{h \in \tilde{\Omega}^{N}}).$$ Regard $N$-framed quivers as quivers and apply \cite[Theorem 5.4]{MR1088333}, we get the following proposition and corollary.
 \begin{proposition}  \label{FD0}
 	With the notations above, for any semisimple perverse sheaves $L_{1}$ and $L_{2}$,
 	\begin{center}
 		$\mathcal{F}_{\Omega,\tilde{\Omega}}(\mathbf{Ind}^{\mathbf{V}} _{\mathbf{V}',\mathbf{V}''}(L_{1} \boxtimes L_{2})) \cong  \mathbf{Ind}^{\mathbf{V}} _{\mathbf{V}',\mathbf{V}''}(\mathcal{F}_{\Omega,\tilde{\Omega}}(L_{1}) \boxtimes \mathcal{F}_{\Omega,\tilde{\Omega}}(L_{2})).$
 	\end{center}
 In particular, $\mathcal{F}_{\Omega,\tilde{\Omega}}(L_{\boldsymbol{\nu}})=L_{\boldsymbol{\nu}}$ for any $\boldsymbol{\nu} \in \mathcal{S}_{|\mathbf{V}\oplus \mathbf{W}^{\bullet}|}.$
 \end{proposition}
 \begin{corollary}\label{FD2}
 	The functor $\mathcal{F}_{\Omega,\tilde{\Omega}}$ induces an equivalence of categories $\mathcal{Q}_{\mathbf{V}\oplus\mathbf{W}^{\bullet},\Omega^{N}} \cong \mathcal{Q}_{\mathbf{V}\oplus\mathbf{W}^{\bullet},\tilde{\Omega}^{N}}$ and a bijection $\eta_{\Omega,\tilde{\Omega}}:\mathcal{P}_{\mathbf{V}\oplus\mathbf{W}^{\bullet},\Omega^{N}}\rightarrow \mathcal{P}_{\mathbf{V}\oplus\mathbf{W}^{\bullet},\tilde{\Omega}^{N}}$. 	The functor $\mathcal{F}_{\Omega,\tilde{\Omega}}$ also restricts to an equivalence of categories $\mathcal{Q}_{\mathbf{V},\mathbf{W}^{\bullet},\Omega^{N}} \cong \mathcal{Q}_{\mathbf{V},\mathbf{W}^{\bullet},\tilde{\Omega}^{N}}$ and a bijection $\eta_{\Omega,\tilde{\Omega}}:\mathcal{P}_{\mathbf{V},\mathbf{W}^{\bullet},\Omega^{N}}\rightarrow \mathcal{P}_{\mathbf{V},\mathbf{W}^{\bullet},\tilde{\Omega}^{N}}$.  Moreover, for three orientations $\Omega,\Omega',\Omega''$ of $Q$, we have $\eta_{\Omega',\Omega''}\eta_{\Omega,\Omega'}=\eta_{\Omega,\Omega''}$. 
 \end{corollary}

With  Corollary \ref{FD2}, we can denote the categories $\mathcal{Q}_{\mathbf{V}\oplus\mathbf{W}^{\bullet},\Omega^{N}}$ and $\mathcal{Q}_{\mathbf{V},\mathbf{W}^{\bullet},\Omega^{N}}$ by $\mathcal{Q}_{\mathbf{V}\oplus\mathbf{W}^{\bullet}}$ and $\mathcal{Q}_{\mathbf{V},\mathbf{W}^{\bullet}}$ respectively. We can also denote the sets $\mathcal{P}_{\mathbf{V}\oplus\mathbf{W}^{\bullet},\Omega^{N}}$ and $\mathcal{P}_{\mathbf{V},\mathbf{W}^{\bullet},\Omega^{N}}$ by $\mathcal{P}_{\mathbf{V}\oplus\mathbf{W}^{\bullet}}$ and $\mathcal{P}_{\mathbf{V},\mathbf{W}^{\bullet}}$ respectively.

 \begin{remark}\label{remarkFD}
 (1) Indeed, the Fourier-Deligne transform gives an equivalence of the bounded derived categories $\mathcal{D}^{b}_{G_{\mathbf{V}\oplus\mathbf{W}^{\bullet}}}(\mathbf{E}_{\mathbf{V},\mathbf{W}^{\bullet},\Omega^{N}}) \rightarrow \mathcal{D}^{b}_{G_{\mathbf{V}\oplus\mathbf{W}^{\bullet}}}(\mathbf{E}_{\mathbf{V},\mathbf{W}^{\bullet},
 		\tilde{\Omega}^{N}})$ and also restricts to an equivalence of  $\mathcal{D}^{b,ss}_{G_{\mathbf{V}\oplus\mathbf{W}^{\bullet}}}(\mathbf{E}_{\mathbf{V},\mathbf{W}^{\bullet},\Omega^{N}}) \rightarrow \mathcal{D}^{b,ss}_{G_{\mathbf{V}\oplus\mathbf{W}^{\bullet}}}(\mathbf{E}_{\mathbf{V},\mathbf{W}^{\bullet},
 		\tilde{\Omega}^{N}})$, see details in the proof of \cite[Theorem 3.13]{MR3202708}. 	
 (2) If we regard all vertices in $I^{(N)}$ as unframed vertices, we can see that the Fourier-Deligne transforms can also transform arrows $i \rightarrow i^{k}$ to $i^{k} \rightarrow i$. These Fourier-Deligne transforms will be only used in Section 4.	
 \end{remark}

 \section{Localizations and functors}
 \subsection{Localization at $i$}
 In this and the next subsection, we always forget the $G_{\mathbf{W}^{\bullet}}$-equivariant structure and fix an $i \in I$ and an orientation $\Omega_{i}$ such that $i$ is a source in $Q=(I,H,\Omega)$. Then there is a partition $\bigcup \limits_{p \geq 0} \mathbf{E}^{p}_{\mathbf{V},\mathbf{W}^{\bullet},i}=\mathbf{E}_{\mathbf{V},\mathbf{W}^{\bullet},\Omega_{i}^{N}} $, where $\mathbf{E}^{p}_{\mathbf{V},\mathbf{W}^{\bullet},i}$ is the locally closed subset defined by
 \begin{equation*}
 	\mathbf{E}^{p}_{\mathbf{V},\mathbf{W}^{\bullet},i}=\{ x\in \mathbf{E}_{\mathbf{V},\mathbf{W}^{\bullet},\Omega^{N}_{i}}| {{\rm{dim}}}{\rm{Ker}}((\bigoplus \limits_{h \in \Omega^{N}_{i}, h'=i} x_{h}): \mathbf{V}_{i} \rightarrow \bigoplus\limits_{h'=i}\mathbf{V}_{h''}\oplus\bigoplus\limits_{1\leqslant k \leqslant N} \mathbf{W}_{i^{k}})=p  \}.
 \end{equation*}
In particular, $\mathbf{E}^{0}_{\mathbf{V},\mathbf{W}^{\bullet},i}$ is an open subset and its complement $\mathbf{E}^{\geqslant 1}_{\mathbf{V},\mathbf{W}^{\bullet},i}$ is closed.
 
 Let $\mathcal{N}_{\mathbf{V},i}$ be the full subcategory of $\mathcal{D}^{b}_{G_{\mathbf{V}}}(\mathbf{E}_{\mathbf{V},\mathbf{W}^{\bullet},\Omega^{N}_{i}})$ consisting of those objects whose supports contained in the closed subset $\mathbf{E}^{\geqslant 1}_{\mathbf{V},\mathbf{W}^{\bullet},i}$, then $\mathcal{N}_{\mathbf{V},i}$ is a thick subcategory. Hence we can define the localization  $\mathcal{D}^{b}_{G_{\mathbf{V}}}(\mathbf{E}_{\mathbf{V},\mathbf{W}^{\bullet},\Omega^{N}_{i}})/ \mathcal{N}_{\mathbf{V},i}$  to be the Verdier quotient of $\mathcal{D}^{b}_{G_{\mathbf{V}}}(\mathbf{E}_{\mathbf{V},\mathbf{W}^{\bullet},\Omega^{N}_{i}})$ by the thick subcategory $\mathcal{N}_{\mathbf{V},i}$. The Verdier quotient $\mathcal{D}^{b}_{G_{\mathbf{V}}}(\mathbf{E}_{\mathbf{V},\mathbf{W}^{\bullet},\Omega^{N}_{i}})/ \mathcal{N}_{\mathbf{V},i}$ is a triangulated category and admits a $t$-structure induced by the perverse $t$-structure of  $\mathcal{D}^{b}_{G_{\mathbf{V}}}(\mathbf{E}_{\mathbf{V},\mathbf{W}^{\bullet},\Omega^{N}_{i}})$, see  \cite{MR1074006}. 
 
 \begin{definition}
 	(1) Define the localization $\mathcal{Q}_{\mathbf{V},\mathbf{W}^{\bullet}}/\mathcal{N}_{\mathbf{V},i}$ of  $\mathcal{Q}_{\mathbf{V},\mathbf{W}^{\bullet}}$  at $i$ to be the full subcategory of the localization $\mathcal{D}^{b}_{G_{\mathbf{V}}}(\mathbf{E}_{\mathbf{V},\mathbf{W}^{\bullet},\Omega^{N}_{i}})/ \mathcal{N}_{\mathbf{V},i}$ consisting of objects which are isomorphic to some objects of $\mathcal{Q}_{\mathbf{V},\mathbf{W}^{\bullet}}$ in $\mathcal{D}^{b}_{G_{\mathbf{V}}}(\mathbf{E}_{\mathbf{V},\mathbf{W}^{\bullet},\Omega^{N}_{i}})/ \mathcal{N}_{\mathbf{V},i}$.
 	
 	(2) Similarly, define the localization $\mathcal{D}^{b,ss}_{G_{\mathbf{V}}}(\mathbf{E}_{\mathbf{V},\mathbf{W}^{\bullet},\Omega^{N}_{i}})/ \mathcal{N}_{\mathbf{V},i}$ of $\mathcal{D}^{b,ss}_{G_{\mathbf{V}}}(\mathbf{E}_{\mathbf{V},\mathbf{W}^{\bullet},\Omega^{N}_{i}})$ at $i$ to be the  full subcategory of the localization $\mathcal{D}^{b}_{G_{\mathbf{V}}}(\mathbf{E}_{\mathbf{V},\mathbf{W}^{\bullet},\Omega^{N}_{i}})/ \mathcal{N}_{\mathbf{V},i}$ consisting of objects which are isomorphic to some semisimple complexes in $\mathcal{D}^{b}_{G_{\mathbf{V}}}(\mathbf{E}_{\mathbf{V},\mathbf{W}^{\bullet},\Omega^{N}_{i}})/ \mathcal{N}_{\mathbf{V},i}$.  
 \end{definition}

 Consider the open embedding $j_{\mathbf{V},i}: \mathbf{E}_{\mathbf{V},\mathbf{W}^{\bullet},i}^{0} \rightarrow \mathbf{E}_{\mathbf{V},\mathbf{W}^{\bullet},\Omega^{N}_{i}}$. We recall that the middle extension functor for an open embedding $j: U \rightarrow X$ 
 \begin{equation*}
 	j_{!\ast}:Perv(U) \rightarrow Perv(X)
 \end{equation*}
 can be naturally extended to semisimple complexes. (The middle extension is not a functor for the triangulated category, since it can not be extended to  morphisms between semisimple complexes.) More precisely, if $L=\bigoplus\limits_{K \in Perv(U)} K[n]$, we can set
 \begin{equation*}
 	j_{!\ast}(L)=\bigoplus\limits_{K \in Perv(U)}j_{!\ast}(K)[n].
 \end{equation*}

 Note that $(j_{\mathbf{V},i})_{!\ast} (L) \cong (j_{\mathbf{V},i})_{!} (L)$ holds in $\mathcal{D}^{b}_{G_{\mathbf{V}}}(\mathbf{E}_{\mathbf{V},\mathbf{W}^{\bullet},\Omega^{N}_{i}})/\mathcal{N}_{\mathbf{V},i}$ for any semisimple complex $L$, so by \cite[Lemma 3.4]{fang2023lusztig} the functor  $(j_{\mathbf{V},i})_{!}$  and $(j_{\mathbf{V},i})^{\ast}$ define quasi-inverse equivalences 
 \[
 \xymatrix{
 	(j_{\mathbf{V},i})^{\ast} (\mathcal{Q}_{\mathbf{V},\mathbf{W}^{\bullet}}) \ar@<0.5ex>[r]^{(j_{\mathbf{V},i})_{!}} & \mathcal{Q}_{\mathbf{V},\mathbf{W}^{\bullet}}/\mathcal{N}_{\mathbf{V},i} \ar@<0.5ex>[l]^{(j_{\mathbf{V},i})^{\ast}}
 },
 \]
 \[
 \xymatrix{
 		(j_{\mathbf{V},i})^{\ast}(\mathcal{D}^{b,ss}_{G_{\mathbf{V}}}(\mathbf{E}^{0}_{\mathbf{V},\mathbf{W}^{\bullet},i})) \ar@<0.5ex>[r]^{(j_{\mathbf{V},i})_{!}} & \mathcal{D}^{b,ss}_{G_{\mathbf{V}}}(\mathbf{E}_{\mathbf{V},\mathbf{W}^{\bullet},\Omega^{N}_{i}})/ \mathcal{N}_{\mathbf{V},i} \ar@<0.5ex>[l]^{(j_{\mathbf{V},i})^{\ast}}
 }.
 \]

 \subsection{Functors for localization at $i$}
 
 For any $n\in \mathbb{N}$, take graded spaces $\mathbf{V}, \mathbf{V}'$ of dimension vector $\nu,\nu'$ respectively, such that $\nu'+ni=\nu$, we will define varieties and morphisms appearing in following diagram, and then define a functor $\mathcal{E}_{i}^{(n)}$.
 
 \[
 \xymatrix{
 	\mathbf{E}_{\mathbf{V},\mathbf{W}^{\bullet},\Omega^{N}_{i}} 
 	&
 	& \mathbf{E}_{\mathbf{V}',\mathbf{W}^{\bullet},\Omega^{N}_{i}} \\	
 	\mathbf{E}^{0}_{\mathbf{V},\mathbf{W}^{\bullet},i} \ar[d]_{\phi_{\mathbf{V},i}} \ar[u]^{j_{\mathbf{V},i}}
 	&
 	& \mathbf{E}^{0}_{\mathbf{V}',\mathbf{W}^{\bullet},i} \ar[d]^{\phi_{\mathbf{V}',i}} \ar[u]_{j_{\mathbf{V}',i}} \\
 	\dot{\mathbf{E}}_{\mathbf{V},\mathbf{W}^{\bullet},i} \times \mathbf{Grass}(\nu_{i}, \tilde{\nu}_{i})
 	& \dot{\mathbf{E}}_{\mathbf{V},\mathbf{W}^{\bullet},i} \times \mathbf{Flag}(\nu_{i}-n,\nu_{i},\tilde{\nu}_{i}) \ar[r]^{q_{2}} \ar[l]_{q_{1}}
 	& \dot{\mathbf{E}}_{\mathbf{V},\mathbf{W}^{\bullet},i} \times \mathbf{Grass}(\nu_{i}-n, \tilde{\nu}_{i}).
 }
 \]
 
 Define the affine space $\dot{\mathbf{E}}_{\mathbf{V},\mathbf{W}^{\bullet},i}$ by
 \begin{equation*}
 	\dot{\mathbf{E}}_{\mathbf{V},\mathbf{W}^{\bullet},i} =\bigoplus\limits_{h \in \Omega_{i}, h'\neq i} \mathbf{Hom}(\mathbf{V}_{h'},\mathbf{V}_{h''})\oplus\bigoplus\limits_{j \in I,j \neq i, 1\leqslant k \leqslant N} \mathbf{Hom}(\mathbf{V}_{j},\mathbf{W}_{j^{k}}) .
 \end{equation*}
 For given $x \in \mathbf{E}_{\mathbf{V},\mathbf{W}^{\bullet},\Omega^{N}_{i}} $, we denote $(x_{h})_{h\in \Omega^{N}_{i},h' \neq i}$ by $\dot{x}$. For any graded space $\mathbf{V}$, we denote $\bigoplus\limits_{j\in I,j\neq i} \mathbf{V}_{j}$ by $\dot{\mathbf{V}}$. Then the morphism $\phi_{\mathbf{V},i}:\mathbf{E}^{0}_{\mathbf{V},\mathbf{W}^{\bullet},i} \rightarrow  \dot{\mathbf{E}}_{\mathbf{V},\mathbf{W}^{\bullet},i} \times \mathbf{Grass}(\nu_{i}, \tilde{\nu}_{i})$ is defined by:
 \begin{equation*}
 	\phi_{\mathbf{V},i}(x)=(\dot{x}, {\rm{Im}}  (\bigoplus \limits_{h \in \Omega^{N}_{i}, h'=i} x_{h}) )
 \end{equation*}
 where $\nu_{i}={{\rm{dim}}}\mathbf{V}_{i}$, $\tilde{\nu}_{i}=\sum \limits_{h'=i,h \in \Omega}{{\rm{dim}}}\mathbf{V}_{h''}+\sum\limits_{1\leqslant k \leqslant N}{{\rm{dim}}}\mathbf{W}_{i^{k}}$, and $\mathbf{Grass}(\nu_{i}, \tilde{\nu}_{i})$ is the Grassmannian consisting of $\nu_{i}$-dimensional subspaces of the $\tilde{\nu}_{i}$-dimensional space $(\bigoplus\limits_{h'=i}\mathbf{V}_{h''})\oplus \bigoplus\limits_{1\leqslant k \leqslant N}\mathbf{W}_{i^{k}}$. Note that $\phi_{\mathbf{V},i}$ is a principal $\mathbf{GL}(\mathbf{V}_{i})$-bundle. The morphisms $q_{1},q_{2}$ are projections defined by
 \begin{equation*}
 	q_{1}(\dot{x}, \mathbf{U}_{1},\mathbf{U}_{2})=(\dot{x},\mathbf{U}_{2}),
 \end{equation*}
 \begin{equation*}
 	q_{2}(\dot{x},\mathbf{U}_{1},\mathbf{U}_{2})=(\dot{x},\mathbf{U}_{1}).
 \end{equation*}

 \begin{definition}
 	Take graded spaces $\mathbf{V}$ and $\mathbf{V}'$ such that $|\mathbf{V}'|+ni=|\mathbf{V}|$, we define the functor $\mathcal{E}^{(n)}_{i}: \mathcal{D}^{b}_{G_{\mathbf{V}}}(\mathbf{E}_{\mathbf{V},\mathbf{W}^{\bullet},\Omega^{N}_{i}})/ \mathcal{N}_{\mathbf{V},i} \rightarrow  \mathcal{D}^{b}_{G_{\mathbf{V}'}}(\mathbf{E}_{\mathbf{V}',\mathbf{W}^{\bullet},\Omega^{N}_{i}})/ \mathcal{N}_{\mathbf{V}',i} $ by
 	\begin{equation*}
 		\mathcal{E}^{(n)}_{i}= (j_{\mathbf{V}',i})_{!} (\phi_{\mathbf{V}',i})^{\ast} (q_{2})_{!}(q_{1})^{\ast} (\phi_{\mathbf{V},i})_{\flat}(j_{\mathbf{V},i})^{\ast}[-n\nu_{i}].
 	\end{equation*}
 	By definition,  $\mathcal{E}^{(n)}_{i}$ sends semisimple complexes to semisimple complexes up to isomorphisms in localization. We also denote $\mathcal{E}^{(1)}_{i}$ by $\mathcal{E}_{i}$.
 \end{definition} 
 
 \begin{remark}
 	It's easy to see that if $\tilde{\Omega}_{i}$ is an another orientation such that $i$ is a source, then the Fourier-Deligne transform $F_{\tilde{\Omega}_{i},\Omega_{i} }$ only involves those arrows in $\dot{\mathbf{E}}_{\mathbf{V},\mathbf{W}^{\bullet},i}$. Hence we can  prove that $\mathcal{E}^{(n)}_{i}$ is independent of the choice of $\Omega_{i}$ by base change.
 \end{remark}

 \begin{definition}
 (1)	For $j \in I $ and graded spaces $\mathbf{V},\mathbf{V'}$ and $\mathbf{V''}$ such that $|
 	\mathbf{V}|=|\mathbf{V}''|+nj$ and $|\mathbf{V'}|=nj$, we define the functor  $\mathcal{F}^{(n)}_{j}: \mathcal{D}^{b}_{G_{\mathbf{V}''}}(\mathbf{E}_{\mathbf{V}'',\mathbf{W}^{\bullet},\Omega^{N}_{i}})\rightarrow \mathcal{D}^{b}_{G_{\mathbf{V}}}(\mathbf{E}_{\mathbf{V},\mathbf{W}^{\bullet},\Omega^{N}_{i}})$ by
 	\begin{equation*}
 		\mathcal{F}^{(n)}_{j}(-)=  \mathbf{Ind}^{\mathbf{V}\oplus\mathbf{W}^{\bullet}}_{\mathbf{V'},\mathbf{V}''\oplus\mathbf{W}^{\bullet}} (\overline{\mathbb{Q}}_{l} \boxtimes -),
 	\end{equation*}
 	where $\overline{\mathbb{Q}}_{l} $ is the constant sheaf on $\mathbf{E}_{\mathbf{V}',0,\Omega^{N}_{i}}$. In particular, we denote $\mathcal{F}_{j}^{(1)}$ by $\mathcal{F}_{j}$.

(2)	For $j \in I $ and $1\leqslant k \leqslant N$, take $\mathbf{W}^{\bullet}$,${^{'}\mathbf{W}}^{\bullet}$ and ${^{''}\mathbf{W}}^{\bullet}$ such that  $|{^{'}\mathbf{W}}^{\bullet}|=n j^{k}=\dim \mathbf{W}_{j^{k}} j^{k}$. In particular, ${^{''}\mathbf{W}}^{\bullet}$ is supported on $I^{(N)}\backslash \{ j^{k}\}$.  Define the functor  $\mathcal{F}^{(n)}_{j^{k}}: \mathcal{D}^{b}_{G_{\mathbf{V}}}(\mathbf{E}_{\mathbf{V},{^{''}\mathbf{W}}^{\bullet},\Omega^{N}_{i}})\rightarrow \mathcal{D}^{b}_{G_{\mathbf{V}}}(\mathbf{E}_{\mathbf{V},\mathbf{W}^{\bullet},\Omega^{N}_{i}})$ by
	\begin{equation*}
		\mathcal{F}^{(n)}_{j^{k}}(-)=  \mathbf{Ind}^{\mathbf{V}\oplus\mathbf{W}^{\bullet}}_{{^{'}\mathbf{W}}^{\bullet},\mathbf{V}\oplus{^{''}\mathbf{W}}^{\bullet}} (\overline{\mathbb{Q}}_{l} \boxtimes -),
	\end{equation*}
	where $\overline{\mathbb{Q}}_{l} $ is the constant sheaf on $\mathbf{E}_{0,{^{'}\mathbf{W}}^{\bullet},\Omega^{N}_{i}}$. 
\end{definition}

One may ask why the induction functors are well defined for $G_{\mathbf{V}}$-equivariant complexes. We explain it for (1) and the other one can be proved similarly. Recall that  when we define Lusztig's induction functor $\mathbf{Ind}^{\mathbf{V}\oplus\mathbf{W}^{\bullet}}_{\mathbf{V'},\mathbf{V}''\oplus\mathbf{W}^{\bullet}} (\overline{\mathbb{Q}}_{l} \boxtimes -)$ for $j \in I $ and garded spaces $\mathbf{V},\mathbf{V'}$ and $\mathbf{V''}$ such that $|
\mathbf{V}|=|\mathbf{V}''|+nj$ and $|\mathbf{V'}|=nj$, we need to consider the following diagram
\begin{center}
	$ \mathbf{E}_{\mathbf{V}'',{\mathbf{W}}^{\bullet},\Omega^{N}} \xleftarrow{p_{1}} \mathbf{E}'_{\Omega} \xrightarrow{p_{2}} \mathbf{E}''_{\Omega} \xrightarrow{p_{3}} \mathbf{E}_{\mathbf{V},\mathbf{W}^{\bullet},\Omega^{N}}$
\end{center}
where $\mathbf{E}'_{\Omega}$ is the variety consisting of $(x,\tilde{\mathbf{U}}, \rho_{1}, \rho_{2})$, where $x \in \mathbf{E}_{\mathbf{V},\mathbf{W}^{\bullet},\Omega^{N}}$, $\tilde{\mathbf{U}}$ is an $I^{(N)}$-graded $x$-stable subspace of $\mathbf{V}\oplus \mathbf{W}^{\bullet}$ of dimension vector $|\mathbf{V}''\oplus {\mathbf{W}}^{\bullet}|$ and $ \rho_{1}: \mathbf{V}\oplus \mathbf{W}^{\bullet}/\tilde{\mathbf{U}} \simeq \mathbf{V}',\rho_{2}:\tilde{\mathbf{U}} \simeq \mathbf{V}''\oplus {\mathbf{W}}^{\bullet}$ are linear isomorphisms. 
Now we consider the subspace $\mathbf{E}^{\diamond}_{\Omega}$ of $\mathbf{E}'_{\Omega}$  consisting of those $(x,\tilde{\mathbf{U}}, \rho_{1}, \rho_{2})$ such that $\rho_{2}$ restricts to identity on $\mathbf{W}^{\bullet}$, then there are commutative diagrams

\[
\xymatrix{
	&  \mathbf{E}'_{\Omega}  \ar[d]^{p_{2}} \ar[ld]_{\rho}
	\\
	\mathbf{E}^{\diamond}_{\Omega}  \ar[r]^{p^{\diamond}_{2}}
	&  \mathbf{E}''_{\Omega}  ,
}
\]

\[
\xymatrix{
	\mathbf{E}_{\mathbf{V}'',{\mathbf{W}}^{\bullet},\Omega^{N}}  
	&  \mathbf{E}'_{\Omega} \ar[l]_-{p_{1}}  
	\\
	\mathbf{E}^{\diamond}_{\Omega} \ar[u]^{p^{\diamond}_{1}} \ar[ru]_{e}
	&   ,
}
\]
where $p^{\diamond}_{1}$ and $p^{\diamond}_{2}$ are restrictions of $p_{1}$ and $p_{2}$, $e$ is the inclusion and $\rho$ is the obvious $G_{\mathbf{W}^{\bullet}}$-principal bundle. Then $ (p^{\diamond}_{2})_{\flat}e^{\ast}(p^{\diamond}_{1})^{\ast}\cong (p^{\diamond}_{2})_{\flat}(p^{\diamond}_{1})^{\ast}$ and $(p_{2})_{\flat} (p_{1})^{\ast}= (p^{\diamond}_{2})_{\flat}{\rho}_{\flat}(p^{\diamond}_{1})^{\ast}$. However, $\rho e =id$ implies that $e^{\ast} \cong  {\rho}_{\flat}$. Since $(p_{3})_{!}(p^{\diamond}_{2})_{\flat}(p^{\diamond}_{1})^{\ast}$ defines a functor from $\mathcal{D}^{b}_{G_{\mathbf{V}''}}(\mathbf{E}_{\mathbf{V}'',\mathbf{W}^{\bullet},\Omega^{N}_{i}})$ to $ \mathcal{D}^{b}_{G_{\mathbf{V}}}(\mathbf{E}_{\mathbf{V},\mathbf{W}^{\bullet},\Omega^{N}_{i}})$, so does $\mathcal{F}^{(n)}_{j}$. Notice that objects of $\mathcal{Q}'_{\mathbf{V},\mathbf{W}^{\bullet}}$ coincide with objects of $\mathcal{Q}_{\mathbf{V},\mathbf{W}^{\bullet}}$, $\mathcal{F}^{(n)}_{j}$ also sends $\mathcal{Q}_{\mathbf{V}'',\mathbf{W}^{\bullet}} $ to $ \mathcal{Q}_{\mathbf{V},\mathbf{W}^{\bullet}}$.

\begin{lemma}\label{lemmaF}
	For any $j \in I$, the functor $\mathcal{F}^{(n)}_{j}$ (or $\mathcal{F}^{(n)}_{j^{m}}$ for some $1\leqslant m \leqslant N$ and $n=\dim \mathbf{W}_{j^{m}}$) sends objects of $ \mathcal{N}_{\mathbf{V}'',k}$ (or $ \mathcal{N}_{\mathbf{V},k}$ respectively) to objects of $ \mathcal{N}_{\mathbf{V},k}$ for any $k \in I$.
\end{lemma}
\begin{proof}
	Notice that the induction functor commutes with Fourier-Deligne transforms, we can assume that $k$ is a source, then  the simple representation $S_{k}$ of $Q^{(N)}$ at $k$ is injective. In particular,  any exact sequence of representations of $Q^{(N)}$ of the form
	$$ 0 \rightarrow S_{k}^{\oplus m} \rightarrow Y \rightarrow Z\rightarrow 0$$
	must be split, then by definition of $\mathcal{F}^{(n)}_{j}$, we can see that $\text{supp} (\mathcal{F}^{(n)}_{j}(L)) \subseteq \mathbf{E}^{\geqslant 1}_{\mathbf{V},\mathbf{W}^{\bullet},k}$ if  $L$ in $\mathcal{N}_{\mathbf{V}'',k}$. Hence the functor $\mathcal{F}^{(n)}_{j}$ sends objects of $ \mathcal{N}_{\mathbf{V}'',k}$ to objects of $ \mathcal{N}_{\mathbf{V},k}$. 
\end{proof}
\begin{corollary}
	For $j\in I$, the functor $\mathcal{F}^{(n)}_{j}$ induces a functor  $\mathcal{F}^{(n)}_{j}:\mathcal{D}^{b}_{G_{\mathbf{V}''}}(\mathbf{E}_{\mathbf{V}'',\mathbf{W}^{\bullet},\Omega^{N}_{i}})/\mathcal{N}_{\mathbf{V}'',i}\rightarrow \mathcal{D}^{b}_{G_{\mathbf{V}}}(\mathbf{E}_{\mathbf{V},\mathbf{W}^{\bullet},\Omega^{N}_{i}})/\mathcal{N}_{\mathbf{V},i}.$ For $j \in I $ and $1\leqslant k \leqslant N$ and $n=\dim \mathbf{W}_{j^{k}}$,  the functor $\mathcal{F}^{(n)}_{j^{k}}$ induces a functor $\mathcal{F}^{(n)}_{j^{k}}: \mathcal{D}^{b}_{G_{\mathbf{V}}}(\mathbf{E}_{\mathbf{V},{^{''}\mathbf{W}}^{\bullet},\Omega^{N}_{i}})/\mathcal{N}_{\mathbf{V},i}\rightarrow \mathcal{D}^{b}_{G_{\mathbf{V}}}(\mathbf{E}_{\mathbf{V},\mathbf{W}^{\bullet},\Omega^{N}_{i}})/\mathcal{N}_{\mathbf{V},i}.$
\end{corollary}

Now we can state the commutative relation of the functors $\mathcal{E}^{(n)}_{i}$ and $\mathcal{F}^{(m)}_{j},j \in I$ between localizations at $i$.
\begin{lemma} \label{lemma c1}
	For any graded spaces $\mathbf{V},\mathbf{V}'$ such that $|\mathbf{V}'|+(n-1)i=|\mathbf{V}| $, there is an isomorphism as functors between localizations $\mathcal{D}^{b}_{G_{\mathbf{V}}}(\mathbf{E}_{\mathbf{V},\mathbf{W}^{\bullet},\Omega_{i}^{N}})/\mathcal{N}_{\mathbf{V},i} \rightarrow \mathcal{D}^{b}_{G_{\mathbf{V}'}}(\mathbf{E}_{\mathbf{V}',\mathbf{W}^{\bullet},\Omega_{i}^{N}})/\mathcal{N}_{\mathbf{V}',i}$,
	\begin{equation*}
		\mathcal{E}^{(n)}_{i}\mathcal{F}_{i} \oplus \bigoplus\limits_{0\leqslant m \leqslant M-1} \mathcal{E}^{(n-1)}_{i}[M-1-2m] \cong \mathcal{F}_{i}\mathcal{E}^{(n)}_{i} \oplus \bigoplus\limits_{0\leqslant m \leqslant -M-1} \mathcal{E}^{(n-1)}_{i}[-2m-M-1],
	\end{equation*}
	where $M=2\nu_{i}-\tilde{\nu}_{i}-n+1$. More precisely,
	\begin{align*}
		&\mathcal{E}^{(n)}_{i}\mathcal{F}_{i}\cong \mathcal{F}_{i}\mathcal{E}^{(n)}_{i}, &\textrm{if}\ M=0,\\
		&\mathcal{E}^{(n)}_{i}\mathcal{F}_{i}\oplus\bigoplus_{0\leqslant m\leqslant M-1}\mathcal{E}^{(n-1)}_{i}[M-1-2m]\cong \mathcal{F}_{i}\mathcal{E}^{(n)}_{i}, &\textrm{if}\  M\geqslant 1,\\
		&\mathcal{E}^{(n)}_{i}\mathcal{F}_{i}\cong \mathcal{F}_{i}\mathcal{E}^{(n)}_{i} \oplus \bigoplus\limits_{0\leqslant m \leqslant -M-1} \mathcal{E}^{(n-1)}_{i}[-2m-M-1], &\textrm{if}\ M\leqslant -1.
	\end{align*}
\end{lemma}
\begin{proof}
	The proof is similar to \cite[Lemma 3.13]{fang2023lusztig}.
\end{proof}

\begin{lemma} \label{lemma c2}
	For any graded spaces $\mathbf{V},\mathbf{V}'$ such that $|\mathbf{V}'|+ni=|\mathbf{V}|+mj $ and $i \neq j \in I$, there is an isomorphism as functors between localizations $$\mathcal{D}^{b}_{G_{\mathbf{V}}}(\mathbf{E}_{\mathbf{V},\mathbf{W}^{\bullet},\Omega_{i}^{N}})/\mathcal{N}_{\mathbf{V},i} \rightarrow \mathcal{D}^{b}_{G_{\mathbf{V}'}}(\mathbf{E}_{\mathbf{V}',\mathbf{W}^{\bullet},\Omega_{i}^{N}})/\mathcal{N}_{\mathbf{V}',i},$$
	\begin{equation*}
		\mathcal{E}^{(n)}_{i}\mathcal{F}^{(m)}_{j} \cong \mathcal{F}^{(m)}_{j}\mathcal{E}^{(n)}_{i}. 
	\end{equation*}
\end{lemma}
\begin{proof}
	The proof is similar to \cite[Lemma 3.14]{fang2023lusztig}.
\end{proof}

Viewing $j^{k}$ as an unframed vertex, we also have the following result by a similar argument as  Lemma \ref{lemma c2}.
\begin{lemma} \label{lemma c3}
	For any graded spaces $\mathbf{V},\mathbf{V}'$ such that $|\mathbf{V}'|+ni=|\mathbf{V}| $ and fixed graded spaces $\mathbf{W}^{\bullet}$,${^{'}\mathbf{W}}^{\bullet}$ and ${^{''}\mathbf{W}}^{\bullet}$ such that  $|{^{'}\mathbf{W}}^{\bullet}|=n j^{k}=\dim \mathbf{W}_{j^{k}}j^{k}$ and $|{^{''}\mathbf{W}}^{\bullet}|+|{^{'}\mathbf{W}}^{\bullet}|=|\mathbf{W}^{\bullet}| $, 
	 there is an isomorphism as functors between localizations $\mathcal{D}^{b}_{G_{\mathbf{V}}}(\mathbf{E}_{\mathbf{V},{^{''}\mathbf{W}}^{\bullet},\Omega_{i}^{N}})/\mathcal{N}_{\mathbf{V},i} \rightarrow \mathcal{D}^{b}_{G_{\mathbf{V}'}}(\mathbf{E}_{\mathbf{V}',\mathbf{W}^{\bullet},\Omega_{i}^{N}})/\mathcal{N}_{\mathbf{V}',i}$,
	\begin{equation*}
		\mathcal{E}^{(n)}_{i}\mathcal{F}^{(m)}_{j^{k}} \cong \mathcal{F}^{(m)}_{j^{k}}\mathcal{E}^{(n)}_{i}. 
	\end{equation*}
\end{lemma}

The following lemma shows that $\mathcal{E}^{(n)}_{i}$ and $\mathcal{F}^{(n)}_{j}$ satisfy relations of divided powers, and can be proved in a similar way as \cite[Lemma 3.15]{fang2023lusztig}.
\begin{lemma}\label{lemma c4}
	As endofunctors of the localization $\coprod\limits_{\mathbf{V}}\mathcal{D}^{b}_{G_{\mathbf{V}}}(\mathbf{E}_{\mathbf{V},\mathbf{W}^{\bullet},\Omega^{N}_{i}} )/\mathcal{N}_{\mathbf{V},i}$, the functors $\mathcal{E}^{(n)}_{i}$ and $ \mathcal{F}^{(n)}_{j}$ with $j \in I, n \geqslant 1$ satisfy the following relation
	\begin{equation*}
		\bigoplus \limits_{0 \leqslant m < n } \mathcal{E}^{(n)}_{i}[n-1-2m] \cong \mathcal{E}^{(n-1)}_{i}\mathcal{E}_{i}, n \geqslant 2,
	\end{equation*}
	\begin{equation*}
		\bigoplus \limits_{0 \leqslant m < n } \mathcal{F}^{(n)}_{j}[n-1-2m] \cong \mathcal{F}^{(n-1)}_{j}\mathcal{F}_{j}, n \geqslant 2.
	\end{equation*}
\end{lemma}

The following lemma will be used in the next section to define functors between global localizations.
\begin{lemma}\label{lemmaE}
	The functor $\mathcal{E}^{(n)}_{i}:\mathcal{D}^{b}_{G_{\mathbf{V}}}(\mathbf{E}_{\mathbf{V},\mathbf{W}^{\bullet},\Omega_{i}^{N}})/\mathcal{N}_{\mathbf{V},i} \rightarrow \mathcal{D}^{b}_{G_{\mathbf{V}'}}(\mathbf{E}_{\mathbf{V}',\mathbf{W}^{\bullet},\Omega_{i}^{N}})/\mathcal{N}_{\mathbf{V}',i}$ sends objects of $ \mathcal{N}_{\mathbf{V},k}$ to objects isomorphic to those of $ \mathcal{N}_{\mathbf{V}',k}$ for any $k \in I$.
\end{lemma}
\begin{proof}
	The proof is similar to \cite[Proposition 3.18]{fang2023lusztig}.
\end{proof}

 \subsection{Global localizations and their functors}
 
 Fix an orientation $\Omega^{N}_{i}$ for each $i\in I$ such that $i$ is a source in $\Omega^{N}_{i}$. For a given orientation $\Omega^{N}$, let $\mathcal{N}_{\mathbf{V}}$ be the thick subcategory generated by $\mathcal{F}_{\Omega^{N}_{i},\Omega^{N}}(\mathcal{N}_{\mathbf{V},i}), i\in I$, then one can consider the Verdier quotient $\mathcal{D}^{b}_{G_{\mathbf{V}}}(\mathbf{E}_{\mathbf{V},\mathbf{W}^{\bullet}},\Omega^{N})/\mathcal{N}_{\mathbf{V}}$. It is also a triangulated category and admits a $t$-structure induced from the perverse $t$-structure of $\mathcal{D}^{b}_{G_{\mathbf{V}}}(\mathbf{E}_{\mathbf{V},\mathbf{W}^{\bullet}},\Omega^{N})$. In this subsection, we omit the Fourier-Deligne tarnsformation and simply denote $\mathcal{F}_{\Omega^{N}_{i},\Omega^{N}}(\mathcal{N}_{\mathbf{V},i})$ by $\mathcal{N}_{\mathbf{V},i} $.
 
 \begin{definition}	
 	(1) Define the global localization $\mathcal{L}_{\mathbf{V}}(\Lambda^{\bullet}) =\mathcal{Q}_{\mathbf{V},\mathbf{W}^{\bullet}}/\mathcal{N}_{\mathbf{V}}$ of  $\mathcal{Q}_{\mathbf{V},\mathbf{W}^{\bullet}}$  to be the full subcategory of the Verdier quotient $\mathcal{D}^{b}_{G_{\mathbf{V}}}(\mathbf{E}_{\mathbf{V},\mathbf{W}^{\bullet},\Omega^{N}})/ \mathcal{N}_{\mathbf{V}}$ consisting of objects which are isomorphic to some objects of $\mathcal{Q}_{\mathbf{V},\mathbf{W}^{\bullet}}$ in $\mathcal{D}^{b}_{G_{\mathbf{V}}}(\mathbf{E}_{\mathbf{V},\mathbf{W}^{\bullet},\Omega^{N}})/ \mathcal{N}_{\mathbf{V}}$.
 	
 	(2) Similarly, define the global localization $\mathcal{D}^{b,ss}_{G_{\mathbf{V}}}(\mathbf{E}_{\mathbf{V},\mathbf{W}^{\bullet},\Omega^{N}})/ \mathcal{N}_{\mathbf{V}}$ of $\mathcal{D}^{b,ss}_{G_{\mathbf{V}}}(\mathbf{E}_{\mathbf{V},\mathbf{W}^{\bullet},\Omega^{N}})$ to be the  full subcategory consisting of objects which are isomorphic to some semisimple complexes in $\mathcal{D}^{b}_{G_{\mathbf{V}}}(\mathbf{E}_{\mathbf{V},\mathbf{W}^{\bullet},\Omega^{N}})/ \mathcal{N}_{\mathbf{V}}$. 
 \end{definition}

 By Lemma \ref{lemmaF} and \ref{lemmaE}, if $i$ is a source, $\mathcal{E}^{(n)}_{i}$ and $\mathcal{F}^{(n)}_{j},j \in I$ induces functors between global localizations
 $$ \mathcal{E}^{(n)}_{i}: \mathcal{D}^{b}_{G_{\mathbf{V}}}(\mathbf{E}_{\mathbf{V},\mathbf{W}^{\bullet},\Omega^{N}})/\mathcal{N}_{\mathbf{V}} \rightarrow \mathcal{D}^{b}_{G_{\mathbf{V}'}}(\mathbf{E}_{\mathbf{V}',\mathbf{W}^{\bullet},\Omega^{N}})/\mathcal{N}_{\mathbf{V}'}, $$
 $$ \mathcal{F}^{(n)}_{j}: \mathcal{D}^{b}_{G_{\mathbf{V}''}}(\mathbf{E}_{\mathbf{V}'',\mathbf{W}^{\bullet},\Omega^{N}})/\mathcal{N}_{\mathbf{V}''} \rightarrow \mathcal{D}^{b}_{G_{\mathbf{V}}}(\mathbf{E}_{\mathbf{V},\mathbf{W}^{\bullet},\Omega^{N}})/\mathcal{N}_{\mathbf{V}}. $$

By Proposition \ref{indformula}, the functor $\mathcal{F}^{(n)}_{j}: \mathcal{D}^{b}_{G_{\mathbf{V}''}}(\mathbf{E}_{\mathbf{V}'',\mathbf{W}^{\bullet},\Omega^{N}})/\mathcal{N}_{\mathbf{V}''} \rightarrow \mathcal{D}^{b}_{G_{\mathbf{V}}}(\mathbf{E}_{\mathbf{V},\mathbf{W}^{\bullet},\Omega^{N}})/\mathcal{N}_{\mathbf{V}}$ also sends objects of $\mathcal{Q}_{\mathbf{V}'',\mathbf{W}^{\bullet}}$ to those of $\mathcal{Q}_{\mathbf{V},\mathbf{W}^{\bullet}}$, hence restricts to a functor  $\mathcal{F}^{(n)}_{j}:\mathcal{Q}_{\mathbf{V}'',\mathbf{W}^{\bullet}}/ \mathcal{N}_{\mathbf{V}''} \rightarrow \mathcal{Q}_{\mathbf{V},\mathbf{W}^{\bullet}}/ \mathcal{N}_{\mathbf{V}} $. The following proposition shows that the functor $\mathcal{E}^{(n)}_{i}$ has the same property.

\begin{proposition}
	The functor $\mathcal{E}^{(n)}_{i}$   restricts to a functor  $\mathcal{E}^{(n)}_{i}:\mathcal{Q}_{\mathbf{V},\mathbf{W}^{\bullet}}/ \mathcal{N}_{\mathbf{V}} \rightarrow \mathcal{Q}_{\mathbf{V}',\mathbf{W}^{\bullet}}/ \mathcal{N}_{\mathbf{V}'} $.
\end{proposition}

 \begin{proof}
 	We prove that $\mathcal{E}^{(n)}_{i}$ sends objects of $\mathcal{Q}_{\mathbf{V},\mathbf{W}^{\bullet}}$ to those of $\mathcal{Q}_{\mathbf{V}',\mathbf{W}^{\bullet}}$ up to isomorphisms in global localizations by induction on $|\mathbf{V}|$ and $n$.
 	It suffices to show that the functor $\mathcal{E}^{(n)}_{i}$ sends semisimple complex $L_{\boldsymbol{\nu}^{1}\boldsymbol{d}^{1}\cdots\boldsymbol{\nu}^{N}\boldsymbol{d}^{N}}$ to a direct summand of direct sums of similar semisimple complexes  up to isomorphisms in global localizations.
 	
 	(1)If $|\mathbf{V}|=0$, then every $\boldsymbol{\nu}^{k}$ is empty and  $\mathcal{E}^{(n)}_{i}(L_{\boldsymbol{d}^{1}\cdots\boldsymbol{d}^{N}})=0$, the statement holds trivially.
 	
 	(2)Now we assume that for any $|\tilde{\mathbf{V}}|<|\mathbf{V}|$, we have proved that $\mathcal{E}^{(n)}_{i}$ sends objects of $\mathcal{Q}_{\tilde{\mathbf{V}},\mathbf{W}^{\bullet}}$ to those of $\mathcal{Q}_{\tilde{\mathbf{V}}',\mathbf{W}^{\bullet}}$  up to isomorphisms in global localizations.
 	Now we take $L_{\boldsymbol{\nu}^{1}\boldsymbol{d}^{1}\cdots\boldsymbol{\nu}^{N}\boldsymbol{d}^{N}}$ in $\mathcal{Q}_{\mathbf{V},\mathbf{W}^{\bullet}}$,  we need to show that $\mathcal{E}^{(n)}_{i}(L_{\boldsymbol{\nu}^{1}\boldsymbol{d}^{1}\cdots\boldsymbol{\nu}^{N}\boldsymbol{d}^{N}})$ still belongs to $\mathcal{Q}_{\mathbf{V}',\mathbf{W}^{\bullet}}$  up to isomorphisms in global localizations.
 	
 	Assume $\boldsymbol{\nu}^{k}$ is empty for $k < l$ and  $\boldsymbol{\nu}^{l}=((i_{1})^{a_{1}},(i_{2})^{a_{2}},\cdots (i_{p})^{a_{p}})$ is a nonempty sequence,  let  $\boldsymbol{\nu}'= (i_{1},\cdots,i_{1},i_{2},\cdots,i_{2},\cdots,i_{p},\cdots i_{p})$ be the flag type such that $i_{q}$ appears for $a_{q}$ times, $1\leqslant q \leqslant p$, then $L_{\boldsymbol{d}^{1}\boldsymbol{d}^{2}\cdots\boldsymbol{d}^{l-1}\boldsymbol{\nu}'\boldsymbol{d}^{l} \cdots  \boldsymbol{\nu}^{N}\boldsymbol{d}^{N}}$ is a direct sum of shifts of $L_{\boldsymbol{d}^{1}\boldsymbol{d}^{2}\cdots\boldsymbol{d}^{l-1}\boldsymbol{\nu}^{l}\boldsymbol{d}^{l}\cdots\boldsymbol{\nu}^{N}\boldsymbol{d}^{N}}$. Hence it suffices to show that the statement holds for $L_{\boldsymbol{d}^{1}\boldsymbol{d}^{2}\cdots\boldsymbol{d}^{l-1}\boldsymbol{\nu}'\boldsymbol{d}^{l} \cdots  \boldsymbol{\nu}^{N}\boldsymbol{d}^{N}}$.
 	
 	Set $\mathbf{W}^{',\bullet}=(\mathbf{W}^{',1},\cdots, \mathbf{W}^{',N})$ such that $\mathbf{W}^{',k}=0 $ for $1\leqslant k < l$ and $\mathbf{W}^{',k}=\mathbf{W}^{k}$ for $k \geqslant l$.  We consider $\boldsymbol{\nu}''= (i_{1}\cdots,i_{1},i_{2},\cdots,i_{2},\cdots,i_{p},\cdots i_{p})$, where $i_{1}$ apears for $ a_{1}-1$ times and the other $i_{q}$ apears for $a_{q}$ times, then we have $\boldsymbol{\nu}'=(i_{1},\boldsymbol{\nu}'')$ and $L_{\boldsymbol{\nu}'\boldsymbol{d}^{l} \cdots  \boldsymbol{\nu}^{N}\boldsymbol{d}^{N}}=\mathcal{F}_{i_{1}}L_{\boldsymbol{\nu}''\boldsymbol{d}^{l} \cdots  \boldsymbol{\nu}^{N}\boldsymbol{d}^{N}}$ belongs to $\mathcal{Q}_{\mathbf{V},\mathbf{W}^{',\bullet}}$.

 	If $i \neq i_{1}$, then by Lemma \ref{lemma c2}, we have the following equation in localization at $i$ (up to shifts)
 	\begin{equation*}
 		\begin{split}
 			\mathcal{E}^{(n)}_{i} L_{\boldsymbol{\nu}'\boldsymbol{d}^{l} \cdots  \boldsymbol{\nu}^{N}\boldsymbol{d}^{N}} \cong & \mathcal{E}^{(n)}_{i}\mathcal{F}_{i_{1}}L_{\boldsymbol{\nu}''\boldsymbol{d}^{l} \cdots \boldsymbol{\nu}^{N}\boldsymbol{d}^{N}}\\
 			\cong &  \mathcal{F}_{i_{1}}\mathcal{E}^{(n)}_{i}L_{\boldsymbol{\nu}''\boldsymbol{d}^{l} \cdots \boldsymbol{\nu}^{N}\boldsymbol{d}^{N}}.
 		\end{split}
 	\end{equation*}
 	
 	By the inductive assumption, we have $\mathcal{E}^{(n)}_{i}L_{\boldsymbol{\nu}''\boldsymbol{d}^{l} \cdots  \boldsymbol{\nu}^{N}\boldsymbol{d}^{N}} =K$ belongs to $\mathcal{Q}_{\mathbf{V}'',\mathbf{W}^{',\bullet}}$ for some $\mathbf{V}''$. Since $\mathcal{F}_{i_{1}}$  preserves objects of  $\coprod\limits_{\mathbf{V}'}\mathcal{Q}_{\mathbf{V}',\mathbf{W}^{',\bullet}}$, by Lemma \ref{lemma c3} we can see that $\mathcal{E}^{(n)}_{i} L_{\boldsymbol{d}^{1}\boldsymbol{d}^{2}\cdots\boldsymbol{d}^{l-1}\boldsymbol{\nu}'\boldsymbol{d}^{l} \cdots  \boldsymbol{\nu}^{N}\boldsymbol{d}^{N}}$ belongs to $\mathcal{Q}_{\mathbf{V}',\mathbf{W}^{\bullet}}$.
 	
 	Otherwise $i=i_{1}$, by Lemma \ref{lemma c1} and \ref{lemma c3} we can see that  $\mathcal{E}^{(n)}_{i}\mathcal{F}_{i_{1}}L_{\boldsymbol{\nu}''\boldsymbol{d}^{l} \cdots  \boldsymbol{\nu}^{N}\boldsymbol{d}^{N}}$ and  $ \mathcal{F}_{i_{1}}\mathcal{E}^{(n)}_{i}L_{\boldsymbol{\nu}''\boldsymbol{d}^{l} \cdots  \boldsymbol{\nu}^{N}\boldsymbol{d}^{N}}$ differ by direct sums of shifts of $\mathcal{E}^{(n-1)}_{i}L_{\boldsymbol{\nu}''\boldsymbol{d}^{l} \cdots  \boldsymbol{\nu}^{N}\boldsymbol{d}^{N}}$, which  belongs to $\mathcal{Q}_{\mathbf{V}'',\mathbf{W}^{',\bullet}}$ for some $\mathbf{V}''$ by inductive assumption. Hence similarly, $\mathcal{E}^{(n)}_{i}(L_{\boldsymbol{d}^{1}\boldsymbol{d}^{2}\cdots\boldsymbol{d}^{l-1}\boldsymbol{\nu}'\boldsymbol{d}^{l} \cdots  \boldsymbol{\nu}^{N}\boldsymbol{d}^{N}})$ belongs to $\mathcal{Q}_{\mathbf{V}',\mathbf{W}^{\bullet}}$  up to isomorphisms in global localizations.
 \end{proof}

 Now we can define the functors between global localizations.
 \begin{definition}
 	For graded spaces $\mathbf{V},\mathbf{V}'$ such that $|\mathbf{V}|=|\mathbf{V}'|+ni$, the functor $E^{(n)}_{i}:\mathcal{L}_{\mathbf{V}}(\Lambda^{\bullet}) \rightarrow \mathcal{L}_{\mathbf{V}'}(\Lambda^{\bullet})$ is defined by
 	\begin{equation*}
 		E^{(n)}_{i}=\mathcal{F}_{\Omega^{N}_{i},\Omega^{N}}\mathcal{E}^{(n)}_{i}\mathcal{F}_{\Omega^{N},\Omega^{N}_{i}}.
 	\end{equation*}
 	For graded spaces $\mathbf{V},\mathbf{V}''$ such that $|\mathbf{V}|=|\mathbf{V}''|+ni$, the functor $F^{(n)}_{i}:\mathcal{L}_{\mathbf{V}''}(\Lambda^{\bullet}) \rightarrow \mathcal{L}_{\mathbf{V}}(\Lambda^{\bullet})$ is defined by
 	\begin{equation*}
 		F^{(n)}_{i}= \mathcal{F}^{(n)}_{i}. 
 	\end{equation*}
 	The invertible functor $K_{i}: \mathcal{L}_{\mathbf{V}}(\Lambda^{\bullet}) \rightarrow \mathcal{L}_{\mathbf{V}}(\Lambda^{\bullet})$ is defiend by
 	\begin{equation*}
 		K_{i}=Id[\tilde{\nu}_{i}-2\nu_{i}],
 	\end{equation*}
 	and its inverse is defined by
 	\begin{equation*}
 		K^{-}_{i}=Id[2\nu_{i}-\tilde{\nu}_{i}],
 	\end{equation*}
 	where $\tilde{\nu}_{i}=\sum \limits_{h'=i,h \in \Omega_{i}}{{\rm{dim}}}\mathbf{V}_{h''}+\sum\limits_{1\leqslant k \leqslant N}{{\rm{dim}}}\mathbf{W}_{i^{k}}.$
 \end{definition}
 
 \begin{remark}
 	From the definition and  Corollary \ref{FD2}, we can see that the definitions of $E_{i},F_{i},K^{\pm}_{i}$ does not depend on the choices of $\Omega^{N},\Omega^{N}_{i},i\in I$.
 \end{remark}

 \begin{definition}
 	Let $\mathcal{K}_{0}(\Lambda^{\bullet})=\mathcal{K}_{0}(\mathcal{L}(\Lambda^{\bullet}))$ be the Grothendieck group of $\mathcal{L}(\Lambda^{\bullet})= \coprod\limits_{ \mathbf{V}} \mathcal{L}_{\mathbf{V}}(\Lambda^{\bullet})$. More precisely, $\mathcal{K}_{0}(\Lambda^{\bullet})$ is the $\mathbb{Z}[v,v^{-1}]$-module spanned by objects $[L]$ in $\mathcal{L}(\Lambda^{\bullet})$, modulo relations:
 	\begin{equation*}
 		[X \oplus Y]=[X]+[Y],
 	\end{equation*}
 	\begin{equation*}
 		[X[1]]=v[X].
 	\end{equation*}
 \end{definition}

 We also denote by $\mathcal{V}_{0}(\Lambda^{\bullet})_{\mathbf{V}}$  the Grothendieck group of $\mathcal{Q}_{\mathbf{V},\mathbf{W}^{\bullet}}$, and set $\mathcal{V}_{0}(\Lambda^{\bullet})= \bigoplus \limits_{\mathbf{V}} \mathcal{V}_{0}(\Lambda^{\bullet})_{\mathbf{V}}$. Then $\mathcal{K}_{0}(\Lambda^{\bullet})$ is exactly the $\mathbb{Z}[v,v^{-1}]$-submodule of  $\mathcal{V}_{0}(\Lambda^{\bullet})$ spanned by those $[L], L \in \bigcup\limits_{\mathbf{V}}\mathcal{P}_{\mathbf{V},\mathbf{W}^{\bullet}}$ but $L \notin  \coprod\limits_{\mathbf{V}}\mathcal{N}_{\mathbf{V}}$. Now we can state our first main theorm.
 \begin{theorem}\label{commutative relation}
(1)	As endofunctors of $\mathcal{L}(\Lambda^{\bullet})= \coprod \limits_{\mathbf{V}} \mathcal{L}_{\mathbf{V}}(\Lambda^{\bullet})$, the functors $E^{(n)}_{i},F^{(n)}_{i},K^{\pm}_{i},i \in I, n\in \mathbb{N}$ satisfy the following relations:
 	\begin{equation*} 
 		K_{i}K_{j}=K_{j}K_{i},
 	\end{equation*}
 	\begin{equation*}
 		E_{i}K_{j}=K_{j}E_{i}[-a_{j,i}],
 	\end{equation*}
 	\begin{equation*}
 		F_{i}K_{j}=K_{j}F_{i}[a_{i,j}],
 	\end{equation*}
 	\begin{equation*}
 		\bigoplus \limits_{0 \leqslant m < n } E^{(n)}_{i}[n-1-2m] \cong E^{(n-1)}_{i}E_{i}, n \geq 2,
 	\end{equation*}
 	\begin{equation*}
 		\bigoplus \limits_{0 \leqslant m < n } F^{(n)}_{j}[n-1-2m] \cong F^{(n-1)}_{j}F_{j}, n \geq 2.
 	\end{equation*}
 	\begin{equation*}
 		E_{i}F_{j}=F_{j}E_{i},i \neq j,
 	\end{equation*}
 	\begin{equation*}
 		E_{i}F_{i} \oplus \bigoplus\limits_{0\leqslant m \leqslant M-1} Id[M-1-2m] \cong F_{i}E_{i} \oplus \bigoplus\limits_{0\leqslant m \leqslant -M-1} Id[-2m-M-1]: \mathcal{L}_{\mathbf{V}}(\Lambda^{\bullet}) \rightarrow \mathcal{L}_{\mathbf{V}}(\Lambda^{\bullet})
 	\end{equation*}
 	where $M=2\nu_{i}-\tilde{\nu}_{i}$. 
 	
 	(2) We also have the following Serre relations
 	\begin{equation*}
 		\bigoplus\limits_{0\leqslant m \leqslant 1- a_{i,j},m~odd}F^{(m)}_{i}F_{j}F^{(1-a_{i,j}-m)}_{i} \cong 	\bigoplus\limits_{0\leqslant m \leqslant 1- a_{i,j},m~even}F^{(m)}_{i}F_{j}F^{(1-a_{i,j}-m)}_{i} ,
 	\end{equation*}
 	\begin{equation*}
 		\bigoplus\limits_{0\leqslant m \leqslant 1- a_{i,j},m~odd}E^{(m)}_{i}E_{j}E^{(1-a_{i,j}-m)}_{i}\cong 	\bigoplus\limits_{0\leqslant m \leqslant 1- a_{i,j},m~even}E^{(m)}_{i}E_{j}E^{(1-a_{i,j}-m)}_{i} .
 	\end{equation*}
 	
 	In particular, with the functors $E^{(n)}_{i},F^{(n)}_{i},K^{\pm}_{i},i \in I, n\in \mathbb{N}$, the Grothendieck group $\mathcal{K}_{0}(\Lambda^{\bullet})$ becomes an integrable ${_{\mathcal{A}}\mathbf{U}}$-module.
 \end{theorem}
 \begin{proof}
 	The first three relations follow by definition. The fourth and the fifth follow by Lemma \ref{lemma c4}. The last two relations follow by Lemma \ref{lemma c1} and Lemma \ref{lemma c2}. The Serre relation between $F_{i}$ and $F_{j}$ follows by Lusztig's categorification theorem \cite[Theorem 10.17]{MR1088333}. The Serre relation between $E_{i}$ and $E_{j}$ follows by the adjointness of $E_{i}$ and $K_{i}F_{i}$ (up to shifts). One can see details in \cite[Propostion 3.25]{fang2023lusztig} .
 	
 	Assume $L$ is a semisimple complex in $\mathcal{Q}_{\mathbf{V},\mathbf{W}^{\bullet}}$, $E_{i}^{(M)}([L])=0$ holds for $M> \nu_{i}$. For $M> \tilde{\nu}_{i}$, we assume $F_{i}^{(M)}([L])$ belongs to $\mathcal{Q}_{\mathbf{V}'',\mathbf{W}^{\bullet}}/\mathcal{N}_{\mathbf{V}''}$, then $\nu''_{i}-\tilde{\nu}''_{i}\geqslant M -\tilde{\nu}''_{i} \geqslant 1$. In particular, $\mathbf{E}^{\geqslant 1}_{\mathbf{V}'',\mathbf{W}^{\bullet},i}= \mathbf{E}_{\mathbf{V}'',\mathbf{W}^{\bullet},\Omega^{N}_{i}}$ and $F_{i}^{(M)}([L])=0$ holds.  Hence $\mathcal{K}_{0}(\Lambda^{\bullet})$ is integrable. 
 \end{proof}

 One can easily check by definition as in \cite[Proposition 3.24]{fang2023lusztig} to obtain the following results.
 \begin{proposition}
 	The functors $F_{i},E_{i},K_{i},i \in I$ of $\mathcal{L}(\Lambda^{\bullet})$ and Verdier duality $\mathbf{D}$ have the following commutative relations,
 	\begin{align*}
 		F_{i}\mathbf{D}=&\mathbf{D}F_{i},\\
 		E_{i}\mathbf{D}=&\mathbf{D}E_{i},\\
 		K_{i}\mathbf{D}=&\mathbf{D}(K_{i})^{-1}.
 	\end{align*}
 In particular, Verdier duality $\mathbf{D}$ induces the bar-involution on $\mathcal{K}_{0}(\Lambda^{\bullet})$.
 \end{proposition}

\section{Tensor products of the integrable highest weight modules}

In this section, we  determine the structure of $\mathcal{K}_{0}(\Lambda^{\bullet})$ by induction on $N$. At the beginning, we need to state two propositions deduced from Lusztig's key lemma \cite[Lemma 6.4]{MR1604167} in this subsection. 

If $|\mathbf{V}'|=ri,|\mathbf{V}''|=|\mathbf{V}|-ri$, we denote $\mathbf{V'}$ by $\mathbf{V}'_{ri}$ and $\mathbf{V}''$ by $\mathbf{V}''_{\nu-ri}$. For an orientation $\tilde{\Omega}^{N}$ and a simple perverse sheaf $L$, we define $s_{i}(L)$ to be the largest integer $r$ satisfying that there exists a semisimple complex $L'$ such that $L$ is isomorphic to a shift of a direct summand of $\mathbf{Ind}^{\mathbf{V}\oplus \mathbf{W}^{\bullet}}_{\mathbf{V}'_{ri},\mathbf{V}''_{\nu-ri}\oplus \mathbf{W}^{\bullet}}(\bar{\mathbb{Q}}_{l} \boxtimes L')$. Notice that the definition of $s_{i}(L)$ does not depend on the choice of $\tilde{\Omega}^{N}$ by Proposition \ref{FD0} and Remark \ref{remarkFD}. Then we have the following result.

\begin{proposition}\label{l0}
	Let $L \in \mathcal{P}_{\mathbf{V}\oplus\mathbf{W}^{\bullet}}$ be a simple perverse sheaf such that $s_{i}(L)=t$ for some $0 \leqslant t \leqslant \nu_{i}$ and assume $|\mathbf{V}'|=ti$,  then there exists a unique simple perverse sheaf $K \in \mathcal{P}_{\mathbf{V}''\oplus\mathbf{W}^{\bullet}}$ with $s_{i}(K)=0$ such that $$\mathbf{Ind}^{\mathbf{V}\oplus\mathbf{W}^{\bullet}}_{\mathbf{V}',\mathbf{V}''\oplus\mathbf{W}^{\bullet}}(\bar{\mathbb{Q}}_{l} \boxtimes K) \cong L \oplus \bigoplus \limits_{L',g'} L'[g'], $$ 
	where  $L'[g']$ are various simple perverse sheaves in $ \mathcal{P}_{\mathbf{V}\oplus\mathbf{W}^{\bullet}}$ with $s_{i}(L')>t$ and $g' \in \mathbb{Z}$. Moreover, if $L \in \mathcal{P}_{\mathbf{V},\mathbf{W}^{\bullet}}$, then  $K$ also belongs to $\mathcal{P}_{\mathbf{V}'',\mathbf{W}^{\bullet}}$.
\end{proposition}
\begin{proof}
	View $Q^{(N)}$ as a large unframed quiver,	take an orientation such that $i$ is a source. Then apply \cite[Lemma 6.4,Proposition 6.6]{MR1088333} for $Q^{(N)}$, we can see that there exists a unique simple perverse sheaf $K \in \mathcal{P}_{\mathbf{V}''\oplus\mathbf{W}^{\bullet}}$ with $s_{i}(K)=0$ such that $$\mathbf{Ind}^{\mathbf{V}\oplus\mathbf{W}^{\bullet}}_{\mathbf{V}',\mathbf{V}''\oplus\mathbf{W}^{\bullet}}(\bar{\mathbb{Q}}_{l} \boxtimes K) \cong L \oplus \bigoplus \limits_{L',g'} L'[g'], $$ 
	where  $L'[g']$ are various simple perverse sheaves in $ \mathcal{P}_{\mathbf{V}\oplus\mathbf{W}^{\bullet}}$ with $s_{i}(L')>t$ and $g' \in \mathbb{Z}$, and 
	$$\mathbf{Res}^{\mathbf{V}\oplus\mathbf{W}^{\bullet}}_{\mathbf{T},\mathbf{W}\oplus \mathbf{W}^{\bullet}}(L)\cong K \oplus \bigoplus \limits_{K',f'} K'[f']$$ is a direct sum of finitely many summands of the form $K'[f']$ for various simple perverse sheaves $K'$ in  $ \mathcal{P}_{\mathbf{V}''\oplus\mathbf{W}^{\bullet}}$ with $s_{i}(K')>0$ and $f' \in \mathbb{Z}$.  By Proposition \ref{indformula} and \ref{res formula}, we can see that if $L \in \mathcal{P}_{\mathbf{V},\mathbf{W}^{\bullet}}$, then  $K$ also belongs to $\mathcal{P}_{\mathbf{V}'',\mathbf{W}^{\bullet}}$. Since Lusztig's induction functor commutes with the Fourier-Deligne transforms by Proposition \ref{FD0}, we get the proof for orientation $\Omega^{N}$.
\end{proof}

Dually, if $|\mathbf{V}'|=|\mathbf{V}|-ri$ and $|\mathbf{V}''|=ri$, we denote $\mathbf{V}'$ by $\mathbf{V}'_{\nu-ri}$ and $\mathbf{V}''$ by $\mathbf{V}''_{ri}$. For an orientation $\tilde{\Omega}^{N}$ and a simple perverse sheaf $L$, we define $s_{i}^{\ast}(L)$ to be the largest integer $r$ satisfying that there exists $L'$ such that $L$ is isomorphic to a direct summand of $\mathbf{Ind}^{\mathbf{V}\oplus \mathbf{W}^{\bullet}}_{\mathbf{V}'_{\nu-ri}\oplus \mathbf{W}^{\bullet},\mathbf{V}''_{ri}}(L' \boxtimes \bar{\mathbb{Q}}_{l})$. By a similar argument, we have the following result dual to Proposition \ref{l0}.

\begin{proposition}\label{r0}
	Let $L\in\mathcal{P}_{\mathbf{V}\oplus\mathbf{W}^{\bullet}}$ be a simple perverse sheaf such that $s^{\ast}_{i}(L)=t$ for some $0 \leqslant t \leqslant \nu_{i}$ and assume $|\mathbf{V}''|=ti$,  then there exists a unique simple perverse sheaf $K \in \mathcal{P}_{\mathbf{V}'\oplus\mathbf{W}^{\bullet}}$ with $s^{\ast}_{i}(K)=0$ such that $$\mathbf{Ind}^{\mathbf{V}\oplus\mathbf{W}^{\bullet}}_{\mathbf{V}'\oplus\mathbf{W}^{\bullet},\mathbf{V}''}(K\boxtimes\bar{\mathbb{Q}}_{l}) \cong L \oplus \bigoplus \limits_{L',g'} L'[g'], $$ 
	where  $L'[g']$ are various simple perverse sheaves in $ \mathcal{P}_{\mathbf{V}\oplus\mathbf{W}^{\bullet}}$ with $s^{\ast}_{i}(L')>t$ and $g' \in \mathbb{Z}$. In particular, $L\in\mathcal{P}_{\mathbf{V}\oplus\mathbf{W}^{\bullet}}$ belongs to $\mathcal{N}_{\mathbf{V},i}$ if and only if $s^{\ast}_{i}(L)>0$ if and only if $L$ is a direct summand of some $L_{\boldsymbol{\mu}}$ such that $\boldsymbol{\mu}$ ends by  $i^{a}$ for some $a>0$.  
\end{proposition}

\subsection{Irreducible highest weight modules for $N=1$}
Firstly, we assume that $N=1$ and simply denote $\mathcal{Q}_{\mathbf{V},\mathbf{W}^{\bullet}}$ by $\mathcal{Q}_{\mathbf{V},\mathbf{W}}$, $\mathcal{L}(\Lambda^{\bullet})$ by $\mathcal{L}(\Lambda)$,  $\mathcal{V}_{0}(\Lambda^{\bullet})$ by $\mathcal{V}_{0}(\Lambda)$ and $\mathcal{K}_{0}(\Lambda^{\bullet})$ by $\mathcal{K}_{0}(\Lambda)$. We assume that $\boldsymbol{d}$ is the flag type choosed for $\mathbf{W}=\mathbf{W}^{1}$, then $\mathcal{Q}_{\mathbf{V},\mathbf{W}}$ is the semisimple category consisting of direct sums of shifted summands of those $L_{\boldsymbol{\nu}\boldsymbol{d}}$. We denote the set of simple objects in $\mathcal{Q}_{\mathbf{V},\mathbf{W}}$ by $\mathcal{P}_{\mathbf{V},\mathbf{W}}$. We denote by $\mathcal{S}=\bigcup\limits_{\mathbf{V}} \mathcal{S}_{|\mathbf{V}|}$ all flag types of $I$-graded spaces $\mathbf{V}$. For any subset $X$ of an $\mathcal{A}$-module, we denote the $\mathcal{A}$-linear span of $X$ by $\mathcal{A}X$.

\begin{lemma}\label{highest}
	The $\mathcal{A}$-module $\mathcal{V}_{0}(\Lambda)$ is $\mathcal{A}$-spanned by those semisimple complexes $[L_{\boldsymbol{\nu}\boldsymbol{d}}],\boldsymbol{\nu} \in \mathcal{S}$.
\end{lemma}
\begin{proof}
	We apply  Proposition \ref{l0} for $Q^{(1)}$.
	Suppose $|\mathbf{V}| \neq 0$, then any $L \in \mathcal{P}_{\mathbf{V},\mathbf{W}}$ is a direct summand of a semisimple complex $[L_{\boldsymbol{\nu}\boldsymbol{d}}]$ and $\boldsymbol{\nu}$ is a nonempty sequence. Hence there exists some $i$ such that $s_{i}(L)=r >0$ by definition of $s_{i}$. We will prove by an increasing induction on $|\mathbf{V}|$ and a decreasing induction on $r$.
	
	We assume that for any  $L' \in \mathcal{P}_{\mathbf{V}',\mathbf{W}}$ with $|\mathbf{V}'|<|\mathbf{V}|$, $[L']$ can be linearly expressed by those $[L_{\boldsymbol{\nu}\boldsymbol{d}}]$. We also assume that for any $L' \in \mathcal{P}_{\mathbf{V},\mathbf{W}}$ with $s_{i}(L') >r$, $[L']$ can be linearly expressed by those $[L_{\boldsymbol{\nu}\boldsymbol{d}}]$.
	
	By Proposition \ref{l0},  we can find a simple perverse sheaf $L'$ such that $s_{i}(L')=0$ and the following equation holds:
	\begin{equation*}
		[\mathbf{Ind}( \overline{\mathbb{Q}}_{l}\boxtimes L' )]=[L]+\sum\limits_{s_{i}(L'')>r} c_{L''}[L''] 
	\end{equation*}
	
	By the inductive assumption, 
	\begin{equation*}
		[\mathbf{Ind}( \overline{\mathbb{Q}}_{l}  \boxtimes L')] \in \mathcal{A}\{[L_{\boldsymbol{\nu}\boldsymbol{d}}|\boldsymbol{\nu} \in \mathcal{S} \},
	\end{equation*}
	\begin{equation*}
		[L''] \in  \mathcal{A}\{[L_{\boldsymbol{\nu}\boldsymbol{d}}]|\boldsymbol{\nu}\in \mathcal{S} \}.
	\end{equation*}
	Hence  $[L] \in \mathcal{A}\{[L_{\boldsymbol{\nu}\boldsymbol{d}}]|\boldsymbol{\nu}\in \mathcal{S} \}$ and we complete the induction.
\end{proof}

It follows from the previous lemma and Theorem \ref{commutative relation} that $\mathcal{K}_{0}(\Lambda)$ is an integrable highest weight module, hence it is isomorphic to ${_{\mathcal{A}}L(\Lambda)}$. We have the following theorem.
\begin{theorem}\label{thm1} \cite[Theorem 3.26]{fang2023lusztig}
	For graded space $\mathbf{W}$ and dominant weight $\Lambda$  such that $\langle \Lambda, \alpha^{\vee}_{i} \rangle ={\rm{dim}} \mathbf{W}_{i}$, the Grothendieck group $\mathcal{K}_{0}(\Lambda)$ of  $\mathcal{L}(\Lambda)$ together with the functors $E^{(n)}_{i},F^{(n)}_{i}$ and  $K^{\pm}_{i}, i \in I, n\in \mathbb{N}$ becomes a $_{\mathcal{A}}\mathbf{U}$-module. And there exists an isomorphism of $_{\mathcal{A}}\mathbf{U}$-modules
	\begin{equation*}
		\varsigma^{\Lambda}:\mathcal{K}_{0}(\Lambda) \rightarrow {_{\mathcal{A}}L(\Lambda)}
	\end{equation*}
	sending the constant sheaf $[\overline{\mathbb{Q}}_{l}]=[L_{\boldsymbol{d}}]$ on
	$\mathbf{E}_{0,\mathbf{W},\Omega^{1}}$
	to the highest weight vector $v_{\Lambda}$ in  ${_{\mathcal{A}}L(\Lambda)}$. Moreover, the set
	$\{\varsigma^{\Lambda}([L])|L$ is a nonzero simple perverse sheaf in $\mathcal{L}_{\mathbf{V}}(\Lambda)\}$ form an $\mathcal{A}$-basis of the weight space ${_{\mathcal{A}}L(\Lambda)_{|\mathbf{V}|}}$, which is exactly the canonical basis.
\end{theorem}

\begin{remark}\label{lowest}
	If we swap the definition of $\mathcal{E}^{(n)}_{i}$ and $\mathcal{F}^{(n)}_{i}$, and swap the definition of $K_{i}$ and $K^{-1}_{i}$, then the Grothendieck group of  $\mathcal{L}(\Lambda)$ will become the irreducible integrable lowest weight representation with the lowest weight $-\Lambda$.
\end{remark}

\subsection{Tensor products for general $N$}

We denote the sequences of graded spaces $$(\mathbf{W}^{1},\mathbf{W}^{2},\cdots \mathbf{W}^{N-1})$$ and $$(\mathbf{W}^{2},\mathbf{W}^{3},\cdots \mathbf{W}^{N})$$ by $\mathbf{W}^{\bullet -1}$ and $\mathbf{W}^{-1\bullet}$ respectively. We also denote the sequences of weights $(\Lambda_{1},\Lambda_{2},\cdots,\Lambda_{N-1})$ and $(\Lambda_{2},\Lambda_{3},\cdots,\Lambda_{N})$ by $\Lambda^{\bullet -1}$ and $\Lambda^{-1\bullet }$ respectively. We can also consider the moduli space $	\mathbf{E}_{\mathbf{V},\mathbf{W}^{\bullet-1},\Omega^{N-1}}$ of $N-1$-framed quivers and its  category $\mathcal{Q}_{\mathbf{V},\mathbf{W}^{\bullet-1}}$ and the global localization $\mathcal{L}_{\mathbf{V}}(\Lambda^{\bullet-1})= \mathcal{Q}_{\mathbf{V},\mathbf{W}^{\bullet-1}}/ \mathcal{N}_{\mathbf{V}}$. 

Let $\mathcal{V}_{0}(\Lambda^{\bullet})_{\mathbf{V}}$ be the Grothendieck group of $\mathcal{Q}_{\mathbf{V},\mathbf{W}^{\bullet}}$, and let $\mathcal{V}_{0}(\Lambda^{\bullet-1})_{\mathbf{V}}$ be the Grothendieck group of $\mathcal{Q}_{\mathbf{V},\mathbf{W}^{\bullet-1}}$. In this subsection, we only give the proof for the case $N=2$, the statements for general cases can be proved similarly.

\begin{proposition}
	If $\mathbf{W}^{2}=0$, $\mathcal{V}_{0}(\Lambda_{1},0)$ is $\mathcal{A}$-spanned by those complexes $[L_{\boldsymbol{\nu}^{1}\boldsymbol{d}^{1}\boldsymbol{\nu}^{2}}],\boldsymbol{\nu}^{1},\boldsymbol{\nu}^{2} \in \mathcal{S}$.
	
\end{proposition}
\begin{proof}
	We only need to show that for any $L \in \mathcal{P}_{\mathbf{V},\mathbf{W}^{\bullet}}$, we have $[L] \in \mathcal{A}\{[L_{\boldsymbol{\nu}^{1}\boldsymbol{d}^{1}\boldsymbol{\nu}^{2}}]|\boldsymbol{\nu}^{1},\boldsymbol{\nu}^{2} \in \mathcal{S} \}$. We argue by increasing induction on $|\mathbf{V}|$ and descending induction on $s_{i}(L)$ or $s_{i}^{\ast}(L)$.

	Suppose $|\mathbf{V}| \neq 0$, then any $L \in \mathcal{P}_{\mathbf{V},\mathbf{W}^{\bullet}}$ is a direct summand of some $L_{\boldsymbol{\nu}^{1}\boldsymbol{d}^{1}\boldsymbol{\nu}^{2}}$, where either $\boldsymbol{\nu}^{1} $ or $\boldsymbol{\nu}^{2} $ is a nonempty sequence. By definition of $s_{i}$ and $s^{\ast}_{i}$, there exists $i \in I$ such that $s_{i}(L)=r >0$ or $s_{i}^{\ast}(L)=r>0$.
	
	Without loss of generality, we assume $s_{i}^{\ast}(L)=r>0$.
	We assume that for any $L' \in \mathcal{P}_{\mathbf{V}',\mathbf{W}^{\bullet}}$ with $|\mathbf{V}'|<|\mathbf{V}|$, we have $[L'] \in \mathcal{A}\{[L_{\boldsymbol{\nu}^{1}\boldsymbol{d}^{1}\boldsymbol{\nu}^{2}}]|\boldsymbol{\nu}^{1},\boldsymbol{\nu}^{2} \in \mathcal{S} \}$. 
	
	We also assume that for any $L' \in \mathcal{P}_{\mathbf{V},\mathbf{W}^{\bullet}}$ with $s_{i}^{\ast}(L')>r$, we have   $[L'] \in \mathcal{A}\{[L_{\boldsymbol{\nu}^{1}\boldsymbol{d}^{1}\boldsymbol{\nu}^{2}}]|\boldsymbol{\nu}^{1},\boldsymbol{\nu}^{2} \in \mathcal{S} \}$.
	
	Then by Proposition \ref{r0}, we can find a simple perverse sheaf $L'$ such that $s_{i}^{\ast}(L')=0$ and the following equation holds
	\begin{equation*}
		[\mathbf{Ind}(L' \boxtimes \overline{\mathbb{Q}}_{l} )]=[L]+\sum\limits_{s_{i}^{\ast}(L'')>r} c_{L''}[L''] .
	\end{equation*}
	By inductive assumption, 
	\begin{equation*}
		[\mathbf{Ind}(L' \boxtimes \overline{\mathbb{Q}}_{l} )] \in \mathcal{A}\{[L_{\boldsymbol{\nu}^{1}\boldsymbol{d}^{1}\boldsymbol{\nu}^{2}}]|\boldsymbol{\nu}^{1},\boldsymbol{\nu}^{2} \in \mathcal{S} \},
	\end{equation*}
	\begin{equation*}
		[L''] \in  \mathcal{A}\{[L_{\boldsymbol{\nu}^{1}\boldsymbol{d}^{1}\boldsymbol{\nu}^{2}}]|\boldsymbol{\nu}^{1},\boldsymbol{\nu}^{2} \in \mathcal{S} \}.
	\end{equation*}
	Hence $[L] \in \mathcal{A}\{[L_{\boldsymbol{\nu}^{1}\boldsymbol{d}^{1}\boldsymbol{\nu}^{2}}]|\boldsymbol{\nu}^{1},\boldsymbol{\nu}^{2} \in \mathcal{S} \}$ and we complete the induction.
\end{proof}

By similar argument, we can prove the following inductive proposition for general $N$.
\begin{proposition}
	If   we have   $\mathcal{V}(\Lambda^{\bullet-1})= \mathcal{A}\{[L_{\boldsymbol{\nu}^{1}\boldsymbol{d}^{1}\cdots  \boldsymbol{\nu}^{N-1}\boldsymbol{d}^{N-1}}]|\boldsymbol{\nu}^{k} \in \mathcal{S},1 \leqslant k \leqslant N-1  \} $ for  $	\mathbf{E}_{\mathbf{V},\mathbf{W}^{\bullet-1},\Omega^{N-1}}$, then for $\mathbf{W}^{N}=0$, $\mathcal{V}(\Lambda^{\bullet-1},0)$ is $\mathcal{A}$-spanned by elements of the form $[L_{\boldsymbol{\nu}^{1}\boldsymbol{d}^{1}\cdots  \boldsymbol{\nu}^{N}}], \boldsymbol{\nu}^{k} \in \mathcal{S}$ for $1 \leqslant k \leqslant N $ .
\end{proposition}
\begin{proof}
	If $L \in \mathcal{P}_{\mathbf{V},\mathbf{W}^{\bullet}}$ satisfies $s^{\ast}_{i}(L)=0 $ for any $i$, then it can only appear as a direct summand of $L_{\boldsymbol{\nu}^{1}\boldsymbol{d}^{1}\cdots  \boldsymbol{\nu}^{N}}$ with $\boldsymbol{\nu}^{N}$ empty. Then by $N-1$ case, it lies in $\mathcal{A}\{[L_{\boldsymbol{\nu}^{1}\boldsymbol{d}^{1}\cdots  \boldsymbol{\nu}^{N-1}\boldsymbol{d}^{N-1}}]|\boldsymbol{\nu}^{k} \in \mathcal{S},1 \leqslant k \leqslant N-1  \} \subseteq  \mathcal{A}\{[L_{\boldsymbol{\nu}^{1}\boldsymbol{d}^{1}\cdots  \boldsymbol{\nu}^{N}}]| \boldsymbol{\nu}^{k} \in \mathcal{S},1 \leqslant k \leqslant N \} $.
	
	Otherwise, $s^{\ast}_{i}(L)=r>0 $ for some $i$, then we can argue by induction on $r$ and $|\mathbf{V}|$ as what we have done in the $N=2$ case.
\end{proof}

\begin{corollary}
	The  Grothendieck group $\mathcal{V}_{0}(\Lambda_{1},\Lambda_{2})$ is $\mathcal{A}$-spanned by $[L_{\boldsymbol{\nu}^{1}\boldsymbol{d}^{1}\boldsymbol{\nu}^{2}\boldsymbol{d}^{2}}],\boldsymbol{\nu}^{1},\boldsymbol{\nu}^{2} \in \mathcal{S}$.
\end{corollary}

\begin{proof}
	Consider the projection $\pi_{\mathbf{V},\mathbf{W}^{2}}: \mathbf{E}_{\mathbf{V},\mathbf{W}^{\bullet},\Omega^{2}} \rightarrow \mathbf{E}_{\mathbf{V},\mathbf{W}^{1},\Omega^{1}}$
	\begin{equation*}
		\pi_{\mathbf{V},\mathbf{W}^{2}}((x_{h})_{h\in \Omega^{2}} )=(x_{h})_{h\in \Omega^{1}}
	\end{equation*}
	Since $\pi_{\mathbf{V},\mathbf{W}^{2}}$ is a trivial vector bundle, the composition of functors $(\pi_{\mathbf{V},\mathbf{W}^{2}})_{!}[\sum\limits_{i \in I}{{\rm{dim}}}\mathbf{V}_{i}{{\rm{dim}}}\mathbf{W}_{i^{2}}  ] $ and $\pi_{\mathbf{V},\mathbf{W}^{2}}^{\ast}[\sum\limits_{i \in I}{{\rm{dim}}}\mathbf{V}_{i}{{\rm{dim}}}\mathbf{W}_{i^{2}}  ]$ is isomorphic to $Id$. The functors $(\pi_{\mathbf{V},\mathbf{W}^{2}})_{!}[\sum\limits_{i \in I}{{\rm{dim}}}\mathbf{V}_{i}{{\rm{dim}}}\mathbf{W}_{i^{2}}  ] $ and $\pi_{\mathbf{V},\mathbf{W}^{2}}^{\ast}[\sum\limits_{i \in I}{{\rm{dim}}}\mathbf{V}_{i}{{\rm{dim}}}\mathbf{W}_{i^{2}}  ]$ induce quasi-inverse linear isomorphisms.
	
	On the other hand, notice that by definition there is an isomorphism
	\begin{equation*}
		\mathbf{Ind}^{\mathbf{V}\oplus\mathbf{W}^{\bullet}}_{\mathbf{V}\oplus\mathbf{W}^{1},\mathbf{W}^{2} } ( - \boxtimes L_{\boldsymbol{d}^{2}}) \cong \pi_{\mathbf{V},\mathbf{W}^{2}}^{\ast}[\sum\limits_{i \in I}{{\rm{dim}}}\mathbf{V}_{i}{{\rm{dim}}}\mathbf{W}_{i^{2}}].
	\end{equation*}
	
	Hence
	\begin{equation*}
		\begin{split}
			\mathcal{V}_{0}(\Lambda_{1},\Lambda_{2})
			\cong & \mathcal{V}_{0}(\Lambda_{1},0)\\
			= &\mathcal{A}\{ [L_{\boldsymbol{\nu}^{1}\boldsymbol{d}^{1}\boldsymbol{\nu}^{2}}]|\boldsymbol{\nu}^{1},\boldsymbol{\nu}^{2} \in \mathcal{S}\}\\		
			\cong &\mathcal{A}\{[L_{\boldsymbol{\nu}^{1}\boldsymbol{d}^{1}\boldsymbol{\nu}^{2}\boldsymbol{d}^{2}}]|\boldsymbol{\nu}^{1},\boldsymbol{\nu}^{2} \in \mathcal{S}\}
		\end{split}
	\end{equation*}
	
	Since $\mathcal{A}\{[L_{\boldsymbol{\nu}^{1}\boldsymbol{d}^{1}\boldsymbol{\nu}^{2}\boldsymbol{d}^{2}}]|\boldsymbol{\nu}^{1},\boldsymbol{\nu}^{2} \in \mathcal{S}\} \subseteq \mathcal{V}_{0}(\Lambda_{1},\Lambda_{2})$, and the above argument tells us that they have the same rank for each $\mathbf{V}$, hence $\mathcal{A}\{[L_{\boldsymbol{\nu}^{1}\boldsymbol{d}^{1}\boldsymbol{\nu}^{2}\boldsymbol{d}^{2}}]|\boldsymbol{\nu}^{1},\boldsymbol{\nu}^{2} \in \mathcal{S}\}= \mathcal{V}_{0}(\Lambda_{1},\Lambda_{2})$.
\end{proof}

\begin{corollary}
	If we have 
	\begin{equation*}
		\mathcal{V}(\Lambda^{\bullet-1}) =\mathcal{A}\{[L_{\boldsymbol{\nu}^{1}\boldsymbol{d}^{1}\cdots  \boldsymbol{\nu}^{N-1}\boldsymbol{d}^{N-1}}]| \boldsymbol{\nu}^{k} \in \mathcal{S},1 \leqslant k \leqslant N-1 \} 
	\end{equation*}
	for $	\mathbf{E}_{\mathbf{V},\mathbf{W}^{\bullet-1},\Omega^{N-1}}$, then $\mathcal{V}(\Lambda^{\bullet})$ is $\mathcal{A}$-spanned by elements of the form $[L_{\boldsymbol{\nu}^{1}\boldsymbol{d}^{1}\cdots  \boldsymbol{\nu}^{N}\boldsymbol{d}^{N}}], \boldsymbol{\nu}^{k} \in \mathcal{S}$ for $1 \leqslant k \leqslant N$.
\end{corollary}
\begin{proof}
	Consider the trivial vector bundle
	\begin{equation*}
		\pi_{\mathbf{W}^{N}}: \mathbf{E}_{\mathbf{V},\mathbf{W}^{\bullet},\Omega^{N}} \rightarrow  \mathbf{E}_{\mathbf{V},\mathbf{W}^{\bullet-1},\Omega^{N-1}}
	\end{equation*}
	defined by forgetting the components ${\rm{Hom}}(\mathbf{V}_{i},\mathbf{W}_{i^{N}})$, then by similar arguments as before, we get the proof.
\end{proof}

A natural idea is to take direct sum of those $\mathbf{W}^{k}$ and only one framing as what has been done in the framework of Nakajima's quiver varieties in \cite{MR1865400} and \cite{MR3077693}. However, in the framework of perverse sheaves, if we consider the framed quiver and take  direct sum of $\mathbf{W}^{k}$ as \cite{MR3177922}, the functor $\cdot L_{\boldsymbol{d}^{N}}$ given by the right multiplication of $L_{\boldsymbol{d}^{N}}$ will not  be isomorphic to the pull back of a trivial vector bundle and it is not fully-faithful. 

Indeed, we can consider $\mathfrak{sl}_{2}$-case and the framed quiver is of type $A_{2}$ in this case, then $L_{\boldsymbol{d}^{1}}$ is the constant sheaf on $pt$ and $L_{\boldsymbol{d}^{1}\boldsymbol{d}^{2}}$ is the cohomology of a Grassmannian, which is not simple. Hence the functor $\cdot L_{\boldsymbol{d}^{2}}$ cannot be fully-faithful. This is one reason for taking $N$ framings.

\begin{proposition}\label{tonggouN}
	(1) If $\mathbf{W}^{k}=0$ for every $1\leqslant k < N$, we have an isomorphism of ${_{\mathcal{A}}\mathbf{U}}$-modules  $\mathcal{K}_{0}(0,\Lambda_{N}) \cong \mathcal{K}_{0}(\Lambda_{N})$.
	
	(2) If $\mathbf{W}^{k}=0$ for some $1\leqslant k < N$, we have an isomorphism of ${_{\mathcal{A}}\mathbf{U}}$-modules  $$\mathcal{K}_{0}(\Lambda_{1},\Lambda_{2},\cdots,\Lambda_{k-1},0,\Lambda_{k+1},\cdots,\Lambda_{N-1},\Lambda_{N}) \cong \mathcal{K}_{0}(\Lambda_{1},\Lambda_{2}\cdots,\Lambda_{k-1},\Lambda_{k+1},\cdots,\Lambda_{N-1},\Lambda_{N}).$$ 
	
	(3) If $\mathbf{W}^{N}=0$, we have an isomorphism of ${_{\mathcal{A}}\mathbf{U}}$-modules  $\mathcal{K}_{0}(\Lambda^{\bullet-1},0) \cong \mathcal{K}_{0}(\Lambda^{\bullet-1})$.
\end{proposition}

\begin{proof}
	The first and second statements hold by definition. Notice that if $\boldsymbol{\nu}^{N}$ is nonempty, then the complex $L_{\boldsymbol{\nu}^{1}\boldsymbol{d}^{1}\cdots  \boldsymbol{\nu}^{N}}$ is contained in $\mathcal{N}_{\mathbf{V},i}$ for some $i$. Hence  $\mathcal{A}\{[L_{\boldsymbol{\nu}^{1}\boldsymbol{d}^{1}\cdots  \boldsymbol{\nu}^{N}}]| \boldsymbol{\nu}^{k} \in \mathcal{S},1 \leqslant k \leqslant N \}$ and $\mathcal{A}\{[L_{\boldsymbol{\nu}^{1}\boldsymbol{d}^{1}\cdots  \boldsymbol{\nu}^{N-1}\boldsymbol{d}^{N-1}}]| \boldsymbol{\nu}^{k} \in \mathcal{S},1 \leqslant k \leqslant N-1 \}$ have the same image  in $\mathcal{K}_{0}(\Lambda^{\bullet-1},0)$ and the last statement follows.
\end{proof}

Now we can study the structure of $\mathcal{K}_{0}(\Lambda_{1},\Lambda_{2})$ (and $\mathcal{K}_{0}(\Lambda^{\bullet})$). Now in this section we simplify the notations  as the following
\begin{equation*}
	\mathbf{Ind}^{\mathbf{V}}_{\mathbf{V}^{1},\mathbf{V}^{2}}=\mathbf{Ind}^{\mathbf{V}\oplus\mathbf{W}^{\bullet}}_{\mathbf{V}^{1}\oplus\mathbf{W}^{\bullet-1},\mathbf{V}^{2}\oplus\mathbf{W}^{N}},
\end{equation*}
\begin{equation*}
	\mathbf{Res}^{\mathbf{V}}_{\mathbf{V}^{1},\mathbf{V}^{2}}=\mathbf{Res}^{\mathbf{V}\oplus\mathbf{W}^{\bullet}}_{\mathbf{V}^{1}\oplus\mathbf{W}^{\bullet-1},\mathbf{V}^{2}\oplus\mathbf{W}^{N}},
\end{equation*}
\begin{equation*}
	\mathbf{Res'}^{\mathbf{V}}_{\mathbf{V}^{1},\mathbf{V}^{2}}=\mathbf{Res}^{\mathbf{V}\oplus\mathbf{W}^{\bullet}}_{\mathbf{V}^{1}\oplus\mathbf{W}^{N},\mathbf{V}^{2}\oplus\mathbf{W}^{\bullet-1}}.
\end{equation*}

Let $\tilde{\Delta}_{N},\tilde{\Delta}_{N}'$ be the linear maps induced by the functor $\bigoplus\limits_{\mathbf{V}^{1},\mathbf{V}^{2}} \mathbf{Res}^{\mathbf{V}}_{\mathbf{V}^{1},\mathbf{V}^{2}}$ and $\bigoplus\limits_{\mathbf{V}^{1},\mathbf{V}^{2}} \mathbf{Res'}^{\mathbf{V}}_{\mathbf{V}^{1},\mathbf{V}^{2}}$ respectively
\begin{equation*}
	\tilde{\Delta}_{N}:\mathcal{V}_{0}(\Lambda^{\bullet}) \rightarrow \mathcal{V}_{0}(\Lambda^{\bullet-1},0) \otimes \mathcal{V}_{0}(0,\Lambda_{N}),
\end{equation*}
\begin{equation*}
	\tilde{\Delta}'_{N}:\mathcal{V}_{0}(\Lambda^{\bullet}) \rightarrow  \mathcal{V}_{0}(0,\Lambda_{N})   \otimes \mathcal{V}_{0}(\Lambda^{\bullet-1},0) .
\end{equation*}
and let $\cdot_{N}$ be the linear map induced by the functor $\bigoplus\limits_{\mathbf{V}^{1},\mathbf{V}^{2}}\mathbf{Ind}^{\mathbf{V}}_{\mathbf{V}^{1},\mathbf{V}^{2}}$
\begin{equation*}
	\cdot_{N}:\mathcal{V}_{0}(\Lambda^{\bullet-1},0) \otimes \mathcal{V}_{0}(0,\Lambda_{N}) \rightarrow \mathcal{V}_{0}(\Lambda^{\bullet}) .
\end{equation*}
In particular, for any complex $A$, we let $[A]\cdot_{N}$ be the linear map induced by the left multiplication. Then we define $\Delta_{N}$ to be the morphisms  $\Delta_{N}:\mathcal{V}_{0}(\Lambda^{\bullet}) \rightarrow  \mathcal{V}_{0}(0,\Lambda_{N}) \otimes  \mathcal{V}_{0}(\Lambda^{\bullet-1},0)$ given by $$\Delta_{N}= P (\mathbf{D}\otimes \mathbf{D}) 	\tilde{\Delta}_{N} \mathbf{D},$$ where $\mathbf{D}$ is the linear map induced by the Verdier duality and $P$ is the linear map swapping the components. Similarly, we define $\Delta'_{N}:\mathcal{V}_{0}(\Lambda^{\bullet}) \rightarrow    \mathcal{V}_{0}(\Lambda^{\bullet-1},0) \otimes \mathcal{V}_{0}(0,\Lambda_{N}) $ by $$\Delta'_{N}= P (\mathbf{D}\otimes \mathbf{D}) 	\tilde{\Delta}'_{N} \mathbf{D}.$$

Let $\mathcal{I}=\mathcal{I}(\Lambda_{1},\Lambda_{2}) \subseteq \mathcal{V}_{0}(\Lambda_{1},\Lambda_{2})$ be the $\mathbb{Z}[v,v^{-1}]$-submodule  spanned by objects in $\mathcal{N}_{\mathbf{V}}$, then by Proposition \ref{indformula} , \ref{res formula} and \ref{r0}, we can check that
\begin{equation*}
	\mathcal{V}_{0}(\Lambda_{1},0)	\cdot_{2} (\mathcal{I}(0,\Lambda_{2})) \subseteq \mathcal{I}(\Lambda_{1},\Lambda_{2})
\end{equation*}
\begin{equation*}
	\tilde{\Delta}_{2}(\mathcal{I}(\Lambda_{1},\Lambda_{2}))\subseteq  \mathcal{I}(\Lambda_{1},0) \otimes \mathcal{V}_{0}(0,\Lambda_{2})  + \mathcal{V}_{0}(\Lambda_{1},0) \otimes \mathcal{I}(0,\Lambda_{2}).
\end{equation*}
If we regard  $\mathcal{K}_{0}(\Lambda_{1},\Lambda_{2})$ as a $\mathbb{Z}[v,v^{-1}]$-quotient of $\mathcal{V}_{0}(\Lambda_{1},\Lambda_{2})$, we obtain linear maps
\begin{equation*}
	\cdot_{2}:\mathcal{V}_{0}(\Lambda_{1},0) \otimes \mathcal{K}_{0}(0,\Lambda_{2}) \rightarrow \mathcal{K}_{0}(\Lambda_{1},\Lambda_{2}),
\end{equation*}
\begin{equation*}
	\tilde{\Delta}_{2}:\mathcal{K}_{0}(\Lambda_{1},\Lambda_{2}) \rightarrow \mathcal{K}_{0}(\Lambda_{1},0) \otimes \mathcal{K}_{0}(0,\Lambda_{2}),
\end{equation*}
and 
\begin{equation*}
	\Delta_{2}:\mathcal{K}_{0}(\Lambda_{1},\Lambda_{2}) \rightarrow \mathcal{K}_{0}(0,\Lambda_{2}) \otimes \mathcal{K}_{0}(\Lambda_{1},0).
\end{equation*}

\begin{proposition}\label{chengfaman}
	Regarding $\mathcal{K}_{0}(\Lambda_{1},0)$ as a $\mathbb{Z}[v,v^{-1}]$-submodule of $\mathcal{V}_{0}(\Lambda_{1},0)$, we can restrict $\cdot_{2}$ on $\mathcal{K}_{0}(\Lambda_{1},0) \otimes \mathcal{K}_{0}(0,\Lambda_{2})$ 
	\begin{equation*}
		\cdot_{2}:\mathcal{K}_{0}(\Lambda_{1},0) \otimes \mathcal{K}_{0}(0,\Lambda_{2}) \rightarrow \mathcal{K}_{0}(\Lambda_{1},\Lambda_{2}) 
	\end{equation*}
	then $\cdot_{2}:\mathcal{K}_{0}(\Lambda_{1},0) \otimes \mathcal{K}_{0}(0,\Lambda_{2}) \rightarrow \mathcal{K}_{0}(\Lambda_{1},\Lambda_{2}) $ is surjective.
\end{proposition}
\begin{proof}
	By Proposition 4.8 and Corollary 4.10, we know that $
	\cdot_{2}:\mathcal{V}_{0}(\Lambda_{1},0) \otimes \mathcal{V}_{0}(0,\Lambda_{2}) \rightarrow \mathcal{V}_{0}(\Lambda_{1},\Lambda_{2})$ is surjective. Hence $\cdot_{2}:\mathcal{V}_{0}(\Lambda_{1},0) \otimes \mathcal{K}_{0}(0,\Lambda_{2}) \rightarrow \mathcal{K}_{0}(\Lambda_{1},\Lambda_{2})$ is surjective. It suffices to show $\mathcal{I}(\Lambda_{1},0)\cdot_{2} \mathcal{K}_{0}(0,\Lambda_{2})  \subseteq  \mathcal{K}_{0}(\Lambda_{1},0) \cdot_{2} \mathcal{K}_{0}(0,\Lambda_{2})$, then $\mathcal{K}_{0}(\Lambda_{1},\Lambda_{2}) \subseteq \mathcal{V}_{0}(\Lambda_{1},0) \cdot_{2}\mathcal{K}_{0}(0,\Lambda_{2}) \subseteq  \mathcal{K}_{0}(\Lambda_{1},0) \cdot_{2} \mathcal{K}_{0}(0,\Lambda_{2})$.  
	
	Consider a simple perverse sheaf $A$ in $\mathcal{N}_{\mathbf{V}}$, we need to prove $[A] \cdot_{2} \mathcal{K}_{0}(0,\Lambda_{2}) \subseteq\mathcal{K}_{0}(\Lambda_{1},0) \cdot_{2} \mathcal{K}_{0}(0,\Lambda_{2}) $. Since $A$ belongs to $\mathcal{N}_{\mathbf{V},i}$ for some $i$, $s_{i}^{\ast}(A)=r>0$. We argue by  increasing induction on $|\mathbf{V}|$ and descending induction on $s^{\ast}_{i}(A)=r$. We assume that for any $L' \in \mathcal{I}(\Lambda_{1},0)$ such that $|\mathbf{V}'|<|\mathbf{V}|$, we have  $[L' ]\cdot_{2} \mathcal{K}_{0}(0,\Lambda_{2}) \subseteq\mathcal{K}_{0}(\Lambda_{1},0) \cdot_{2} \mathcal{K}_{0}(0,\Lambda_{2}) $. We also assume that for any $L' \in \mathcal{I}(\Lambda_{1},0)$ such that $s_{i}^{\ast}(L')>r$, we have  $[L' ]\cdot_{2} \mathcal{K}_{0}(0,\Lambda_{2}) \subseteq\mathcal{K}_{0}(\Lambda_{1},0) \cdot_{2} \mathcal{K}_{0}(0,\Lambda_{2}) $.
	
	Then by Proposition \ref{r0}, we can find a simple perverse sheaf $L'$ such that $s_{i}^{\ast}(L')=0$ and the following equation holds
	\begin{equation*}
		[\mathbf{Ind}(L' \boxtimes \overline{\mathbb{Q}}_{l} )]=[A]+\sum\limits_{s_{i}^{\ast}(L'')>r} c_{L''}[L''] .
	\end{equation*}
	Notice that by the associativity of the induction functor, we have
	\begin{equation*}
		\begin{split}
			&[A]\cdot_{2} \mathcal{K}_{0}(0,\Lambda_{2})\\
			\subseteq & [L'] \cdot_{2} [\overline{\mathbb{Q}}_{l}] \cdot_{2} \mathcal{K}_{0}(0,\Lambda_{2})+ \sum\limits_{s_{i}^{\ast}(L'')>r}[L''] \cdot_{2} \mathcal{K}_{0}(0,\Lambda_{2}) \\
			\subseteq & [L']\cdot_{2} \mathcal{K}_{0}(0,\Lambda_{2})+ \sum\limits_{s_{i}^{\ast}(L'')>r}[L'']\cdot_{2} \mathcal{K}_{0}(0,\Lambda_{2})\\
			\subseteq & \mathcal{K}_{0}(\Lambda_{1},0) \cdot_{2} \mathcal{K}_{0}(0,\Lambda_{2}).	
		\end{split}
	\end{equation*}
	We complete the proof.
\end{proof}

\begin{proposition}\label{yuchengman}
	The linear maps $\tilde{\Delta}_{2}:\mathcal{K}_{0}(\Lambda_{1},\Lambda_{2}) \rightarrow \mathcal{K}_{0}(\Lambda_{1},0) \otimes \mathcal{K}_{0}(0,\Lambda_{2}) $ and $\Delta_{2}:\mathcal{K}_{0}(\Lambda_{1},\Lambda_{2}) \rightarrow \mathcal{K}_{0}(0,\Lambda_{2}) \otimes \mathcal{K}_{0}(\Lambda_{1},0) $ are surjective. 
\end{proposition}
\begin{proof}
	Since $\Delta_{2}$ is the composition of $\tilde{\Delta}_{2}$ and some linear isomorphisms, we only need to prove the statement for $\tilde{\Delta}_{2}$.
	
	Since $\mathcal{K}_{0}(\Lambda_{1},0) \otimes \mathcal{K}_{0}(0,\Lambda_{2})$ is spanned by $[L_{\boldsymbol{\nu}^{1}\boldsymbol{d}^{1}}] \otimes [L_{\boldsymbol{\nu}^{2}\boldsymbol{d}^{2}}]$, we only need to show that each element of the form $[L_{\boldsymbol{\nu}^{1}\boldsymbol{d}^{1}}] \otimes [L_{\boldsymbol{\nu}^{2}\boldsymbol{d}^{2}}] \in \mathcal{K}_{0}(\Lambda_{1},0)_{\mathbf{V}^{1} } \otimes \mathcal{K}_{0}(0,\Lambda_{2})_{\mathbf{V}^{2}} $ is contained in the image of $\tilde{\Delta}_{2}$.
	
	We argue by induction on $|\mathbf{V}^{1}|$.
	
	(1) If $|\mathbf{V}^{1}|=0$, we claim that $\tilde{\Delta}_{2}([L_{\boldsymbol{d}^{1}\boldsymbol{\nu}^{2}\boldsymbol{d}^{2}}])=L_{\boldsymbol{d}^{1}}\otimes L_{\boldsymbol{\nu}^{2}\boldsymbol{d}^{2}}$. Indeed, after applying the restriction functor, $L_{\boldsymbol{d}^{1}\boldsymbol{\nu}^{2}\boldsymbol{d}^{2}}$ is the direct sum of $L_{\boldsymbol{d}^{1}}\otimes L_{\boldsymbol{\nu}^{2}\boldsymbol{d}^{2}}$ and some shifts of the other complexes of the form $L_{\boldsymbol{d}^{1}\boldsymbol{\nu}'}\otimes L_{\boldsymbol{\nu}''\boldsymbol{d}^{2}}$. If  $\boldsymbol{\nu}'$ is a nonempty sequence, $L_{\boldsymbol{d}^{1}\boldsymbol{\nu}'}$ belongs to $\mathcal{N}_{\mathbf{V}'}$, so we get the proof.
	We have $\mathcal{K}_{0}(\Lambda_{1},0)_{\mathbf{V}^{1} } \otimes \mathcal{K}_{0}(0,\Lambda_{2})_{\mathbf{V}^{2}} \subseteq {\rm{Im}}\tilde{\Delta}_{2}$ holds for $\mathbf{V}^{1}=0$ and all $\mathbf{V}^{2}$.

	(2) Assume for $|\mathbf{V}'|<|\mathbf{V}^{1}|$ and all $\mathbf{V}^{2}$, we have $\mathcal{K}_{0}(\Lambda_{1},0)_{\mathbf{V}' } \otimes \mathcal{K}_{0}(0,\Lambda_{2})_{\mathbf{V}^{2}} \subseteq {\rm{Im}}\tilde{\Delta}_{2}$.
	Consider $[L_{\boldsymbol{\nu}^{1}\boldsymbol{d}^{1}}] \otimes [L_{\boldsymbol{\nu}^{2}\boldsymbol{d}^{2}}] \in \mathcal{K}_{0}(\Lambda_{1},0)_{\mathbf{V}^{1} } \otimes \mathcal{K}_{0}(0,\Lambda_{2})_{\mathbf{V}^{2}}$, we have 
	\begin{equation*}
		\begin{split}
		\tilde{\Delta}_{2}([L_{\boldsymbol{\nu}^{1}\boldsymbol{d}^{1}\boldsymbol{\nu}^{2}\boldsymbol{d}^{2}}] )
			= &v^{M}[L_{\boldsymbol{\nu}^{1}\boldsymbol{d}^{1}}] \otimes [L_{\boldsymbol{\nu}^{2}\boldsymbol{d}^{2}}] 
			+ \sum\limits_{\boldsymbol{\nu}',\boldsymbol{\nu}'',\boldsymbol{\nu}'\neq \boldsymbol{\nu}} v^{M'}[L_{\boldsymbol{\nu}'\boldsymbol{d}^{1}}] \otimes [L_{\boldsymbol{\nu}''\boldsymbol{d}^{2}}]  \\
			+&  \sum\limits_{\boldsymbol{\nu}',\boldsymbol{\nu}'',\boldsymbol{\nu}'''} v^{M''}[L_{\boldsymbol{\nu}'\boldsymbol{d}^{1}\boldsymbol{\nu}''}] \otimes [L_{\boldsymbol{\nu}'''\boldsymbol{d}^{2}}] . 
		\end{split}
	\end{equation*}
	For those $[L_{\boldsymbol{\nu}'\boldsymbol{d}^{1}}] \otimes [L_{\boldsymbol{\nu}''\boldsymbol{d}^{2}}]$ such that $\boldsymbol{\nu}' \neq \boldsymbol{\nu}$, by inductive hypothesis we have
	\begin{equation*}
		[L_{\boldsymbol{\nu}'\boldsymbol{d}^{1}}] \otimes [L_{\boldsymbol{\nu}''\boldsymbol{d}^{2}}] \in \mathcal{K}_{0}(\Lambda_{1},0)_{\mathbf{V}' } \otimes \mathcal{K}_{0}(0,\Lambda_{2})_{\mathbf{V}'_{2}} \subseteq {\rm{Im}}\tilde{\Delta}_{2}.
	\end{equation*}
	For those nonempty sequences $\boldsymbol{\nu}''$, we have
	\begin{equation*}
		[L_{\boldsymbol{\nu}'\boldsymbol{d}^{1}\boldsymbol{\nu}''}] \otimes [L_{\boldsymbol{\nu}'''\boldsymbol{d}^{2}}] \in \mathcal{I}(\Lambda_{1},0) \otimes \mathcal{K}_{0}(0,\Lambda_{2}).
	\end{equation*}
	Hence $\mathcal{K}_{0}(\Lambda_{1},0)_{\mathbf{V}^{1}} \otimes \mathcal{K}_{0}(0,\Lambda_{2})_{\mathbf{V}^{2}} \subseteq {\rm{Im}}\tilde{\Delta}_{2}$ and we complete the proof.
\end{proof}

\begin{proposition}
	We have the following isomorphism (for N=2)
	\begin{equation*}
		\begin{split}
			\bigoplus\limits_{|\mathbf{V}^{1}|+|\mathbf{V}^{2}|=|\mathbf{V}|}\mathbf{Res}^{\mathbf{V}}_{\mathbf{V}^{1},\mathbf{V}^{2}}\circ F_{i}& \cong (F_{i}\otimes Id)\circ (\bigoplus\limits_{|\mathbf{V}^{1}|+|\mathbf{V}^{2}|=|\mathbf{V}|-i}\mathbf{Res}^{\mathbf{V}}_{\mathbf{V}^{1},\mathbf{V}^{2}}) \\
			&\oplus  (Id\otimes F_{i})\circ(\bigoplus\limits_{|\mathbf{V}^{1}|+|\mathbf{V}^{2}|=|\mathbf{V}|-i}\mathbf{Res}^{\mathbf{V}}_{\mathbf{V}^{1},\mathbf{V}^{2}}[-(\alpha_{i},|\mathbf{V}^{1}|+|\mathbf{W}^{1}|) ]) .
		\end{split}
	\end{equation*}
\end{proposition}
\begin{proof}
	We can use the main theorem in \cite{MR4524567} to get the proof directly. It is indeed a special case of  the main theorem in \cite{MR4524567}. 
\end{proof}

If we only  want to prove the proposition on the level of the Grothendieck group, there is anthoer way. Indeed, by using Proposition \ref{indformula} and \ref{res formula}, we can check the isomorphism for each $L_{\boldsymbol{\nu}^{1}\boldsymbol{d}^{1}\boldsymbol{\nu}^{2}\boldsymbol{d}^{2}}$. Then the proposition holds since these complexes span  $\mathcal{V}_{0}(\Lambda^{\bullet})$.

\begin{corollary} \label{tongtai}
	The morphism $\Delta_{2}:\mathcal{K}_{0}(\Lambda_{1},\Lambda_{2}) \rightarrow \mathcal{K}_{0}(0,\Lambda_{2}) \otimes \mathcal{K}_{0}(\Lambda_{1},0) $ is a $_{\mathcal{A}}\mathbf{U}^{-}$-linear morphism.
\end{corollary}
\begin{proof}
	After extending to $\mathbb{Q}(v)$, we show that $\Delta_{2}$ is a $\mathbf{U}^{-}$-morphism, or equivalently, $\Delta_{2} F_{i} = F_{i} \Delta_{2}$ for every $i\in I$.  Notice that for $m_{1}\otimes m_{2} \in \mathcal{K}_{0}(\Lambda_{1},0)_{\mathbf{V}^{1}} \otimes \mathcal{K}_{0}(0,\Lambda_{2})_{\mathbf{V}^{2}} $ with dimension vectors $(\nu^{1},\omega^{1}),(\nu^{2},\omega^{2})$, we have
	\begin{equation*}
		\begin{split}
			& (K_{i} \otimes F_{i}) (m_{1}\otimes m_{2})\\
			=& v^{
				\langle\Lambda_{\omega^{1}}-\alpha_{\nu^{1}},\alpha_{i}^{\vee} \rangle}m_{1}\otimes F_{i}m_{2}\\
			=& v^{-(\alpha_{i}, \omega^{1}+\nu^{1})}m_{1}\otimes F_{i}m_{2}.
		\end{split}
	\end{equation*}
	By Proposition 4.13, $\tilde{\Delta}_{2} F_{i} =  (F_{i}\otimes 1 + K_{i} \otimes F_{i}) \tilde{\Delta}_{2}$ holds for $\tilde{\Delta}_{2}:\mathcal{V}_{0}(\Lambda_{1},\Lambda_{2}) \rightarrow \mathcal{V}_{0}(\Lambda_{1},0) \otimes \mathcal{V}_{0}(0,\Lambda_{2}) $,  hence it still holds for quotients $\tilde{\Delta}_{2}:\mathcal{K}_{0}(\Lambda_{1},\Lambda_{2}) \rightarrow \mathcal{K}_{0}(\Lambda_{1},0) \otimes \mathcal{K}_{0}(0,\Lambda_{2}) $.
	Then \begin{equation*}
		\begin{split}
			\Delta_{2} F_{i}=& P (\mathbf{D}\otimes \mathbf{D}) \tilde{\Delta}_{2} \mathbf{D} F_{i}\\
			=& P (\mathbf{D}\otimes \mathbf{D}) \tilde{\Delta}_{2}F_{i} \mathbf{D} \\
			=&P (\mathbf{D}\otimes \mathbf{D}) (F_{i}\otimes 1 + K_{i} \otimes F_{i})\tilde{\Delta}_{2} \mathbf{D} \\
			=&P (F_{i}\otimes 1 + K^{-1}_{i} \otimes F_{i}) (\mathbf{D}\otimes \mathbf{D}) \tilde{\Delta}_{2} \mathbf{D}\\
			=&  (F_{i}\otimes K^{-1}_{i} + 1 \otimes F_{i})P (\mathbf{D}\otimes \mathbf{D}) \tilde{\Delta}_{2} \mathbf{D}\\
			=&\Delta(F_{i}) \Delta_{2},
		\end{split}
	\end{equation*}
where $\Delta$ is the coproduct of $\mathbf{U}$. Hence $\Delta_{2}$ is a $_{\mathcal{A}}\mathbf{U}^{-}$-morphism.
\end{proof}

\begin{theorem}\label{thm2}
	For graded space $\mathbf{W}^{1},\mathbf{W}^{2}$ and dominant weight $\Lambda_{1},\Lambda_{2}$  such that $\langle \Lambda_{k}, \alpha^{\vee}_{i} \rangle ={\rm{dim}}\mathbf{W}_{i^{k}}$, the Grothendieck group $\mathcal{K}_{0}(\Lambda_{1},\Lambda_{2})$ of global localization $\mathcal{L}(\Lambda^{\bullet})$ together with the functors $E^{(n)}_{i},F^{(n)}_{i},K^{\pm}_{i},i \in I, n\in \mathbb{N}$ becomes a $_{\mathcal{A}}\mathbf{U}$-module, which is isomorphic to the tensor product ${_{\mathcal{A}}L}(\Lambda_{2})\otimes{_{\mathcal{A}}L}(\Lambda_{1})$ of integrable highest weight modules via $\Delta_{2}$. 
\end{theorem}
\begin{proof}
	By Proposition \ref{yuchengman} and Corollary \ref{tongtai}, we have a surjective $_{\mathcal{A}}\mathbf{U}^{-}$-morphism $$\Delta_{2}:\mathcal{K}_{0}(\Lambda_{1},\Lambda_{2}) \rightarrow \mathcal{K}_{0}(0,\Lambda_{2})  \otimes \mathcal{K}_{0}(\Lambda_{1},0). $$
	On the other hand, by \ref{chengfaman}, we have a surjective linear map $\cdot_{2}P: \mathcal{K}_{0}(0,\Lambda_{2})  \otimes \mathcal{K}_{0}(\Lambda_{1},0)\rightarrow \mathcal{K}_{0}(\Lambda_{1},\Lambda_{2})$.
	
	By comparing their ranks as $\mathbb{Z}[v,v^{-1}]$-modules, we know that $\Delta_{2} $ is indeed a $_{\mathcal{A}}\mathbf{U}^{-}$-isomorphism. Since both $\mathcal{K}_{0}(\Lambda_{1},\Lambda_{2})$ and  $\mathcal{K}_{0}(0,\Lambda_{2}) \otimes \mathcal{K}_{0}(\Lambda_{1},0)$ are integrable, we know that they are isomorphic to each other as $_{\mathcal{A}}\mathbf{U}$-modules. (One can also check that $\Delta_{2}$ commutes with $E_{i}$ on each $[L_{\boldsymbol{\nu}^{1}\boldsymbol{d}^{1}\boldsymbol{\nu}^{2}\boldsymbol{d}^{2}}]$ to show that $\Delta_{2}$ is $\mathbf{U}^{+} $-linear.)  Then the theorem follows by Proposition \ref{tonggouN}.
\end{proof}

\begin{remark}\label{ltensor}
		If we identify $\mathcal{K}_{0}(\Lambda)$ with the lowest weight module as Remark \ref{lowest}, then from the proof of  Corollary \ref{tongtai} we can see that $\tilde{\Delta}_{2}$ is indeed $_{\mathcal{A}}\mathbf{U}^{+}$-linear and gives an isomorphism from $\mathcal{K}_{0}(\Lambda_{1},\Lambda_{2})$ to the tensor product of lowest weight modules.
\end{remark}

We can exchange $\mathbf{W}^{1}$ and $\mathbf{W}^{2}$, and consider the functor $\bigoplus\limits_{\mathbf{V}^{1},\mathbf{V}^{2}}\mathbf{Res'}^{\mathbf{V}}_{\mathbf{V}^{1},\mathbf{V}^{2}}$ for $N=2$. This functor induces  $_{\mathcal{A}}\mathbf{U}$-linear morphisms
\begin{equation*}
	\tilde{\Delta}'_{2}:\mathcal{K}_{0}(\Lambda_{1},\Lambda_{2}) \rightarrow \mathcal{K}_{0}(\Lambda_{2})\otimes \mathcal{K}_{0}(\Lambda_{1}),
\end{equation*}
\begin{equation*}
	\Delta'_{2}:\mathcal{K}_{0}(\Lambda_{1},\Lambda_{2}) \rightarrow \mathcal{K}_{0}(\Lambda_{1})\otimes \mathcal{K}_{0}(\Lambda_{2}),
\end{equation*}
which is a surjection (and hence an isomorphism) by a similar argument as Proposition \ref{yuchengman}. Hence $\Delta_{2}'\circ \Delta_{2}^{-1}$ defines the restriction of $\mathcal{R}$-matrix \cite{MR1359532} on $L(\Lambda_{1}) \otimes L(\Lambda_{2})$
\begin{equation*}\label{rmatrix}
	\Delta_{2}'\circ \Delta_{2}^{-1}:L(\Lambda_{2}) \otimes L(\Lambda_{1}) \rightarrow L(\Lambda_{1}) \otimes L(\Lambda_{2}).
\end{equation*}

By similar arguments as what we have done for $N=2$, we have the following proposition.
\begin{proposition}
	(1) The morphism $\cdot_{N}$ gives a $\mathbb{Z}[v,v^{-1}]$-linear surjection
	\begin{equation*}
		\cdot_{N}:\mathcal{K}_{0}(\Lambda^{\bullet-1},0) \otimes \mathcal{K}_{0}(0,\Lambda_{N}) \rightarrow \mathcal{K}_{0}(\Lambda^{\bullet}). 
	\end{equation*}
	
	(2) Both $\Delta_{N}$ and $	\Delta'_{N}$ induce surjective morphisms of $_{\mathcal{A}}\mathbf{U}$-modules:
	\begin{equation*}
		\Delta_{N}:\mathcal{K}_{0}(\Lambda^{\bullet}) \rightarrow \mathcal{K}_{0}(0,\Lambda_{N})  \otimes \mathcal{K}_{0}(\Lambda^{\bullet-1},0) ,
	\end{equation*}
	\begin{equation*}
		\Delta'_{N}:\mathcal{K}_{0}(\Lambda^{\bullet}) \rightarrow   \mathcal{K}_{0}(\Lambda^{\bullet-1},0)  \otimes \mathcal{K}_{0}(0,\Lambda_{N}) .
	\end{equation*}
\end{proposition}

Using the above propositions, we get the following theorem by induction on $N$: 
\begin{theorem}\label{thm3}
	For sequences of graded space $\mathbf{W}^{\bullet}$ and dominant weight $\Lambda^{\bullet}$  such that $\langle \Lambda_{k}, \alpha^{\vee}_{i} \rangle ={\rm{dim}} \mathbf{W}_{i^{k}}$, the Grothendieck group $\mathcal{K}_{0}(\Lambda^{\bullet})$ of global localization $\mathcal{L}(\Lambda^{\bullet})$ together with the functors $E^{(n)}_{i},F^{(n)}_{i},K^{\pm}_{i},i \in I, n\in \mathbb{N}$ becomes a $_{\mathcal{A}}\mathbf{U}$-module, which is isomorphic to tensor products of highest weight modules $_{\mathcal{A}}L(\Lambda_{N})\otimes {_{\mathcal{A}}L(\Lambda_{N-1})}\otimes \cdots \otimes{_{\mathcal{A}}L(\Lambda_{1})}$.
	Moreover, the compositions of Verdier Duality and the restriction functors induce isomorphisms of $_{\mathcal{A}} \mathbf{U}$-modules
	\begin{equation*}
	\Delta_{N}:\mathcal{K}_{0}(\Lambda^{\bullet}) \rightarrow \mathcal{K}_{0}(0,\Lambda_{N})  \otimes \mathcal{K}_{0}(\Lambda^{\bullet-1},0) ,
\end{equation*}
\begin{equation*}
	\Delta'_{N}:\mathcal{K}_{0}(\Lambda^{\bullet}) \rightarrow   \mathcal{K}_{0}(\Lambda^{\bullet-1},0)  \otimes \mathcal{K}_{0}(0,\Lambda_{N}) .
\end{equation*}
\end{theorem}

\begin{corollary}
	For $1\leqslant k \leqslant N$, denote $(\Lambda_{1},\Lambda_{2},\cdots,\Lambda_{k-1},\Lambda_{k+1},\Lambda_{k+2},\cdots,\Lambda_{N})$ by $\Lambda^{\bullet,\hat{k}}$ and $(\mathbf{W}^{1},\mathbf{W}^{2},\cdots,\mathbf{W}^{k-1},\mathbf{W}^{k+1},\mathbf{W}^{k+2},\cdots,\mathbf{W}^{N})$ by $\mathbf{W}^{\bullet,\hat{k}}$.
 The functors $\bigoplus\limits_{\mathbf{V}^{1},\mathbf{V}^{2}} \mathbf{Res}^{\mathbf{V}\oplus \mathbf{W}^{\bullet}}_{\mathbf{V}^{1} \oplus \mathbf{W}^{\bullet,\hat{k}},\mathbf{V}^{2}\oplus \mathbf{W}^{k} } $  and $\bigoplus\limits_{\mathbf{V}^{1},\mathbf{V}^{2}} \mathbf{Res}^{\mathbf{V}\oplus \mathbf{W}^{\bullet}}_{\mathbf{V}^{1} \oplus \mathbf{W}^{k},\mathbf{V}^{2}\oplus \mathbf{W}^{\bullet,\hat{k}} }$ respectively induce linear maps
	$$ \tilde{\Delta}_{N,k}: \mathcal{K}_{0}(\Lambda^{\bullet})\rightarrow \mathcal{K}_{0}(\Lambda^{\bullet,\hat{k}})\otimes \mathcal{K}_{0}(\Lambda_{k}),$$
	$$ _{k}\tilde{\Delta}_{N}: \mathcal{K}_{0}(\Lambda^{\bullet})\rightarrow \mathcal{K}_{0}(\Lambda_{k})\otimes \mathcal{K}_{0}(\Lambda^{\bullet,\hat{k}}).$$
	Denote the compositions $P(\mathbf{D}\otimes \mathbf{D})\tilde{\Delta}_{N,k} \mathbf{D}$ and $P(\mathbf{D}\otimes \mathbf{D}){_{k}\tilde{\Delta}}_{N} \mathbf{D}$ by $\Delta_{N,k}$ and $_{k}\Delta_{N}$ respectively, then the linear maps
	$$\Delta_{N,k}: \mathcal{K}_{0}(\Lambda^{\bullet})\rightarrow \mathcal{K}_{0}(\Lambda_{k})\otimes \mathcal{K}_{0}(\Lambda^{\bullet,\hat{k}}),$$
	$$ _{k}\Delta_{N}: \mathcal{K}_{0}(\Lambda^{\bullet})\rightarrow \mathcal{K}_{0}(\Lambda^{\bullet,\hat{k}})\otimes \mathcal{K}_{0}(\Lambda_{k}) $$
	are isomorphisms of $_{\mathcal{A}}\mathbf{U}$-modules.
\end{corollary}
\begin{proof}
	We only prove the statement for $_{k}\Delta_{N}$, the other one can be proved similarly. By a similar argument as Corollary \ref{tongtai}, we can see that $_{k}\Delta_{N}$ is $\mathbf{U}^{-}$-linear. Since both $\mathcal{K}_{0}(\Lambda^{\bullet})$ and $\mathcal{K}_{0}(\Lambda_{k})\otimes \mathcal{K}_{0}(\Lambda^{\bullet,\hat{k}})$ are isomorphic to the tensor products of $_{\mathcal{A}}L(\Lambda_{l}), 1\leqslant l\leqslant N$, it suffices to show that $ _{k}\Delta_{N}$ is surjective. It is equivalent to show that ${_{k}\tilde{\Delta}}_{N}$ is surjective. By Corollary 4.8, it suffices to show that for any $\boldsymbol{\nu}^{l} \in \mathcal{S}, 1\leqslant l \leqslant N$, the element $[L_{\boldsymbol{\nu}^{k}\boldsymbol{d}^{k}}]\otimes [L_{\boldsymbol{\nu}^{1}\boldsymbol{d}^{1}\cdots\boldsymbol{\nu}^{k-1}\boldsymbol{d}^{k-1} \boldsymbol{\nu}^{k+1}\boldsymbol{d}^{k+1}\cdots\boldsymbol{\nu}^{N}\boldsymbol{d}^{N}}]$ belongs to ${\rm{Im}}({_{k}\tilde{\Delta}}_{N})$. 
	
	(1) If $\boldsymbol{\nu}^{l}$ is empty for $1 \leqslant l \leqslant k-1$, then by
	Proposition \ref{res formula}, the complex $$L_{\boldsymbol{\nu}^{k}\boldsymbol{d}^{k}}\boxtimes L_{\boldsymbol{d}^{1}\cdots\boldsymbol{d}^{k-1} \boldsymbol{\nu}^{k+1}\boldsymbol{d}^{k+1}\cdots\boldsymbol{\nu}^{N}\boldsymbol{d}^{N}}$$ appears as a direct summand of $\bigoplus\limits_{\mathbf{V}^{1},\mathbf{V}^{2}} \mathbf{Res}^{\mathbf{V}\oplus \mathbf{W}^{\bullet}}_{\mathbf{V}^{1} \oplus \mathbf{W}^{k},\mathbf{V}^{2}\oplus \mathbf{W}^{\bullet,\hat{k}}}(L_{\boldsymbol{d}^{1}\cdots\boldsymbol{d}^{k-1}\boldsymbol{\nu}^{k}\boldsymbol{d}^{k} \boldsymbol{\nu}^{k+1}\boldsymbol{d}^{k+1}\cdots\boldsymbol{\nu}^{N}\boldsymbol{d}^{N}})$, and the other summands are either of the form $L_{\boldsymbol{\nu}'\boldsymbol{d}^{k}\boldsymbol{\nu}''}\boxtimes L_{\boldsymbol{d}^{1}\cdots\boldsymbol{d}^{k-1}\boldsymbol{\nu}'^{k} \boldsymbol{\nu}'^{k+1}\boldsymbol{d}^{k+1}\cdots\boldsymbol{\nu}'^{N}\boldsymbol{d}^{N}}$, or of the form $L_{\boldsymbol{\nu}'\boldsymbol{d}^{k}}\boxtimes L_{\boldsymbol{d}^{1}\cdots\boldsymbol{d}^{k-1}\boldsymbol{\nu}'^{k} \boldsymbol{\nu}'^{k+1}\boldsymbol{d}^{k+1}\cdots\boldsymbol{\nu}'^{N}\boldsymbol{d}^{N}}$ such that the dimension vector of $\boldsymbol{\nu}'$ is smaller than that of $\boldsymbol{\nu}^{k}$. Notice that  $L_{\boldsymbol{\nu}'\boldsymbol{d}^{k}\boldsymbol{\nu}''}$ belongs to $\bigoplus \limits_{\mathbf{V}^{1}} \mathcal{N}_{\mathbf{V}^{1}}$. Then by induction on the dimension vector of $\boldsymbol{\nu}^{k}$, we can  show that $$[L_{\boldsymbol{\nu}^{k}\boldsymbol{d}^{k}}]\otimes [L_{\boldsymbol{d}^{1}\cdots\boldsymbol{d}^{k-1} \boldsymbol{\nu}^{k+1}\boldsymbol{d}^{k+1}\cdots\boldsymbol{\nu}^{N}\boldsymbol{d}^{N}}] \in {\rm{Im}}({_{k}\tilde{\Delta}}_{N}).$$

	(2) We assume that for any $\boldsymbol{\nu}^{l} \in \mathcal{S}, 1 \leqslant l \leqslant N$ with $\boldsymbol{\nu}^{l}, 1 \leqslant l \leqslant k-s$ empty, the element $[L_{\boldsymbol{\nu}^{k}\boldsymbol{d}^{k}}]\otimes [L_{\boldsymbol{d}^{1}\cdots\boldsymbol{d}^{k-s}\boldsymbol{\nu}^{k-s+1}\boldsymbol{d}^{k-s+1} \cdots \boldsymbol{\nu}^{k-1} \boldsymbol{d}^{k-1} \boldsymbol{\nu}^{k+1}\boldsymbol{d}^{k+1}\cdots\boldsymbol{\nu}^{N}\boldsymbol{d}^{N}}]$ belongs to ${\rm{Im}}({_{k}\tilde{\Delta}}_{N})$. Now we argue by induction on dimension vectors of those $\boldsymbol{\nu}^{l}$ with $1 \leqslant l \leqslant k-1$ to prove that  for any $\boldsymbol{\nu}^{l} \in \mathcal{S}, 1 \leqslant l \leqslant N$ with $\boldsymbol{\nu}^{l},1 \leqslant l \leqslant k-s-1$ empty, the element $[L_{\boldsymbol{\nu}^{k}\boldsymbol{d}^{k}}]\otimes [L_{\boldsymbol{d}^{1}\cdots\boldsymbol{d}^{k-s-1}\boldsymbol{\nu}^{k-s}\boldsymbol{d}^{k-s} \cdots \boldsymbol{\nu}^{k-1} \boldsymbol{d}^{k-1} \boldsymbol{\nu}^{k+1}\boldsymbol{d}^{k+1}\cdots\boldsymbol{\nu}^{N}\boldsymbol{d}^{N}}]$ belongs to ${\rm{Im}}({_{k}\tilde{\Delta}}_{N})$. Indeed, by Proposition \ref{res formula},
	$$L_{\boldsymbol{\nu}^{k}\boldsymbol{d}^{k}}\boxtimes L_{\boldsymbol{d}^{1}\cdots\boldsymbol{d}^{k-s-1}\boldsymbol{\nu}^{k-s}\boldsymbol{d}^{k-s} \cdots \boldsymbol{\nu}^{k-1} \boldsymbol{d}^{k-1} \boldsymbol{\nu}^{k+1}\boldsymbol{d}^{k+1}\cdots\boldsymbol{\nu}^{N}\boldsymbol{d}^{N}}$$ appears as a direct summand of the complex $\bigoplus\limits_{\mathbf{V}^{1},\mathbf{V}^{2}} \mathbf{Res}^{\mathbf{V}\oplus \mathbf{W}^{\bullet}}_{\mathbf{V}^{1} \oplus \mathbf{W}^{k},\mathbf{V}^{2}\oplus \mathbf{W}^{\bullet,\hat{k}}}(L_{\boldsymbol{d}^{1}\cdots\boldsymbol{d}^{k-s-1}\boldsymbol{\nu}^{k-s}\boldsymbol{d}^{k-s} \cdots\boldsymbol{\nu}^{N}\boldsymbol{d}^{N}})$, and the other direct summands are  of one of the following three forms 
	$$(a)~L_{\boldsymbol{\nu}'\boldsymbol{d}^{k}\boldsymbol{\nu}''}\boxtimes L_{\boldsymbol{d}^{1}\cdots\boldsymbol{d}^{k-s-1}\boldsymbol{\nu}'^{k-s}\boldsymbol{d}^{k-s} \cdots \boldsymbol{\nu}'^{k-1} \boldsymbol{d}^{k-1} \boldsymbol{\nu}'^{k} \boldsymbol{\nu}'^{k+1}\boldsymbol{d}^{k+1}\cdots\boldsymbol{\nu}'^{N}\boldsymbol{d}^{N}}$$
	with $\boldsymbol{\nu}''$ nonempty, which belongs to $\bigoplus \limits_{\mathbf{V}^{1}} \mathcal{N}_{\mathbf{V}^{1}} \boxtimes \bigoplus \limits_{\mathbf{V}^{2}} \mathcal{Q}_{\mathbf{V}^{2},\mathbf{W}^{\bullet,
			\hat{k}}}$,
	$$(b)~L_{\boldsymbol{\nu}'\boldsymbol{d}^{k}}\boxtimes L_{\boldsymbol{d}^{1}\cdots\boldsymbol{d}^{k-s-1}\boldsymbol{\nu}'^{k-s}\boldsymbol{d}^{k-s} \cdots \boldsymbol{\nu}'^{k-1} \boldsymbol{d}^{k-1} \boldsymbol{\nu}'^{k} \boldsymbol{\nu}'^{k+1}\boldsymbol{d}^{k+1}\cdots\boldsymbol{\nu}'^{N}\boldsymbol{d}^{N}}$$ such that the sum of dimension vectors of $\boldsymbol{\nu}'^{l}, k-s \leqslant l  \leqslant k-1$ is smaller than  the sum of dimesion vectors of $\boldsymbol{\nu}^{l}, k-s \leqslant l  \leqslant k-1$,
	$$ (c)~L_{\boldsymbol{\nu}'\boldsymbol{d}^{k}}\boxtimes L_{\boldsymbol{d}^{1}\cdots\boldsymbol{d}^{k-s-1}\boldsymbol{\nu}'^{k-s}\boldsymbol{d}^{k-s} \cdots \boldsymbol{\nu}'^{k-1} \boldsymbol{d}^{k-1} \boldsymbol{\nu}'^{k} \boldsymbol{\nu}'^{k+1}\boldsymbol{d}^{k+1}\cdots\boldsymbol{\nu}'^{N}\boldsymbol{d}^{N}}$$ such that the dimension vector of $\boldsymbol{\nu}'$ is smaller than that of $\boldsymbol{\nu}^{k}$.
	
	Similarly, by induction on dimension vetor of $\boldsymbol{\nu}^{k}$, we may assume  direct summands of type $(c)$ belong to  ${\rm{Im}}({_{k}\tilde{\Delta}}_{N})$. By induction on $s$ and the sum of dimension vectors of $ \boldsymbol{\nu}'^{l}, k-s \leqslant l \leqslant k-1$, the images of the direct summands of type $(b)$  belong to  ${\rm{Im}}({_{k}\tilde{\Delta}}_{N})$. Hence $[L_{\boldsymbol{\nu}^{k}\boldsymbol{d}^{k}}]\otimes [L_{\boldsymbol{d}^{1}\cdots\boldsymbol{d}^{k-s-1}\boldsymbol{\nu}^{k-s}\boldsymbol{d}^{k-s} \cdots \boldsymbol{\nu}^{k-1} \boldsymbol{d}^{k-1} \boldsymbol{\nu}^{k+1}\boldsymbol{d}^{k+1}\cdots\boldsymbol{\nu}^{N}\boldsymbol{d}^{N}}]\in {\rm{Im}}({_{k}\tilde{\Delta}}_{N})$.
	
	(3) Finally, by induction on $s$, we finish the proof.
\end{proof}

\section{The $\mathcal{R}$-matrix and comparison with the canonical basis}
\subsection{The operators $\Theta$ and $\Psi$}
Recall that  for a basis $\mathcal{B}=\{ b \}$ of the algebra $\mathbf{f}$ (in \cite{MR1227098}) and its dual basis $\{ b^{\ast}\}$, the $\mathcal{R}$-matrix  $$\Theta=\sum\limits_{ \nu\in \mathbb{N}I}((-1)^{tr \nu} v_{\nu}\sum\limits_{b \in \mathcal{B}_{\nu}}b^{-} \otimes b^{\ast+}) $$
is well-defined as an operator on each tensor product $M\otimes M'$ of integrable modules, and satisfies 
$$ \Delta(u) \Theta= \Theta \bar{\Delta}(u) $$
for any $u \in \mathbf{U}$, where $tr \nu =\sum\limits_{ i\in I} \nu_{i}$, $\Delta$ is the coproduct of $\mathbf{U}$  and $\bar{\Delta}(x)=\overline{\Delta(\bar{x})}$. (See details in  \cite[24.1.3]{MR1227098}.)  The $\mathbb{Q}$-linear map $\Psi:M\otimes M' \rightarrow M \otimes M'$ is defined by $$\Psi(m\otimes m')=\Theta(\bar{m}\otimes\bar{m'}),$$ then
$$ u\Psi= \Psi \bar{u}$$ for any $u \in \mathbf{U}$.  (See details in  \cite[24.3.2]{MR1227098}.) In this section, we take $M\otimes M'=\mathcal{K}_{0}(\Lambda_{2}) \otimes \mathcal{K}_{0}(\Lambda_{1})$ and give a categorical realization of  $\Theta$ and $\Psi$.

Recall that the restriction functor gives an isomorphism $\Delta_{2}:\mathcal{K}_{0}(\Lambda_{1},\Lambda_{2})\rightarrow \mathcal{K}_{0}(\Lambda_{2}) \otimes \mathcal{K}_{0}(\Lambda_{1})$, we can define 
$$\Psi'= \Delta_{2} \mathbf{D}_{1,2} \Delta_{2}^{-1}, $$ 
$$\Theta'= \Delta_{2} \mathbf{D}_{1,2} \Delta_{2}^{-1}(\mathbf{D} \otimes \mathbf{D}), $$ where $\mathbf{D}_{1,2}$ and $\mathbf{D}\otimes\mathbf{D}$ are the  $\mathbb{Z}$-linear maps induced by the Verdier duality $\mathbf{D}$ on $\mathbf{E}_{\mathbf{V},\mathbf{W}^{\bullet},\Omega^{2}}$ and $\mathbf{E}_{\mathbf{V},\mathbf{W}^{1},\Omega^{2}}\times \mathbf{E}_{\mathbf{V},\mathbf{W}^{2},\Omega^{2}}$ respectively.

\begin{proposition}
	The following equations   hold
	$$\Theta'=\Theta, \Psi'=\Psi: \mathcal{K}_{0}(\Lambda_{2}) \otimes \mathcal{K}_{0}(\Lambda_{1}) \rightarrow \mathcal{K}_{0}(\Lambda_{2}) \otimes \mathcal{K}_{0}(\Lambda_{1}).$$
\end{proposition}
\begin{proof}
	By \cite{fang2023lusztig}, the map $\mathbf{D}$  coincides with the bar involutions of $\mathcal{K}_{0}(\Lambda_{1})$ and $\mathcal{K}_{0}(\Lambda_{2})$, so it suffices to prove $\Psi'=\Psi$.
	
	Notice that the Verdier duality $\mathbf{D}$ on $\mathbf{E}_{\mathbf{V},\mathbf{W}^{\bullet},\Omega^{2}}$ commutes with functors $E^{(n)}_{i}$ and $F^{(n)}_{i}$ by definition, and $\mathbf{D}[1]=[-1]\mathbf{D}$, so $u\Psi'=\Psi'\bar{u}$ for any $u \in \mathbf{U}$. Hence we only need to check that $\Psi=\Psi'$ for a set of generators of $\mathcal{K}_{0}(\Lambda_{2}) \otimes \mathcal{K}_{0}(\Lambda_{1})$. 
	
	Now we consider  $[L_{\boldsymbol{\nu}\boldsymbol{d}^{2}}] \otimes [L_{\boldsymbol{d}^{1}}] =F^{(n_{1})}_{i_{1}} \cdots F^{(n_{s})}_{i_{s}} v_{\Lambda_{2}} \otimes  v_{\Lambda_{1}}$. Since $\tilde{\Delta}_{2}([L_{\boldsymbol{d}^{1}\boldsymbol{\nu}\boldsymbol{d}^{2}}])= [L_{\boldsymbol{d}^{1}}] \otimes [L_{\boldsymbol{\nu}\boldsymbol{d}^{2}}]$, and $\mathbf{D}L_{\boldsymbol{\mu}}\cong L_{\boldsymbol{\mu}} $ for any flag type $\boldsymbol{\mu}$, we can see that $\Delta_{2}([L_{\boldsymbol{d}^{1}\boldsymbol{\nu}\boldsymbol{d}^{2}}])=[L_{\boldsymbol{\nu}\boldsymbol{d}^{2}}] \otimes [L_{\boldsymbol{d}^{1}}]$, and $\Psi'([L_{\boldsymbol{\nu}\boldsymbol{d}^{2}}] \otimes [L_{\boldsymbol{d}^{1}}] )=[L_{\boldsymbol{\nu}\boldsymbol{d}^{2}}] \otimes [L_{\boldsymbol{d}^{1}}].$  For $\nu \neq 0$, $b^{+}v_{\Lambda_{1}}=0$ for $b \in \mathcal{B}_{\nu}$, hence $\Theta(F^{(n_{1})}_{i_{1}} \cdots F^{(n_{s})}_{i_{s}} v_{\Lambda_{2}} \otimes  v_{\Lambda_{1}} )=F^{(n_{1})}_{i_{1}} \cdots F^{(n_{s})}_{i_{s}} v_{\Lambda_{2}} \otimes  v_{\Lambda_{1}}$ and $\Psi(F^{(n_{1})}_{i_{1}} \cdots F^{(n_{s})}_{i_{s}} v_{\Lambda_{2}} \otimes  v_{\Lambda_{1}})=F^{(n_{1})}_{i_{1}} \cdots F^{(n_{s})}_{i_{s}} v_{\Lambda_{2}} \otimes  v_{\Lambda_{1}}$. Notice that these  $[L_{\boldsymbol{\nu}\boldsymbol{d}^{2}}] \otimes [L_{\boldsymbol{d}^{1}}] =F^{(n_{1})}_{i_{1}} \cdots F^{(n_{s})}_{i_{s}} v_{\Lambda_{2}} \otimes  v_{\Lambda_{1}}$ generates $\mathcal{K}_{0}(\Lambda_{2}) \otimes \mathcal{K}_{0}(\Lambda_{1})$ as $\mathbf{U}$-module, we finish the proof.
\end{proof}

Recall that if $P: \mathcal{K}_{0}(\Lambda_{1}) \otimes \mathcal{K}_{0}(\Lambda_{2})  \rightarrow \mathcal{K}_{0}(\Lambda_{2}) \otimes \mathcal{K}_{0}(\Lambda_{1}), m\otimes m' \mapsto m'\otimes m$ is the swapping map and $f$ is a twisting map (See details in \cite[7.3]{MR1359532}), then  the composition of twisted $\mathcal{R}$-matrix $\Theta^{f}$ and $P$  defines an isomorphism of $\mathbf{U}$-modules $ \mathcal{K}_{0}(\Lambda_{1}) \otimes \mathcal{K}_{0}(\Lambda_{2})  \rightarrow \mathcal{K}_{0}(\Lambda_{2}) \otimes \mathcal{K}_{0}(\Lambda_{1})$.The proposition above also allows us to compare the isomorphism $\Delta_{2}' \circ \Delta_{2}^{-1}:\mathcal{K}_{0}(\Lambda_{2}) \otimes \mathcal{K}_{0}(\Lambda_{1})  \rightarrow \mathcal{K}_{0}(\Lambda_{1}) \otimes \mathcal{K}_{0}(\Lambda_{2})$ with  $\Theta^{f}$. 

\begin{proposition}
	Let $\alpha= |\mathbf{V}^{1} \oplus \mathbf{W}^{1}|$,$ \beta= |\mathbf{V}^{2} \oplus \mathbf{W}^{2}|$ and $\gamma=\alpha+\beta=|\mathbf{V} \oplus \mathbf{W}^{\bullet}|$, then 
	$$\mathbf{Res}^{\gamma}_{\alpha,\beta} \mathbf{D} \cong (sw)_{!} \mathbf{D} \mathbf{Res}^{\gamma}_{\beta,\alpha}[(\alpha,\beta)], $$
	where $sw: X\times Y \rightarrow  Y \times X, (x,y) \mapsto (y,x)$ is the swapping isomorphism of varieties.
\end{proposition} 
\begin{proof}
	
	It is well known that  the restriction functor is a hyperbolic localization functor. More precisely, there is a $k^*$-action on $\mathbf{E}_{\mathbf{V},\mathbf{W}^{\bullet},\Omega^{2}}$ such that the following diagrams commute
	\[
	\xymatrix{
		(\mathbf{E}_{\mathbf{V},\mathbf{W}^{\bullet},\Omega^{2}})^{k^*} \ar[d]_{\eta} 
		&  {\mathbf{E}^{+}_{\mathbf{V},\mathbf{W}^{\bullet},\Omega^{2}}} \ar[l]_{\pi^{+}} \ar[d]^{\cong} \ar[r]^{g^+} &\mathbf{E}_{\mathbf{V},\mathbf{W}^{\bullet},\Omega^{2}} \ar@{=}[d] 
		\\	
		{\mathbf{E}_{\mathbf{V}^{1},\mathbf{W}^{1},\Omega^{1}}} \times {\mathbf{E}_{\mathbf{V}^{2},\mathbf{W}^{2},\Omega^{1}}} 
		& F \ar[l]_-{\kappa_{\Omega}} \ar[r]^-{\iota_{\Omega}} &\mathbf{E}_{\mathbf{V},\mathbf{W}^{\bullet},\Omega^{2}} ,
	}
	\]
	
	\[
	\xymatrix{
		{(\mathbf{E}_{\mathbf{V},\mathbf{W}^{\bullet},\Omega^{2}}})^{k^*}  \ar[d]_{\eta'} 
		&  {\mathbf{E}^{-}_{\mathbf{V},\mathbf{W}^{\bullet},\Omega^{2}}} \ar[l]_{\pi^{-}} \ar[d]^{\cong} \ar[r]^{g^{-}} &\mathbf{E}_{\mathbf{V},\mathbf{W}^{\bullet},\Omega^{2}} \ar@{=}[d]
		\\
		{\mathbf{E}_{\mathbf{V}^{2},\mathbf{W}^{2},\Omega^{1}}}\times{\mathbf{E}_{\mathbf{V}^{1},\mathbf{W}^{1},\Omega^{1}}} 
		& F' \ar[l]_-{\kappa'_{\Omega}} \ar[r]^-{\iota'_{\Omega}} &\mathbf{E}_{\mathbf{V},\mathbf{W}^{\bullet},\Omega^{2}},
	}
	\]
	where those $\kappa_{\Omega},\kappa'_{\Omega}$ and $\iota_{\Omega},\iota'_{\Omega}$ are the morphisms appearing in the definition of $\mathbf{Res}^{\gamma}_{\alpha,\beta}$ and $\mathbf{Res}^{\gamma}_{\beta,\alpha}$, and those $\pi^{\pm}$ and $g^{\pm}$ are the morphisms appearing in the definition of the hyperbolic localizations. One can	see \cite{Barden} and \cite[Proposition 2.10]{MR4524567} for details.
	
	Since $\eta= sw \circ \eta'$, 
	\begin{equation*}
		\begin{split}
			\mathbf{Res}^{\gamma}_{\alpha,\beta} \mathbf{D}
			\cong &(\kappa_{\Omega})_{!}(\iota_{\Omega})^{\ast} [-\langle \alpha,\beta \rangle] \mathbf{D}
			\cong \mathbf{D} (\kappa_{\Omega})_{\ast}(\iota_{\Omega})^{!} [\langle \alpha,\beta \rangle] \\
			\cong &\mathbf{D}(\eta)_{\ast}(\pi^{+})_{\ast}(g^{+})^{!}[\langle \alpha,\beta \rangle]
			\cong  \mathbf{D}(\eta)_{\ast}(\pi^{-})_{!}(g^{-})^{\ast}[\langle \alpha,\beta \rangle]\\
			\cong & (sw)_{!} \mathbf{D}(\kappa'_{\Omega})_{!}(\iota'_{\Omega})^{\ast}[-\langle \beta,\alpha \rangle][( \alpha,\beta )]\\
			\cong & (sw)_{!} \mathbf{D} \mathbf{Res}^{\gamma}_{\beta,\alpha}[(\alpha,\beta)].
		\end{split}
	\end{equation*}
\end{proof}

Notice that $(sw)_{!}$ induces $P:\mathcal{K}_{0}(\Lambda_{1}) \otimes \mathcal{K}_{0}(\Lambda_{2})  \rightarrow \mathcal{K}_{0}(\Lambda_{2}) \otimes \mathcal{K}_{0}(\Lambda_{1})$, so the Proposition above implies that 
\begin{equation*}
	\begin{split}
		(\mathbf{D}\otimes\mathbf{D})\Delta_{2}
		=&P\tilde{\Delta}_{2}\mathbf{D}_{1,2}
		=PP(\mathbf{D}\otimes\mathbf{D})\tilde{\Delta}'_{2}[(\alpha,\beta)]\\
		=&(\mathbf{D}\otimes\mathbf{D})\tilde{\Delta}'_{2}[(\alpha,\beta)]
		= P\Delta'_{2}\mathbf{D}_{1,2}[(\alpha,\beta)].
	\end{split}
\end{equation*}
Hence 
\begin{equation*}
	\begin{split}
		&P \Theta^{-1} [(\alpha,\beta)]\\
		= &P (\mathbf{D} \otimes \mathbf{D}) \Delta_{2} \mathbf{D}_{1,2} \Delta_{2}^{-1} [(\alpha,\beta)] \\
		= &P  P\Delta'_{2}\mathbf{D}_{1,2} [(\alpha,\beta)]\mathbf{D}_{1,2} \Delta_{2}^{-1}[(\alpha,\beta)]\\
		=&\Delta'_{2}\Delta_{2}^{-1}.
	\end{split}
\end{equation*}
Since $[(\alpha,\beta)]$ induced a function $\tilde{f}(\alpha,\beta)=v^{(\alpha,\beta)}$, which satisfies the condition of \cite[7.3]{MR1359532}, we can see that $\Delta_{2}(\Delta'_{2})^{-1}=(\Delta'_{2}\Delta_{2}^{-1})^{-1}$ equals to the composition of $P$ and some $\Theta^{f}$.

\subsection{The geometric bilinear form}

\begin{definition}
	Define a bilinear form $(-,-)^{\Lambda^{\bullet}}$ on $\mathcal{K}_{0}(\Lambda^{\bullet})$ by:
	\begin{equation*}
		([A],[B])^{\Lambda^{\bullet}}=\sum\limits_{n \geq 0}{\rm{dim}} {\rm{Ext}}^{n}_{\mathcal{D}^{b}_{G_{\mathbf{V}}}(\mathbf{E}_{\mathbf{V},\mathbf{W}^{\bullet},\Omega^{N}})/ \mathcal{N}_{\mathbf{V}} }(\mathbf{D}A,B)v^{-n},
	\end{equation*}
	for any  semisimple complexes $A,B$ on $\mathbf{E}_{\mathbf{V},\mathbf{W}^{\bullet},\Omega^{N}}$. Otherwise,	if $A$ is a complex on $\mathbf{E}_{\mathbf{V},\mathbf{W}^{\bullet},\Omega^{N}}$ but  $B$ is a complex on $\mathbf{E}_{\mathbf{V}',\mathbf{W}^{\bullet},\Omega^{N}}$  such that $|\mathbf{V}| \neq |\mathbf{V}'|$, define
	\begin{equation*}
		([A],[B])^{\Lambda^{\bullet}}=0.
	\end{equation*}
\end{definition}

\begin{proposition}
	For any $i \in I$, the bilinear form $(-,-)^{\Lambda^{\bullet}}$ is contravariant with respect to $F_{i}$ and $E_{i}$. More precisely, for any objects $A,B$,
	\begin{equation*}
		( [ F_{i}A],[B] )^{\Lambda^{\bullet}}=([A],v [K^{-}_{i}E_{i}B])^{\Lambda^{\bullet}}.
	\end{equation*}
	Moreover, for simple perverse sheaves $L$ and $L'$,\\
	(1)If $[L] \neq [L']$, then  $([L],[L'])^{\Lambda^{\bullet}} \in v^{-1} \mathbb{Z}[[v^{-1}]] \cap \mathbb{Q}(v)$ ;\\
	(2)If $L$ is a simple perverse sheaf which is nonzero in $\mathcal{K}_{0}(\Lambda^{\bullet})$, then $([L],[L])^{\Lambda^{\bullet}} \in 1+ v^{-1}\mathbb{Z}[[v^{-1}]]\cap \mathbb{Q}(v)$ .
\end{proposition}

\begin{proof}	
	The proof is similar to the proof of \cite[Proposition 3.31]{fang2023lusztig}. Notice that simple perverse sheaves in $\mathcal{Q}$ are self-dual, the almost orthogonality follows from  the property of perverse $t$-structure. The contravariant property follows from the fact that $ F_{i}$ is left adjoint to  $K^{-}_{i}E_{i}[-1]$. (One can see details in \cite[Proposition 3.31]{fang2023lusztig}.) 
\end{proof}

Now we focus on $\mathcal{K}_{0}(\Lambda_{2}) \otimes \mathcal{K}_{0}(\Lambda_{1}),$ define a bilinear form $(-,-)^{\otimes}$ on $\mathcal{K}_{0}(\Lambda_{2}) \otimes \mathcal{K}_{0}(\Lambda_{1})$ by $$(m_{2}\otimes m_{1}, n_{2}\otimes n_{1})^{\otimes} = (m_{2},n_{2})^{\Lambda_{2}} (m_{1},n_{1})^{\Lambda_{1}}, $$
then $(-,-)^{\otimes}$ also satisfies the same contravariant property. Recall that $(-,-)^{\Lambda^{\bullet}}$ defines a contravariant bilinear form on $\mathcal{K}_{0}(\Lambda_{1},\Lambda_{2})$, the following proposition shows that $\Delta_{2}$ is not only an isomorphism of $\mathbf{U}$-modules, but also preserve contravariant biliear forms.

\begin{proposition}
	Assume that $N=2$ and $\Delta_{2}:\mathcal{K}_{0}(\Lambda_{1},\Lambda_{2}) \rightarrow \mathcal{K}_{0}(\Lambda_{2})\otimes \mathcal{K}_{0}(\Lambda_{1})$ is the isomorphism induced by the restriction functor and Verdier duality, then the following equation holds
	$$ (x,y)^{\Lambda^{\bullet}}=(\Delta_{2}(x),\Delta_{2}(y))^{\otimes}, $$ 
	for any $x,y$ in $\mathcal{K}_{0}(\Lambda_{1},\Lambda_{2})$.
\end{proposition}
\begin{proof}
	Since $(-,-)^{\otimes}$  and $(-,-)^{\Lambda^{\bullet}}$ have the same contravariant property, it suffices to show $$(x,y)^{\Lambda^{\bullet}}=(\Delta_{2}(x),\Delta_{2}(y))^{\otimes}$$ for a set of $\mathbf{U}^{-}$-generators of $\mathcal{K}_{0}(\Lambda_{1},\Lambda_{2})$. 
	Since $\Delta_{2} ([L_{\boldsymbol{d}^{1}\boldsymbol{\nu}\boldsymbol{d}^{2}}])= [L_{\boldsymbol{\nu}\boldsymbol{d}^{2}}] \otimes [L_{\boldsymbol{d}^{1}}] $ and $([L_{\boldsymbol{d}^{1}}],[L_{\boldsymbol{d}^{1}}])^{\Lambda_{1}}=1$, 
	$$(\Delta_{2}([L_{\boldsymbol{d}^{1}\boldsymbol{\nu}\boldsymbol{d}^{2}}]),\Delta_{2}([L_{\boldsymbol{d}^{1}\boldsymbol{\nu}'\boldsymbol{d}^{2}}]))^{\otimes}=([L_{\boldsymbol{\nu}\boldsymbol{d}^{2}}],[L_{\boldsymbol{\nu}'\boldsymbol{d}^{2}}] )^{\Lambda_{2}}.$$
	
	However,  the left multiplication by $L_{\boldsymbol{d^{1}}}$ is isomorphic to $(i_{\mathbf{W}^{1}})_{!} \cong(i_{\mathbf{W}^{1}})_{\ast}$, where $i_{\mathbf{W}^{1}}:\mathbf{E}_{\mathbf{V},\mathbf{W}^{2},\Omega^{1}} \rightarrow  \mathbf{E}_{\mathbf{V},\mathbf{W}^{\bullet},\Omega^{2}}$ is the zero section of  the vector bundle $\pi_{\mathbf{W}^{1}}$, and $ (i_{\mathbf{W}^{1}})^{\ast}(i_{\mathbf{W}^{1}})_{\ast}(L_{\boldsymbol{\nu}\boldsymbol{d}^{2}}) \cong L_{\boldsymbol{\nu}\boldsymbol{d}^{2}}$, we can see that $${\rm{Ext}}^{n}_{\mathcal{D}^{b}_{G_{\mathbf{V}}}(\mathbf{E}_{\mathbf{V},\mathbf{W}^{2},\Omega^{1}})/ \mathcal{N}_{\mathbf{V}} }(L_{\boldsymbol{\nu}\boldsymbol{d}^{2}},L_{\boldsymbol{\nu}'\boldsymbol{d}^{2}}) \cong {\rm{Ext}}^{n}_{\mathcal{D}^{b}_{G_{\mathbf{V}}}(\mathbf{E}_{\mathbf{V},\mathbf{W}^{\bullet},\Omega^{2}})/ \mathcal{N}_{\mathbf{V}} }(L_{\boldsymbol{d}^{1}\boldsymbol{\nu}\boldsymbol{d}^{2}},L_{\boldsymbol{d}^{1}\boldsymbol{\nu}'\boldsymbol{d}^{2}}),$$ 
	hence $$(\Delta_{2}([L_{\boldsymbol{d}^{1}\boldsymbol{\nu}\boldsymbol{d}^{2}}]),\Delta_{2}([L_{\boldsymbol{d}^{1}\boldsymbol{\nu}'\boldsymbol{d}^{2}}]))^{\otimes}=([L_{\boldsymbol{\nu}\boldsymbol{d}^{2}}],[L_{\boldsymbol{\nu}'\boldsymbol{d}^{2}}] )^{\Lambda_{2}}=([L_{\boldsymbol{d}^{1}\boldsymbol{\nu}\boldsymbol{d}^{2}}],[L_{\boldsymbol{d}^{1}\boldsymbol{\nu}'\boldsymbol{d}^{2}}])^{\Lambda^{\bullet}},$$
	and we finish the proof.
\end{proof}

\begin{remark}
	By a similar argument, we can also prove that $\Delta_{N}$ preserves the bilinear form for general $N$.
\end{remark}

\subsection{The canonical bases in certain tensor products}

\begin{definition}
	Let $\mathcal{B}_{\otimes}$ be the set consisting of $[A] \otimes [B]$, where $A$ is a simple perverse sheaf which is nonzero in $\mathcal{K}_{0}(\Lambda_{2})$ and $B$ is a simple perverse sheaf which is nonzero in $\mathcal{K}_{0}(\Lambda_{1})$.
	We define a partial order $\leqslant$ on $\mathcal{B}_{\otimes}$ as the following: If $A,A',B,B'$ are simple perverse sheaves on $\mathbf{E}_{\mathbf{V}^{2},\mathbf{W}^{2}}$,$\mathbf{E}_{\mathbf{V}^{2'},\mathbf{W}^{2}}$,$\mathbf{E}_{\mathbf{V}^{1},\mathbf{W}^{1}}$ and $\mathbf{E}_{\mathbf{V}^{1'},\mathbf{W}^{1}}$ respectively,
	we say that $[A] \otimes [B] \leqslant [A']\otimes [B']$ if  $$tr|\mathbf{V}^{2}|+tr|\mathbf{V}^{1}|=tr|\mathbf{V}^{2'}|+tr|\mathbf{V}^{1'}|$$ and if we have either $$tr|\mathbf{V}^{2}|<tr|\mathbf{V}^{2'}|,$$ or $$A=A',B=B'. $$ 
	
\end{definition}

By the construction of $\Psi$, we can see that $\Psi([A] \otimes [B] )= \sum\limits_{ [A] \otimes [B] \leqslant [A'] \otimes [B'] } \rho_{A,B,A',B'}[A'] \otimes [B'] $ such that $\rho_{A,B,A,B}=1$ and 
$$\sum\limits_{ A',B'} \rho_{A,B,A',B'}\rho_{A',B',A'',B''}=\delta_{A,A''}\delta_{B,B''}. $$

Let $\mathbf{A}=\mathbb{Q}[[v^{-1}]] \cap \mathbb{Q}(v)$ and $\mathfrak{L}$ be the subset of $\mathbb{Q}(v) \otimes_{\mathcal{A}}\mathcal{K}_{0}(\Lambda_{2}) \otimes \mathcal{K}_{0}(\Lambda_{1})$ consisting of those $x$ such that $(x,x)^{\otimes} \in \mathbf{A}$. Then following the method of \cite[Theorem 24.3.3]{MR1227098} (or \cite{bao2016canonical} and \cite{lusztig1992canonical}), for each pair of elements $([A],[B])$ of canonical basis, there is a unique element $[A] \diamond [B]$ in $\mathcal{K}_{0}(\Lambda_{2}) \otimes\mathcal{K}_{0}(\Lambda_{1})$ such that $\Psi([A] \diamond [B]) =[A] \diamond [B]$ and $[A] \diamond [B]- [A] \otimes [B]\in  v^{-1}\mathfrak{L}$. The set $\mathcal{B}_{\diamond}=\{[A] \diamond [B] |[A]\otimes [B] \in \mathcal{B}_{\otimes} \}$ is a basis of $\mathcal{K}_{0}(\Lambda_{2}) \otimes \mathcal{K}_{0}(\Lambda_{1})$  and called the canonical basis of the tensor product in \cite{lusztig1992canonical} and \cite{bao2016canonical}. The canonical basis is upper triangular with respect to the pure tensor 
$$ [A] \diamond [B] =\sum\limits_{ [A] \otimes [B] \leqslant [A'] \otimes [B'] }\pi_{A,B,A',B'}[A'] \otimes [B'], $$
such that $\pi_{A,B,A,B}=1$ and $\pi_{A,B,A',B'} \in \mathbb{Z}[v^{-1}]$.

The following proposition shows that the canonical basis is exactly consisting of those simple perverse sheaves in $\mathcal{K}_{0}(\Lambda_{1},\Lambda_{2})$.

\begin{proposition}\label{CBT}
	Let $\mathcal{B}_{\Lambda_{1},\Lambda_{2}}$ be the set consisting of $\Delta_{2}([L])$, where $L \in \mathcal{P}_{\mathbf{V},\mathbf{W}^{\bullet},\Omega^{2}}$ for some $\mathbf{V}$ but $L \notin \mathcal{N}_{\mathbf{V}}$, then  $\mathcal{B}_{\diamond}= \mathcal{B}_{\Lambda_{1},\Lambda_{2}}$.
\end{proposition}
\begin{proof}
	By definition of $(-,-)^{\otimes}$ and Proposition 5.3, $\mathcal{B}^{\otimes}$ is an almost orthogonal basis of $\mathbb{Q}(v) \otimes_{\mathcal{A}}\mathcal{K}_{0}(\Lambda_{2}) \otimes \mathcal{K}_{0}(\Lambda_{1})$. By Proposition 5.3, those $[L]$ such that $L \in \mathcal{P}_{\mathbf{V},\mathbf{W}^{\bullet},\Omega^{2}}$ for some $\mathbf{V}$ but $L \notin \mathcal{N}_{\mathbf{V}}$ form an almost orthogonal basis of $\mathbb{Q}(v) \otimes_{\mathcal{A}}\mathcal{K}_{0}(\Lambda_{1},\Lambda_{2})$, hence $\mathcal{B}_{\Lambda_{1},\Lambda_{2}}$ is an almost orthogonal basis of $\mathbb{Q}(v) \otimes_{\mathcal{A}}\mathcal{K}_{0}(\Lambda_{2}) \otimes \mathcal{K}_{0}(\Lambda_{1})$ by Proposition 5.4. $\mathcal{B}_{\Lambda_{1},\Lambda_{2}}$ is $\Psi$ invariant, since each $L$ is $\mathbf{D}$-invariant.  Then by \cite[Lemma 14.2.2(a)]{MR1227098}, both  $\mathcal{B}^{\otimes}$ and $\mathcal{B}_{\Lambda_{1},\Lambda_{2}}$ are $\mathbf{A}$-bases of $\mathfrak{L}$. Moreover, for any $L$, there exists $[A]\otimes[B]$ in $\mathcal{B}_{\otimes}$ such that $ \Delta_{2}([L]) = \pm [A]\otimes[B] \mod v^{-1}\mathfrak{L}$. By the uniqueness of $[A] \diamond [B]$, we have $[A]\diamond [B] =\pm \Delta_{2} ([L])$. However, since $\Delta_{2}(L)$ is the image of a semisimple complex in the Grothendieck group, the coefficients lie in $\mathbb{N}[v,v^{-1}]$, hence $[A]\diamond [B] =\Delta_{2} ([L])$.
\end{proof}

We also obtain the following positive property of the canonical basis as a corollary.
\begin{corollary}\label{tmp}
	For any nonzero $A$ and $B$, $$[A] \diamond [B]- [A]\otimes [B] \in \sum\limits_{ [A] \otimes [B] \leqslant [A'] \otimes [B'] }\mathbb{N}[v^{-1}]([A']\otimes [B']) .$$
\end{corollary}

If the quiver $Q$ is of finite type $ADE$, then an irreducible integrable highest weight module is also a lowest weight module. Since the $\mathbf{U}^{-}$ action is given by the induction functor, we can see that $$b^{-}([A]\diamond [B]) \in \sum\limits_{ [A'], [B'] } \mathbb{N}[v,v^{-1}]([A']\diamond [B']),$$ where $b^{-}$ is an element in the canonical basis of $\mathbf{U}^{-}$. Take $\Lambda_{1},\Lambda_{2} \rightarrow \infty$, we can see that $$b^{-}(b_{1}\diamond_{\zeta} b_{2}) \in \sum\limits_{ b'_{1}, b'_{2} } \mathbb{N}[v,v^{-1}](b'_{1}\diamond_{\zeta} b'_{2}),$$
where those $b_{1}\diamond_{\zeta} b_{2}$ are elements of the canonical basis of modified quantum group $\dot{\mathbf{U}}1_{\zeta}$. See details in \cite[25.2.1]{MR1227098} Recall that the functor $\rho_{!}$ in \cite[10.3.4]{MR1227098} realizes the anti-automorphism $\sigma$ and it sends $\mathbf{U}^{-} \subseteq\dot{\mathbf{U}}$ to $\mathbf{U}^{+,op} \subseteq\dot{\mathbf{U}}^{op}$ and induces a $\mathbf{U}^{op}$-module structure on $\mathcal{K}_{0}(\Lambda^{\bullet})$. By a similar argument,  we can see that
$$(b_{1}\diamond_{\zeta} b_{2})b^{+} \in \sum\limits_{ b'_{1}, b'_{2} } \mathbb{N}[v,v^{-1}](b'_{1}\diamond_{\zeta'} b'_{2}),$$
where $\zeta'$ is uniquely determined by $\zeta$ and the degree of $b$. If $v_{-\Lambda}$ is the lowest weight vector of an irreducible integrable lowest weight module, then  $(b'_{1}\diamond_{\zeta'} b'_{2})v_{-\Lambda}=b'^{+}_{1}v_{-\Lambda}$ for $b'_{2}=1$ and $b'^{+}_{1}v_{-\Lambda} \neq 0$. Otherwise, $(b'_{1}\diamond_{\zeta'} b'_{2})v_{-\Lambda}=0$.  It implies that the canonical basis of $\dot{\mathbf{U}}$ positively acts on irreducible integrable lowest weight modules with respect to the canonical bases of these modules.
\begin{corollary}
	If the quiver is of finite type $ADE$, then the structure coefficients of irreducible integrable lowest (and highest) modules under the action of elements of canonical basis of $\dot{\mathbf{U}}$  with respect to the canonical basis of irreducible modules are in $\mathbb{N}[v,v^{-1}]$,
	$$(b_{1}\diamond_{\zeta} b_{2})(b^{+}v_{-\Lambda}) \in \sum\limits_{ b' } \mathbb{N}[v,v^{-1}](b'^{+}v_{-\Lambda}).$$ 
\end{corollary}

\subsection{Yang-Baxter equation for $N=3$}
Now we assume $N=3$ and let $R_{ij}$ be the following isomorphism of $\mathbf{U}$-modules
\begin{equation*}
	\Delta_{2}\circ (\Delta'_{2})^{-1}:L(\Lambda_{i}) \otimes L(\Lambda_{j}) \rightarrow L(\Lambda_{j}) \otimes L(\Lambda_{i}).
\end{equation*}
where 
\begin{equation*}
	\Delta_{2}:\mathcal{K}_{0}(\Lambda_{i},\Lambda_{j}) \rightarrow \mathcal{K}_{0}(\Lambda_{j})\otimes \mathcal{K}_{0}(\Lambda_{i})
\end{equation*}
\begin{equation*}
	\Delta_{2}':\mathcal{K}_{0}(\Lambda_{i},\Lambda_{j}) \rightarrow \mathcal{K}_{0}(\Lambda_{i})\otimes \mathcal{K}_{0}(\Lambda_{j})
\end{equation*}
are isomorphisms defined by the restriction functors as Proposition \ref{tongtai} and equations in \ref{rmatrix}. Then we can give a solution of the Yang-Baxter equation via restriction functors.
\begin{proposition}\label{YB}
	Regarding $\mathcal{K}_{0}(\Lambda_{i})$ as the integrable lowest weight module $_{\mathcal{A}}L^(\Lambda_{i})$, we have the following Yang-Baxter equation of isomorphisms beween $_{\mathcal{A}}\mathbf{U}$-modules
	\begin{equation*}
		R_{23}R_{13}R_{12}=R_{12}R_{13}R_{23}:  L(\Lambda_{1}) \otimes L(\Lambda_{2}) \otimes L(\Lambda_{3}) \xrightarrow{\cong} L(\Lambda_{3}) \otimes L(\Lambda_{2}) \otimes L(\Lambda_{1}).
	\end{equation*}
\end{proposition}
\begin{proof}
	For $i,j,k \in \{ 1,2,3\}$, we simply denote $\mathcal{K}_{0}(\Lambda_{i},\Lambda_{j}) \otimes \mathcal{K}_{0}(\Lambda_{k})$ by $(ij)k$, $\mathcal{K}_{0}(\Lambda_{i})\otimes \mathcal{K}_{0}(\Lambda_{j},\Lambda_{k})$ by $i(jk)$, $\mathcal{K}_{0}(\Lambda_{i})\otimes \mathcal{K}_{0}(\Lambda_{j})\otimes \mathcal{K}_{0}(\Lambda_{k})$ by $ijk$ respectively. Let $\Delta_{ij}$ and $\Delta'_{ij}$ respectively be the isomorphisms induced by the following functors $$(sw)_{!}(\mathbf{D}\boxtimes \mathbf{D})\bigoplus\limits_{\mathbf{V}^{1},\mathbf{V}^{2}} \mathbf{Res}^{\mathbf{V}}_{\mathbf{V}^{1},\mathbf{V}^{2}}\mathbf{D}:\coprod\limits_{\mathbf{V}}\mathcal{Q}_{\mathbf{V},\mathbf{W}^{i}\oplus\mathbf{W}^{j}} \rightarrow \coprod\limits_{\mathbf{V}^{1}}\mathcal{Q}_{\mathbf{V}^{1},\mathbf{W}^{j}}\boxtimes \coprod \limits_{\mathbf{V}^{2}} \mathcal{Q}_{\mathbf{V}^{2},\mathbf{W}^{i}},$$   $$(sw)_{!}(\mathbf{D}\boxtimes \mathbf{D})\bigoplus\limits_{\mathbf{V}^{1},\mathbf{V}^{2}} \mathbf{Res'}^{\mathbf{V}}_{\mathbf{V}^{1},\mathbf{V}^{2}}\mathbf{D}:\coprod\limits_{\mathbf{V}}\mathcal{Q}_{\mathbf{V},\mathbf{W}^{i}\oplus\mathbf{W}^{j}} \rightarrow\coprod\limits_{\mathbf{V}^{1}}\mathcal{Q}_{\mathbf{V}^{1},\mathbf{W}^{i}}\boxtimes \coprod \limits_{\mathbf{V}^{2}} \mathcal{Q}_{\mathbf{V}^{2},\mathbf{W}^{j}}, $$  and let $_{i}\Delta$ and $\Delta_{i}$ respectively be the isomorphisms induced by the following functors  $$(sw)_{!}(\mathbf{D}\boxtimes \mathbf{D})\bigoplus\limits_{\mathbf{V}^{1},\mathbf{V}^{2}} \mathbf{Res}^{\mathbf{V}}_{\mathbf{V}^{1},\mathbf{V}^{2}}\mathbf{D}:\coprod\limits_{\mathbf{V}}\mathcal{Q}_{\mathbf{V},\mathbf{W}^{\bullet}} \rightarrow \coprod\limits_{\mathbf{V}^{1}}\mathcal{Q}_{\mathbf{V}^{1},\mathbf{W}^{i}}\boxtimes \coprod \limits_{\mathbf{V}^{2}} \mathcal{Q}_{\mathbf{V}^{2},\mathbf{W}^{j}\oplus\mathbf{W}^{k}} ,$$  $$(sw)_{!}(\mathbf{D}\boxtimes \mathbf{D})\bigoplus\limits_{\mathbf{V}^{1},\mathbf{V}^{2}} \mathbf{Res'}^{\mathbf{V}}_{\mathbf{V}^{1},\mathbf{V}^{2}}\mathbf{D}:\coprod\limits_{\mathbf{V}}\mathcal{Q}_{\mathbf{V},\mathbf{W}^{\bullet}} \rightarrow \coprod\limits_{\mathbf{V}^{1}}\mathcal{Q}_{\mathbf{V}^{1},\mathbf{W}^{j}\oplus\mathbf{W}^{k}}\boxtimes \coprod \limits_{\mathbf{V}^{2}} \mathcal{Q}_{\mathbf{V}^{2},\mathbf{W}^{i}} .$$
	Then we have the following diagram:
	
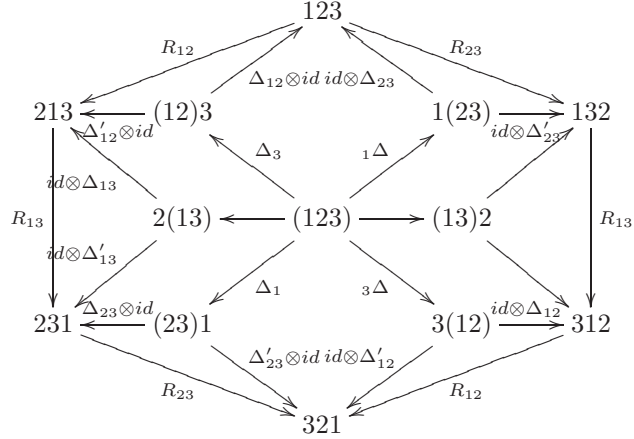
\[
	\xymatrix{
		&        & 123 \ar[dll]_{R_{12}} \ar[drr]^{R_{23}}    &         &    \\
		213 \ar[dd]_{R_{13}}	&  (12)3 \ar[l]^{\Delta_{12} \otimes id} \ar[ur]_{\Delta'_{12}\otimes id} &   &  1(23) \ar[ul]^{id \otimes \Delta'_{23}} \ar[r]_{id \otimes \Delta_{23}} &132 \ar[dd]^{R_{13}}  \\
		&  2(13) \ar[ul]^{id \otimes \Delta'_{13}}  \ar[dl]_{id \otimes \Delta_{13}} & (123)\ar[ul]_{\Delta_{3}} \ar[l]\ar[dl]^{\Delta_{1}} \ar[ur]^{_{1}\Delta} \ar[r] \ar[dr]_{_{3}\Delta}  &  (13)2 \ar[ur]\ar[dr]  &    \\
		231 \ar[drr]_{R_{23}}	&  (23)1 \ar[l]_{\Delta'_{23}\otimes id} \ar[dr]^{\Delta_{23}\otimes id}&        &  3(12) \ar[dl]_{id \otimes \Delta_{12} } \ar[r]^{id \otimes \Delta'_{12} } & 312 \ar[dll]^{R_{12}} \\
		&        & 321    &         &
	}
	\]
	By coassociativity of the restriction functor, we have the following equation of functors  $$\coprod\limits_{\mathbf{V}}\mathcal{Q}_{\mathbf{V},\mathbf{W}^{\bullet}} \rightarrow \coprod\limits_{\mathbf{V}^{1}}\mathcal{Q}_{\mathbf{V}^{1},\mathbf{W}^{1}} \boxtimes \coprod\limits_{\mathbf{V}^{2}}\mathcal{Q}_{\mathbf{V}^{2},\mathbf{W}^{2}} \boxtimes \coprod\limits_{\mathbf{V}^{3}}\mathcal{Q}_{\mathbf{V}^{3},\mathbf{W}^{3}}$$
	\begin{equation*}
		(\bigoplus\limits_{\mathbf{V}^{1},\mathbf{V}^{2}} \mathbf{Res'}^{\mathbf{V}'}_{\mathbf{V}^{1},\mathbf{V}^{2}} \otimes Id ) \circ \bigoplus\limits_{\mathbf{V}',\mathbf{V}^{3}} \mathbf{Res'}^{\mathbf{V}}_{\mathbf{V}',\mathbf{V}^{3}} =  (Id \otimes \bigoplus\limits_{\mathbf{V}^{2},\mathbf{V}^{3}} \mathbf{Res'}^{\mathbf{V}'}_{\mathbf{V}^{2},\mathbf{V}^{3}} ) \circ \bigoplus\limits_{\mathbf{V}^{1},\mathbf{V}'} \mathbf{Res}^{\mathbf{V}}_{\mathbf{V}^{1},\mathbf{V}'}.
	\end{equation*}
After applying the Verdier duality and $(sw)_{!}$ and using Proposition 5.2, one can deduce the following equation in the Grothendieck group of localizations:
	\begin{equation*}
		(\Delta'_{12} \otimes id) \Delta_{3} =(id \otimes \Delta'_{23}) {_{1}\Delta}.
	\end{equation*}
	Similarly, we can check that the six parallelograms in the above diagram commute, so
	\begin{equation*}
		\begin{split}
			R_{23}R_{13}R_{12}= & (\Delta_{23}\otimes id)  (\Delta'_{23}\otimes id)^{-1}(id \otimes \Delta_{13})	(id \otimes \Delta'_{13})^{-1} 	(\Delta_{12} \otimes id) 	(\Delta'_{12} \otimes id)^{-1}\\
			= &  (\Delta_{23}\otimes id)\Delta_{1}(\Delta_{3})^{-1}(\Delta'_{12} \otimes id)^{-1}
		\end{split}
	\end{equation*}
	The first equality holds by definition, the second equality holds by the commutativity of the parallelograms. Similarly, we have
	\begin{equation*}
		\begin{split}
			R_{12}R_{13}R_{23}=(id \otimes\Delta_{12})({_{3}\Delta})({_{1}\Delta})^{-1}(id \otimes \Delta'_{23} )^{-1}
		\end{split}
	\end{equation*}
	Then we have the following equation by the commutativity of the middle parallelograms:
	$$	R_{23}R_{13}R_{12}=R_{12}R_{13}R_{23}: L(\Lambda_{1}) \otimes L(\Lambda_{2}) \otimes L(\Lambda_{3}) \xrightarrow{\cong} L(\Lambda_{3}) \otimes L(\Lambda_{2}) \otimes L(\Lambda_{1}).$$
\end{proof}

\begin{example}
	Now we consider the quiver $A_{m}$, which is associated to $\mathbf{U}_{q}(\mathfrak{sl}_{m+1})$. Take $\Lambda_{1}=\Lambda_{2}=\cdots=\Lambda_{N}=\Lambda$ such that $\langle \Lambda,\alpha^{\vee}_{i}\rangle= \delta_{1,i}, 1\leqslant i \leqslant m$.
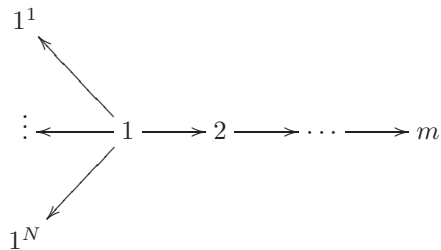
	\[
	\xymatrix{
		1^{1} &           &		&                  & \\
		\vdots	 &	1 \ar[r] \ar[lu] \ar[ld] \ar[l]   & 2    \ar[r] & \cdots  \ar[r]& m \\
		1^{N} &            &         &          &	 
	}
	\]
	
	Let $v_{0}=[L_{\boldsymbol{d}}]$ be the highest weight vector of $\mathcal{K}_{0}(\Lambda)$ and $v_{k}$ be the image of the complex $L_{\boldsymbol{\nu}^{k}\boldsymbol{d}}$ of flag type $\boldsymbol{\nu}^{k}=(k,k-1,\cdots,1)$. By Theorem \ref{thm1}, we have $F_{k}v_{k-1}=v_{k}$ and $E_{k}v_{k}=v_{k-1}, 1\leqslant k \leqslant m$. In particular, $\mathcal{K}_{0}(\Lambda)$ is isomorphic to the natural representation of $\mathbf{U}_{q}(\mathfrak{sl}_{m+1})$. For each $1\leqslant i \leqslant N-1$, we consider the following diagram
	$$(\mathcal{K}_{0}(\Lambda))^{\otimes N}    \xleftarrow{\Delta}  (\mathcal{K}_{0}(\Lambda))^{\otimes i-1}\otimes \mathcal{K}_{0}(\Lambda,\Lambda) \otimes(\mathcal{K}_{0}(\Lambda))^{\otimes N-i-1} \xrightarrow{\Delta'}   (\mathcal{K}_{0}(\Lambda))^{\otimes N}$$
	and let $s_{i}$ be the composition of $\Delta' \circ (\Delta)^{-1}$. Then Proposition \ref{YB} shows that those $s_{i}$ satisfy the Braid relation. In particular, if we identify $(\mathcal{K}_{0}(\Lambda))^{\otimes N}  $ with $\mathcal{K}_{0}(\Lambda^{\bullet})$, then after extending the action of $\mathbf{U}_{q}(\mathfrak{sl}_{m+1})$ to that of $\mathbf{U}_{q}(\mathfrak{gl}_{m+1})$, $\mathcal{K}_{0}(\Lambda^{\bullet})$ should provide a geometric model to realize the Schur-Weyl duality for type $A$.
\end{example}

\end{spacing}


\end{document}